%% file: KaTu2_2024-07-02.tex
\documentclass[10pt]{article}

\usepackage{amsmath, amsthm, amssymb, enumerate, graphicx}
\usepackage{color}
\usepackage{datetime}
\usepackage{fancyhdr}
\usepackage{comment}
\usepackage{appendix}

\chardef\coloryes=1 %%%out
\chardef\isitdraft=0%%%out
\ifnum\isitdraft=1 %%%out
\usepackage[color,notcite]{showkeys}%%%out
   \def\eqref#1{({\ref{#1}})}                %saves writing paranthesis%%%out

\definecolor{refkey}{rgb}{.3,0.3,0.3}%%%out

\else%%%out
  \definecolor{refkey}{rgb}{.8,0.8,0.1}%%%out
  
\definecolor{labelkey}{rgb}{.9,0.6,0.1}%%%out

\fi%%%out

\begin{document}
                                         %when ref. to equations
\def\intint{\int\!\!\!\!\int}
\def\intinttext{\int\!\!\!\int}
\def\intintint{\int\!\!\!\!\int\!\!\!\!\int}
\def\intintintint{\int\!\!\!\!\int\!\!\!\!\int\!\!\!\!\int}
\def\ques{{\cor \underline{??????}\cob}}
\def\nto#1{{\coC \footnote{\em \coC #1}}}
\def\fractext#1#2{{#1}/{#2}}
\def\fracsm#1#2{{\textstyle{\frac{#1}{#2}}}}   %smaller version of frac
\newcommand\Gampl{\Gamma_{\textup{pl}}}
\newcommand\wtil{\tilde{w}}
\newcommand\Wtil{\tilde{W}}
\newcommand\vtil{\tilde{v}}
\newcommand\ptil{\tilde{p}}
\newcommand\phitil{\tilde{\phi}}
\newcommand\psitil{\tilde{\psi}}
\newcommand\qtil{\tilde{q}}
\newcommand\pbar{\overline{p}}
\newcommand\phibar{\overline{\phi}}
\newcommand\qbar{\overline{q}}
\newcommand\alphp{1-2k(\pbar+\ptil)}
\newcommand\regpar{\theta}
\def\R{\mathbb R}
\newcommand {\Dn}[1]{\frac{\partial #1  }{\partial N}}
\newcommand {\uld}[1]{\underline{d #1 }}
\def\gfcn{g}
\def\hfcn{h}
\def\lfcn{\ell}
\def\betaabs{{\beta_a}}
\def\gammaabs{{\gamma_a}}
\def\absbc{a}
\def\wttil{\tilde{\tilde{w}}}
%if \isitdraft=0 or \coloryes=0%%%out
\def\cor{{}}%%%out
\def\cog{{}}%%%out
\def\cob{{}}%%%out
\def\coe{{}}%%%out
\def\coA{{}}%%%out
\def\coB{{}}%%%out
\def\coC{{}}%%%out
\def\coD{{}}%%%out
\def\coE{{}}%%%out
\def\coF{{}}%%%out
%\def\MR#1{}%%%out

%%3
%\ifnum\isitdraft=1%%%out
\ifnum\coloryes=1%%%out

  \definecolor{coloraaaa}{rgb}{0.1,0.2,0.8}%%%out
  \definecolor{colorbbbb}{rgb}{0.1,0.7,0.1}%%%out
  \definecolor{colorcccc}{rgb}{0.8,0.3,0.9}%%%out
  \definecolor{colordddd}{rgb}{0.0,.5,0.0}%%%out
  \definecolor{coloreeee}{rgb}{0.8,0.3,0.9}%%%out
  \definecolor{colorffff}{rgb}{0.8,0.3,0.9}%%%out
  \definecolor{colorgggg}{rgb}{0.5,0.0,0.4}%%%out
  \definecolor{colorhhhh}{rgb}{0.6,0.6,0.6}%%%out

 \def\cog{\color{colordddd}}%%%out
 \def\coy{\color{colorhhhh}}%%%out
 \def\cogray{\color{colorhhhh}}%%%out
 \def\cob{\color{black}}%%%out

 \def\coe{\color{blue}}%%%out
 \def\cor{\color{red}}%%%out

 \def\coA{\color{coloraaaa}}%%%out
 \def\coB{\color{colorbbbb}}%%%out
 \def\coC{\color{colorcccc}}%%%out
 \def\coD{\color{colordddd}}%%%out
 \def\coF{\color{colorffff}}%%%out
 \def\coG{\color{colorgggg}}%%%out

%%4
\fi%%%out
%\fi%%%out
\ifnum\isitdraft=1%%%out
   \chardef\coloryes=1 %%%out
   \baselineskip=17pt%%%out
   \input macros.tex%%%out
   \def\blackdot{{\color{red}{\hskip-.0truecm\rule[-1mm]{4mm}{4mm}\hskip.2truecm}}\hskip-.3truecm}%%%out
   \def\bdot{{\coC {\hskip-.0truecm\rule[-1mm]{4mm}{4mm}\hskip.2truecm}}\hskip-.3truecm}%%%out
   \def\purpledot{{\coA{\rule[0mm]{4mm}{4mm}}\cob}}%%%out
   \def\pdot{\purpledot}%%%out
  \definecolor{labelkey}{rgb}{.5,0.1,0.1}%%%out
\else%%%out  
   \baselineskip=15pt
   \def\blackdot{{\rule[-3mm]{8mm}{8mm}}}%%%out
   \def\purpledot{{\rule[-3mm]{8mm}{8mm}}}%%%out
   \def\pdot{}
%%5
\fi%%%out

\def\tdot{\fbox{\fbox{\bf\tiny\coe I'm here; \today \ \currenttime}}}
\def\nts#1{{\hbox{\bf ~#1~}}} %nts=note to self
\def\nts#1{{\cor\hbox{\bf ~#1~}}} %nts=note to self%%%out
\def\ntsf#1{\footnote{\hbox{\bf ~#1~}}} %nts=note to self
\def\ntsf#1{\footnote{\cor\hbox{\bf ~#1~}}} %nts=note to self%%%out
\def\bigline#1{~\\\hskip2truecm~~~~{#1}{#1}{#1}{#1}{#1}{#1}{#1}{#1}{#1}{#1}{#1}{#1}{#1}{#1}{#1}{#1}{#1}{#1}{#1}{#1}{#1}\\}%%%out
\def\biglineb{\bigline{$\downarrow\,$ $\downarrow\,$}}%%%out
\def\biglinem{\bigline{---}}%%%out
\def\biglinee{\bigline{$\uparrow\,$ $\uparrow\,$}}%%%out

\newcommand{\Dnu}{ \partial_{\nu}}
\renewcommand{\Gampl}{\Gamma_{pl}}
\newcommand{\Gamref}{\Gamma_{0N}}
\newcommand{\Omref}{\Omega_{0}}
\newcommand{\Gamflat}{B}
\newcommand{\lref}{\ell_0}
\newcommand{\regop}{A}
\def\mbar{{\overline M}}
\def\tilde{\widetilde}
\definecolor{darkgreen}{rgb}{0.,0.6,0.4}
\definecolor{darkcyan}{rgb}{0.,0.7,0.7}
\definecolor{darkpurple}{rgb}{0.4,0,0.4}
\definecolor{darkred}{rgb}{0.6,0,0}
\definecolor{grey}{rgb}{0.5,0.5,0.5}
\def\ntAT#1{{\textcolor{blue}{\bf ~#1~}}} 
\def\ntBK#1{{\textcolor{darkgreen}{\bf ~#1~}}} 
\def\ntregl#1{{\textcolor{darkred}{\bf [on regularity of $\lfcn$:] ~#1~}}} 
\def\colAT#1{\textcolor{blue}{ ~#1~}} 
\def\colBK#1{\textcolor{darkgreen}{ ~#1~}} 
\def\chgd#1{\textcolor{darkpurple}{ ~#1~}} 
\def\rmv#1{\textcolor{grey}{ ~#1~}} 
\def\Delpl{\Delta_{pl}}
\def\Ep{\mathcal{E}^p}
\def\Ew{\mathcal{E}^w}
\def\zeroT{0,T;}
\def\HminustwoGampl{H^2(\Gampl)^*}
\newtheorem{Theorem}{Theorem}[section]
\newtheorem{Corollary}[Theorem]{Corollary}
\newtheorem{Proposition}[Theorem]{Proposition}
\newtheorem{Lemma}[Theorem]{Lemma}
\newtheorem{Remark}[Theorem]{Remark}
\newtheorem{definition}{Definition}[section]
\def\theequation{\thesection.\arabic{equation}}
\def\endproof{\hfill$\Box$\\}
\def\square{\hfill$\Box$\\}
\def\comma{ {\rm ,\qquad{}} }            %comma in a formula
\def\commaone{ {\rm ,\qquad{}} }         %second comma in a formula
\def\dist{\mathop{\rm dist}\nolimits}    %distance
\def\sgn{\mathop{\rm sgn\,}\nolimits}    %sgn
\def\Tr{\mathop{\rm Tr}\nolimits}    %trace
\def\div{\mathop{\rm div}\nolimits}    %divergence
\def\TT{R}
\def\supp{\mathop{\rm supp}\nolimits}    %divergence
\def\divtwo{\mathop{{\rm div}_2\,}\nolimits}    %two dimensional divergence
\def\curl{\mathop{\rm curl}\nolimits}    %curl
\def\dbar{\overline\partial}
\def\l{\langle}
\def\r{\rangle}
\def\plusdelta{+\delta}
\def\pd{+\delta}
\def\re{\mathop{\rm {\mathbb R}e}\nolimits}    %distance
\def\indeq{\qquad{}\!\!\!\!}                     %indentation in formulas
\def\period{.}                           %period in a formula
\def\semicolon{\,;}                      %semicolon in a formula
\newcommand{\cD}{\mathcal{D}}
\def\Damppl#1{}
\def\omegazero{\omega^0}
\def\omegaone{\omega^1}
%
%**end of header

\title{\mbox{Optimization of a Nonlinear Acoustics -- Structure Interaction Model}}

\author{ 
Barbara Kaltenbacher,
Amjad Tuffaha} 
\maketitle
\date{}
\bigskip

\bigskip
\indent Department of Mathematics\\
\indent University of Klagenfurt\\
\indent Klagenfurt, Austria\\
\indent e-mail: barbara.kaltenbacher@aau.at

\bigskip
\indent Department of Mathematics\\
\indent American University of Sharjah\\
\indent Sharjah, UAE\\
\indent e-mail: atufaha\char'100aus.edu

\begin{abstract}
In this paper, we consider a control/shape optimization problem of a nonlinear acoustics-structure interaction model of PDEs, whereby acoustic wave propagation in a chamber is governed  by the Westervelt equation, and the motion of the elastic part of the boundary is governed by a 4th order Kirchoff equation. We consider a quadratic objective functional capturing the tracking of prescribed desired states, with three types of controls:   1) An excitation control represented by prescribed Neumann data for the pressure on the excitation part of the boundary
2) A mechanical control represented by a forcing function in the Kirchoff equations and 3) Shape of the excitation part of the boundary represented by a graph function.
Our main result is the existence of solutions to the minimization problem, and 
the characterization of the optimal states  through an adjoint system of PDEs derived from the first-order optimality conditions. 
\end{abstract}

\section{Introduction}

\begin{figure}
\begin{center}
\begin{minipage}{0.4\textwidth}
\includegraphics[width=0.7\textwidth]{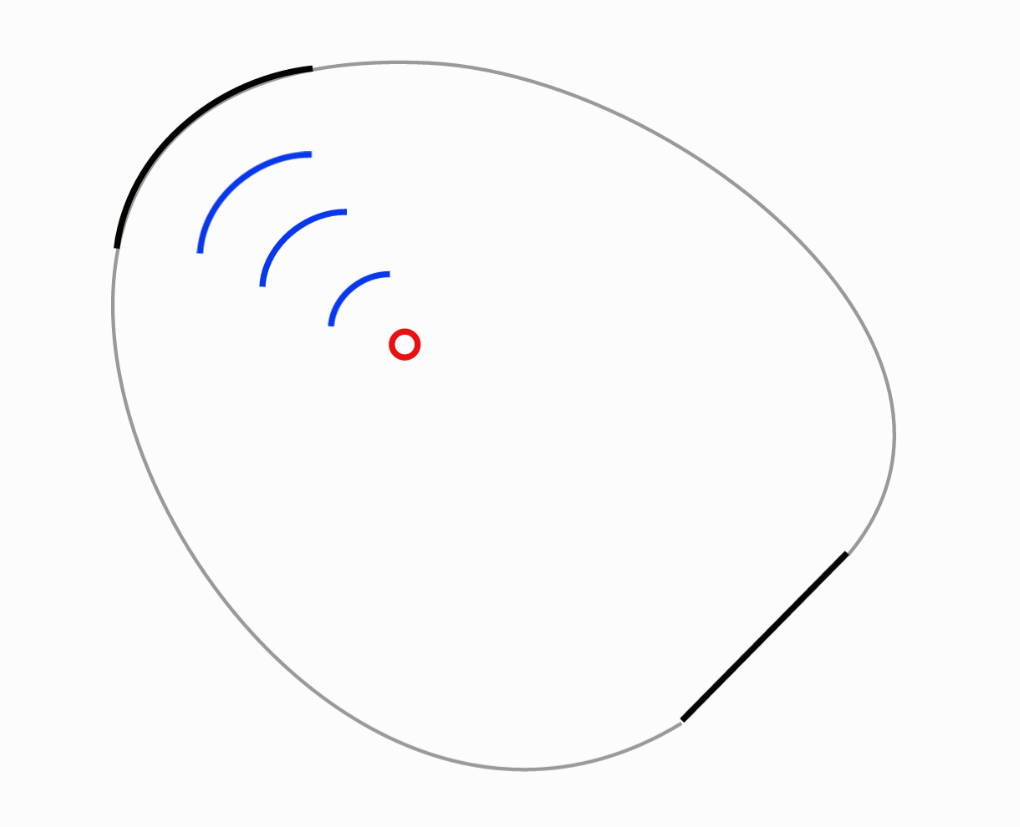}
\\[-26ex]
$\textcolor{darkgreen}{\Gamma_N}$: \,  excitation $\Dnu{p}=\textcolor{darkgreen}{g}$
\\[3ex]
\hspace*{3.9cm}$\Gamma_a$
\\[11ex]
\hspace*{0.5cm}$\Gamma_a$
\\[-3ex]
\hspace*{3.5cm}$\Gampl$: forcing \textcolor{darkgreen}{$h$}
\\[5ex] 
\end{minipage}
\hspace*{0.05\textwidth}
\begin{minipage}{0.4\textwidth}
\includegraphics[width=0.7\textwidth]{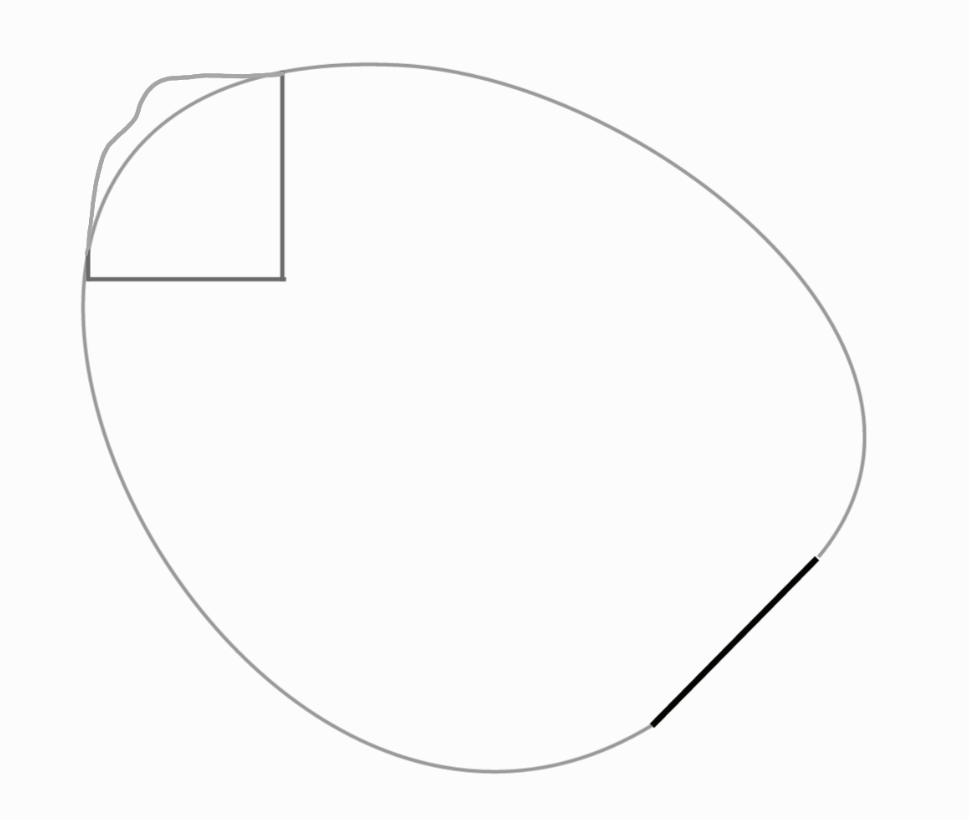}
\\[-25ex]
{\small $\Gamma_N$}
\\[0ex]
\hspace*{0.7cm}{\small $\Gamref$}
\\[0ex]
\hspace*{3.9cm}{\small $\Gamma_a$}
\\[-0.5ex]
\hspace*{0.7cm}{\small $\Gamflat$}
\\[8ex]
\hspace*{0.5cm}{\small $\Gamma_a$}
\\[-3ex]
\hspace*{3.5cm}{\small $\Gampl$}
\\[5ex] 
\end{minipage}
\end{center}
\caption{left: an exemplary experimental setup\label{fig:domain};
right: schematic of boundary parametrization}
\end{figure}

In this paper, we study a shape optimization problem involving an acoustics-structure interaction model. The model describes the propagation of nonlinear acoustic waves in a cavity subject to mechanical excitation through an elastic wall in addition to acoustic excitation through a rigid impermeable wall of 
%prescribed 
definable 
shape. We consider the shape optimization control problem of minimizing some objective 
%(e.g., the acoustic-mechanical energy in the cavity) 
(e.g., tracking of a prescribed desired pressure and elastic wall displacement)
with respect to the mechanical and acoustic excitation functions in addition to the shape of the rigid excitation wall.
The model under study consists of the nonlinear Westervelt equation in the acoustic pressure variable defined on a three dimensional bounded domain, coupled with a 4th order Kirchoff equation in the transversal displacement variable describing the motion of the elastic part of the boundary. The two equations are connected through two mechanisms:  a transmission boundary condition matching the normal acoustic pressure gradient to the structural inertial term,  and a dynamic loading of acoustic pressure onto the plate equation. Neumann boundary conditions for the acoustic pressure are prescribed on part of the boundary in terms of a given boundary function $g$, and the remaining part of the boundary is endowed with absorbing boundary conditions, which signify a wave absorbing wall.

The well-posedness of the PDE model under consideration was studied by the authors in \cite{KaTu1}, where weak energy level solutions were obtained for the linearized model, and both local-in-time and global-in-time (when the plate equation is subjected to damping) strong solutions of the full nonlinear system were obtained for small data. In the present study, we consider a shape optimization/control problem where the objective is to minimize 
%a quadratic energy functional in Sobolev norm of the state variables capturing both acoustic and mechanical energy of the system as well as the shape of the mechanical excitation wall. 
a tracking type objective function with some regularization in terms of Sobolev norms of the design functions.
There are three control variables in the problem under consideration: 1) The acoustic excitation function $\gfcn$ prescribed as Neumann data on the  excitation part of the boundary, 2) A mechanical forcing  function $\hfcn$ acting as an additional loading on the elastic wall 3) The graph function $\lfcn$ describing the shape of the excitation part of the boundary. 
We study the problem over a finite time horizon and
our main result is the existence of 
%unique 
optimal control functions $\gfcn,\,\hfcn,\,\lfcn$ in a certain admissible class, and the existence of adjoint state variables satisfying an adjoint system of
PDEs derived from first order optimality conditions. We utilize a new variational formulation of the system to study the control problem at hand. This is the first mathematically oriented control treatment of a nonlinear structure-acoustic interaction model involving the Westervelt equation and includes both shape optimization and excitation control action. 

The mathematical study of structure-acoustics control problems is motivated by many applications in  engineering systems in which noise and vibration control are desirable objectives, cf., e.g.,  \cite{Bharat2020,Gerges2013,Maestrello1992}. 
There have been numerous works on mathematical treatment of  active control of  structure-acoustics. Thehe problem of active control of a PDE system that describes an acoustic chamber with an elastic wall appears in \cite{BanksSmith1995}. The system, based on a model used by NASA, consists of a linear wave equation coupled with an Euler-Bernoulli beam equation  forced by the acoustic pressure, and governing  the transversal motion of the elastic interface. Well-posedness and control theoretic results were obtained for the systems under piezoelectric control action \cite{AvalosLasiecka1996, lasiecka2002, LasieckaMarshand1997, LasieckaMarchand1998}. Studies of stability and long-term behavior of the system were conducted in, e.g., \cite{ Fahroo1999, AvalosGeredeli2016, AvalosGeredeli2018}. A boundary control model of a system comprising a wave equation with variable coefficients was also considered in \cite{Liu2022}. 
For control treatment of nonlinear/semilinear models of structure acoustics, we refer the reader to \cite{Liu2024}. On the other hand, boundary control analysis of the Westervelt equation has been already considered in
e.g. \cite{ClasonBK:2009}, while shape optimization problems were studied in \cite{BK-Peichl,KaltenbacherNikolic2017,MNWW2019, MN2022}.
For analysis and long-term behavior of solutions for the Westervelt equation with both Dirichlet- and Neumann-type boundary conditions, we refer the reader to \cite{KL09Westervelt,KLV11_inhomDir,KL12_Neumann,MeyerWilke13,SimonettWilke2017}. For  control of the Jordan-Moore-Gibson equation, an advanced model of nonlinear acoustics, see \cite{LZ2019} and \cite{BucciLasiecka2019, BongartiCharoenphonLasiecka20} for linearized models.

\section{Model and Problem Formulation}
We consider a coupled structure-acoustic system consisting of a Westervelt equation defined on a domain 
$\Omega \subset \mathbb{R}^{d}$, $d\in\{2,3\}$  and coupled with a linear Kirchoff plate equation defined on part of the boundary of the domain. 
Two control functions are used to control the dynamics of the system, one control is acting on the plate and the other takes the form of a Neumann boundary condition on another segment of the boundary.
Additionally, we consider design of the Neumann boundary shape.

The Westervelt equation in the variable $p$ reads
\begin{align}
(1- 2k p) p_{tt} - c^{2} \Delta p - b \Delta p_{t} = 2k(p_{t})^{2} ~~~ \Omega \times [0,T],
\end{align}
or equivalently
\begin{align}
((1- 2k p) p_t)_t - c^{2} \Delta p - b \Delta p_{t} = 0 ~~~ \Omega \times [0,T],
\end{align}
where $c$, $ k, b >0$ are given constant parameters. 

The boundary of the domain $\Omega$ is the union of three disjoint parts $\partial \Omega = \Gamma_{a} \cup \Gamma_{N} \cup \Gamma_{pl}$, representing absorbing, Neumann, and plate conditions.
The Westervelt equation is coupled with a 4th order plate equation in the $\tilde{w}$ variable defined on the interface $\Gamma_{pl}$
\begin{align}
\rho \tilde{w}_{tt} + \delta \Delta_{pl}^{2} \tilde{w} 
\Damppl{+ \beta (- \Delta_{pl})^{\gamma} \tilde{w}_t}
& = \kappa p_{t} +\hfcn
~~ \mbox{on} ~~~ \Gamma_{pl} ,
\end{align}
where $\tilde{w}=w_{t}$, $\Delta_{pl}$ is the Laplace Beltrami operator on $\Gamma_{pl}$, and $w$ is the mid surface displacement variable 
which we supplement with hinged boundary conditions
\begin{align}
\Delta_{pl} \tilde{w} = \tilde{w} = 0 ~~~ \mbox{on} ~~ \partial \Gamma_{pl} .
\end{align}

We also impose the following boundary conditions on each component
\begin{align}
%p_{t} + c\Dnu{p} 
\Dnu{p}+\absbc[p,p_t]:=\Dnu{p}+\betaabs p_{t}+\gammaabs p 
&=0~~~~& \mbox{on} ~~~~ \Gamma_{a} \\
\Dnu{p}  &=\gfcn ~~~~ &\mbox{on} ~~~~ \Gamma_{N} \\
 \Dnu{p}& = - \rho w_{tt} = -\rho\wtil_{t}
 ~~~~ &\mbox{on} ~~~~ \Gamma_{pl} ,
\end{align}
where $\nu$ denotes the unit outward normal to the boundary, while $\gfcn$ and $\hfcn$ are control functions. 

The functions $\betaabs\geq0$, $\gammaabs\geq0$ are known smooth boundary functions. In particular, they are also allowed to vanish, therewith enabling Neumann or impedance (Robin) boundary conditions on the fixed boundary part, while the shape of $\Gamma_N$ will be subject to optimization.

The absorbing boundary conditions on $\Gamma_a$ describe a non-reflecting or, in case of vanishing $\betaabs$, $\gammaabs$, sound-hard interface, while on the variable boundary part $\Gamma_N$, excitation by a transducer array,  is modeled by the function $\gfcn$. Additionally to that, also excitation via the forcing $\hfcn$ on the plate part of the boundary is allowed. 

\subsection{Variational Formulation}

The total pressure $p=\pbar+\ptil$ is decomposed into an extension $\pbar$ of the Neumann boundary data 
\footnote{If we only use $\hfcn$ as a control and set $\gfcn=0$, then the Neumann data extension $\pbar$ is not needed} 
and the remainder $\ptil$.
%In view of the different orders of time derivatives with which $\pbar$, $\ptil$ appears in the variational formulation 
%\footnote{note the $\pbar_{tt}$ term remaining in the plate equation terms of \eqref{eqn:var_decoup}, while the $\ptil_{tt}$ goes away due to our special way of formulating \eqref{ptil_decoup} variationally; treating $\pbar$ the same way would disable extension of the inhomogeneous Neumann data(This refer to a previous version)} 
%differently from \cite{KaTu1}, 
We decouple $\pbar$, defining it by 
\begin{equation}\label{pbar_decoup}
\begin{aligned}
\pbar_{tt}  - c^{2}\Delta \pbar - b \Delta \pbar_{t}  &= 0 ~~\mbox{in}~~ \Omega \times [0,T] \\
%c\Dnu{\pbar} +\pbar_t
\Dnu{\pbar}+\absbc[\pbar,\pbar_t]&= 0 ~~\mbox{on}~~ \Gamma_{a} \times [0,T]  \\
\Dnu{ \pbar}&=\gfcn ~~\mbox{on}~~ \Gamma_{N} \times [0,T] \\
\Dnu{ \pbar} &= 0  ~~\mbox{on}~~ \Gampl \times [0,T] \\
\pbar(0)&=0, ~~\pbar_{t}(0)=0.
\end{aligned}
\end{equation}
and correspondingly define $(\ptil,\wtil)$ by
\begin{equation}\label{ptil_decoup}
\begin{aligned}
%((1-2k(\pbar+\ptil))\, \ptil_t)_t  - c^{2}\Delta \ptil - b \Delta \ptil_{t}  
%-2k(\ptil \pbar_t)_t
%&= 2k(\pbar \pbar_t)_t ~~\mbox{in}~~ \Omega \times [0,T] \\
(1-2k(\pbar+\ptil))\, \ptil_{tt}  - c^{2}\Delta \ptil - b \Delta \ptil_{t}  
&=2k((\pbar+\ptil)_t)^2 +2k(\pbar+\ptil)\pbar_{tt} ~~\mbox{in}~~ \Omega \times [0,T] \\
 %c\Dnu{\ptil}+\ptil_{t} 
 \Dnu{\ptil}+\absbc[\ptil,\ptil_t]
 &= 0  ~~\mbox{on}~~ \Gamma_{a} \times [0,T]  \\
\Dnu{\ptil}&=0 ~~\mbox{on}~~ \Gamma_{N} \times [0,T] \\
\Dnu{\ptil} &= - \rho \wtil_t  ~~\mbox{on}~~ \Gampl \times [0,T] \\
\rho\wtil_{tt} + \delta\Delpl^{2} \wtil 
\Damppl{+  \beta(-\Delpl)^{\gamma} \wtil_{t}} 
&= \kappa (\ptil_{t} + \pbar_{t}) +\hfcn  ~~\mbox{on}~~ \Gampl \times [0,T] \\
\tilde{p}(0)&= p_{0}, ~~\tilde{p}_{t}(0)= p_{1}\\
\tilde{w}(0)&= \tilde{w}_{0}, ~~\tilde{w}_{t}(0)= \tilde{w}_{1}.
\end{aligned}
\end{equation}

Note that as compared to \cite{KaTu1}, we assume that $\Gamma_D=\emptyset$ (that is, no Dirichlet boundary conditions imposed on any part of the boundary), which allows to avoid a certain regularity loss.
%and we also assume $\gammaabs=0$; the former assumption , while the latter is anyway the canonical setting in absorbing boundary conditions. 

Differently from  \cite{KaTu1}, we consider the system \eqref{pbar_decoup}, \eqref{ptil_decoup} in an $L^2(\zeroT L^2(\Omega))^2\times L^2(\zeroT H^{2}_{\diamondsuit}(\Gampl)^*)$ sense as follows 
\begin{equation}\label{eqn:var_L2L2}
\begin{aligned}
&\int_0^T\Bigl\{
\int_\Omega \Bigl(
\bigl(\pbar_{tt}  - c^{2}\Delta \pbar - b \Delta \pbar_{t}\bigr)\,\qbar\\
&\qquad\qquad+\bigl((1-2k(\pbar+\ptil))\, \ptil_{tt}  - c^{2}\Delta \ptil - b \Delta \ptil_{t}  
-2k((\pbar+\ptil)_t)^2 -2k(\pbar+\ptil)\pbar_{tt}\bigr)\,\qtil\Bigr)\, dx\\
&\qquad +\tfrac{\rho}{\kappa}\int_{\Gampl}  \Bigl(
\bigl(\rho\wtil_{tt}
\Damppl{+ \beta(-\Delpl)^{\gamma} \wtil_{t} }
- \hfcn-\kappa(\pbar_{t}+\ptil_{t})\bigr)\vtil
+\delta\Delpl\wtil\, \Delpl\vtil \Bigr) \,dS
\\
&\qquad+ 
\int_{\Gamma_N}
%[(\bar{\lambda}+b)\partial_t+c^2]
(\Dnu{\pbar}-\gfcn)\,\mu_{N}\, dS
+\int_{\Gampl}
%[(\tilde{\lambda}(1-2k(\pbar+\ptil))+b)\partial_t+c^2]
(\Dnu{\ptil} + \rho \wtil_t)\,\mu_{pl}\,dS
\Bigr\}\,dt\\
&=0 \qquad\text{ for all }(\qbar,\qtil,\vtil,\mu_N,\mu_{pl})\in Z
\end{aligned}
\end{equation}
with initial conditions $(\bar{p},\tilde{p},\tilde{w}(0))=(0,p_0,\tilde{w}_0)$, $(\bar{p}_t,\tilde{p}_t,\tilde{w}_t(0))=(0,p_1,\tilde{w}_1)$  
%and $\bar{\lambda}$, $\tilde{\lambda}\,>0$ arbitrary (but they will be  suitably chosen cf. \eqref{enid_pp}, \eqref{enest_pp} below) 
on the function spaces
\begin{equation}\label{fcnsp}
\begin{aligned}
%& H^2_{\Delta,0}(\Omega) = \{p\in L^2(\Omega) \,:\, p, \ \nabla p, \ \Delta p\in L^2(\Omega), \ 
%p\vert_{\Gamma_a}\in L^2(\Gamma_a), \ 
%p\vert_{\Gampl} \in H^{2}_{\diamondsuit}(\Gampl)^*\}\\
& H^2_{\Delta,1}(\Omega) = \{p\in L^2(\Omega) \,:\, p, \ \Delta p\in L^2(\Omega), \ \Dnu p\vert_{\partial\Omega}
\in H^s(\partial\Omega))\}\\
& H^{2}_{\diamondsuit}(\Gampl) = \{w \in H^{2}(\Gampl) \,:\, w = \Delta_{pl} w = 0 \text{ on } \partial \Gampl \}\\
&U=U(\Omega)\\
&=\{p\in H^2(\zeroT L^2(\Omega))\cap 
%H^1(\zeroT H^2_{\Delta,0}(\Omega))\cap L^2(\zeroT H^2_{\Delta,1}(\Omega))
H^1(\zeroT H^2_{\Delta,1}(\Omega))
\,:\, 
\Dnu{p}+\absbc[p,p_t]
=0\text{ on } \Gamma_{a}, \ \Dnu{p}=0\text{ on } \Gampl\}\\
&\qquad\times\{p\in H^2(\zeroT L^2(\Omega))\cap 
%H^1(\zeroT H^2_{\Delta,0}(\Omega))\cap L^2(\zeroT H^2_{\Delta,1}(\Omega))
H^1(\zeroT H^2_{\Delta,1}(\Omega))
\,:\, 
\Dnu{p}+\absbc[p,p_t]
=0\text{ on } \Gamma_{a}, \ \Dnu{p}=0\text{ on } \Gamma_{N}\} \\
&\qquad\times\bigl(H^2(\zeroT H^{2}_{\diamondsuit}(\Gampl)^*)
\Damppl{\cap H^1(\zeroT H^{\gamma}_{\diamondsuit}(\Gampl))}
\cap L^2(\zeroT H^{2}_{\diamondsuit}(\Gampl)) \bigr)
\\
&Z=Z(\Omega)\\
&=\bigl(L^2(\zeroT L^2(\Omega))\bigr)^2\times L^2(\zeroT H^{2}_{\diamondsuit}(\Gampl))
\times L^2(\zeroT H^{-s}(\Gamma_N)) \times L^2(\zeroT 
%H^{-\min\{s,\gamma\}}(\Gampl)  % in case Damppl means skipping the plate damping term
%H^{-s}(\Gampl) % 
L^2(\Gampl)
)
\end{aligned}
\end{equation}
%\footnote{\colBK{note that the above regularity of $\wtil$ in 
%$H^2(\zeroT H^{2}_{\diamondsuit}(\Gampl)^*)
%\cap L^2(\zeroT H^{2}_{\diamondsuit}(\Gampl)) \bigr)$
%by interpolation only provides us with 
%$\wtil_t\in L^2(\zeroT L^{2}(\Gampl))$}}
for some fixed $s\in(0,1/2]$, cf. \eqref{enid_pp}, \eqref{enest_pp}.
On a $C^{1,1}$ domain $\Omega\subseteq\mathbb{R}^d$, $d\in\{1,2,3\}$ we have $H^2_{\Delta,1}(\Omega)\subseteq H^{s+3/2}(\Omega)\subseteq L^\infty(\Omega)\cap W^{1,3/(1-s)}(\Omega)$, cf. \cite[Lemma 4.2]{KaTu1}. 
Note that $H^2_{\Delta,1}(\Omega)\subseteq L^\infty(\Omega)$ can also be obtained for an only Lipschitz domain $\Omega$ by means of Stampacchia's / De Giorgi's technique, cf., e.g., 
\cite[Proposition 4.1]{Consiglieri:2014},
\cite[Theorem 4.5 and Section 7.2.1]{Troeltzsch2010} and we are also going to make use of this here to keep regularity requirements on the boundary (in particular also on its variable part) low.

\subsection{A general control/shape optimization problem}\label{sec:opt}

To include the shape $\Gamma_N$ in the optimization (as it is formally done here) we need to either use some kind of shape calculus (as, e.g. in \cite{BK-Peichl}; we also refer to the standard shape optinization references therein) or some parameterization of $\Gamma_N$ as the graph of a function.
We choose the latter option for the following reasons:
(a) When optimizing $\Gamma_N$ for fixed $\gfcn$ it allows to easily define the Neumann condition on the deformed boundary by composition with the parameterization;\footnote{In \cite{BK-Peichl} (see line 8 of Sec.4) this is circumvented by using a function $\bar{g}$ that is defined on all of $\Omega$ and defining $\gfcn:=\bar{g}\vert_{\Gamma_N}$ as its restriction to $\Gamma_N$}; 
(b) The typical deformations that we have in mind for applications (e.g., ultrasound in a reverberant cavity \cite{BrownZhangTreebyBeardCox2019,kunyansky2013photoacoustic}) are covered by this setting;
(c) It allows to handle joint optimization of the controls $\gfcn$, $\hfcn$ and the shape in a common framework of optimal control, as e.g. treated in detail in \cite{Troeltzsch2010};
(d) It requires a limited amount of technicalities.

Concretely, we parameterize $\Gamma_N$ by using a  flat domain $\Gamflat$ of the parameterization and without loss of generality define the $z$ direction to be the direction of variability so that $\Gamma_N(\lfcn)=\{(x',\lfcn(x'))\, : \, x'\in \Gamflat\}$,  $\Gamref=\Gamma_N(\lref)$ for some open simply connected set $\Gamflat\subseteq\mathbb{R}^{d-1}$, so that $\overline{\Omega}=\overline{\Omega(\lfcn)}=\overline{\Omega_{\text{fix}}}\cup\overline{\Omega_{\text{var}}(\lfcn)}$,
$\partial\Omega=\Gamma_a\cup\Gamma_N(\lfcn)\cup\Gampl$, 
$\Omega_{\text{fix}}\cap\Omega_{\text{var}}(\lfcn)=\emptyset$
where 
\[
\Omega_{\text{var}}(\lfcn)=
\{(x',z)\, : \, x'\in \Gamflat, \, 0<z<\lfcn(x')\}, \quad 
\Gamma_N(\lfcn)=
\{(x',\lfcn(x'))\, : \, x'\in \Gamflat\}. 
\]
Thus we can consider $\check{\gfcn}$ as a function defined on $\Gamref= \{(x',\lref(x'))\, : \, x'\in\Gamflat\}$ and determine its value on $\Gamma_N$ by $\gfcn(x',\lfcn(x'))=\check{\gfcn}(x',\lref(x'))$, $x'\in\Gamflat$, which solves the above-mentioned problem of defining a 
%\footnote{known, if we don't optimize for $\gfcn$} 
flux function on an unknown $\Gamma_N$
and of needing a hold-all domain containing all possible domains $\Omega(\lfcn)$.
Of course this simplification also comes with the drawback of less flexibility; in particular, the width of $\Gamma_N$ cannot be modified this way; this could be amended by adding a parametrization of $\partial\Gamflat$ as a design variable.
In order to avoid degeneracy of the variable domain $\Omega_{\text{var}}(\lfcn)$, throughout this paper we will assume boundedness from below of the reference parametrization
\begin{equation}\label{ass_l0}
    \lref >0\text{ and }\|\tfrac{1}{\lref }\|_{L^\infty(\Gamflat)}<\infty
\end{equation}
as well as closeness of $\lfcn$ to $\lref $
\begin{equation}\label{ass_diffl}
    \|\lfcn-\lref \|_{L^\infty(\Gamflat)}\leq\tfrac12 \|{\lref }\|_{L^\infty(\Gamflat)}
\end{equation}
which implies 
\begin{equation}\label{bdd_lfcn}
\lfcn>0\text{ and }\|\tfrac{1}{\lfcn}\|_{L^\infty(\Gamflat)}\leq 2\|\tfrac{1}{\lref }\|_{L^\infty(\Gamflat)}<\infty.
\end{equation} 
In order to guarantee Lipschitz smoothness (as needed for using Stampacchia's method) of the transition between the variable Neumann boundary and the fixed boundary part, we impose (see also \eqref{defG0} below)
\begin{equation}\label{continuity_ell} 
\text{tr}_{\partial \Gamflat}(\lfcn-\lref)=0, \quad \text{tr}_{\partial \Gamflat}\Dnu{(\lfcn-\lref)}=0
\end{equation}
We point to Figure~\ref{fig:domain} for an exemplary setup.
%, where we might simply set $\Gamflat=\Gampl$ (or some line/plane parallel to $\Gampl$). this was referring to the old picture

In \eqref{eqn:var_L2L2}, and the corresponding function spaces \eqref{fcnsp}, the domain $\Omega$ and the boundary part $\Gamma_N$ depend on $\lfcn$. To be able to vary the unknown functions on fixed function spaces in an optimization procedure, we first of all  write the PDE (or actually its variational form \eqref{eqn:var_L2L2}) in terms of the reference domain $\Omref$ and boundary $\Gamref$.

To this end, we make the transformation of variables 
\begin{equation}\label{transf}
\varphi\mapsto\check{\varphi}, \qquad \varphi(x)=\varphi(x',z)=\check{\varphi}(x',\check{z})=\check{\varphi}(\check{x}),
\end{equation}
where  
\[
\begin{aligned}
\check{x}=\begin{cases}
x& \text{ for } x=\check{x}\in\Omega_{\text{fix}}\\
(x',\check{z})\in\Omega_{\text{var}}(\lref ),  & \text{ for }x=(x',z)\in\Omega_{\text{var}}(\lfcn) 
\end{cases}
\text{ with } z=\frac{\lfcn(x')}{\lref(x')} \check{z}
\in[0,\ell(x')],
\quad \check{z}\in[0,\ell_0(x')].
\end{aligned}
\]
The transformation \eqref{transf} implies
\begin{equation}\label{transf_contd}
\begin{aligned}
& \varphi_z(x)= 
%\varphi_z(x',z)=
\begin{cases}
\check{\varphi}_{\check{z}}(\check{x})& \text{ for } x=\check{x}\in\Omega_{\text{fix}}\\
\frac{\lref(x')}{\lfcn(x')} \check{\varphi}_{\check{z}}(x',\check{z})
& \text{ for }x=(x',z)\in\Omega_{\text{var}}(\lfcn), \ \check{z}=\frac{\lref(x')}{\lfcn(x')} z 
\end{cases}\\
&\nabla \varphi(x)
=\left(\begin{array}{c}\nabla_{x'}\varphi(x)\\\varphi_z(x)\end{array}\right)
=:M_\lfcn(\check{x}) \check{\nabla}\check{\varphi}(\check{x}),
\qquad  M_\lfcn=\left(\begin{array}{cc}I&   - \omegazero_\lfcn\check{z}\nabla_{x'}(\tfrac{\lfcn}{\lref })(x')\\0&\omegazero_\lfcn\end{array}\right),\\
&\Delta \varphi(x)=\Delta_{x'}\varphi(x)+\varphi_{zz}(x)=:\check{D}^2_\lfcn\check{\varphi}(\check{x}), \\
&\check{D}^2_\lfcn\check{\varphi}(\check{x})
%=\Delta_{x'}\check{\varphi}(\check{x})
%+2\nabla_{x'}\varphi_z(x',(\tfrac{\lfcn}{\lref })(x')\check{z})\cdot\nabla_{x'}(\tfrac{\lfcn}{\lref })(x')\check{z}
%+\varphi_{zz}(x',(\tfrac{\lfcn}{\lref })(x')\check{z})\check{z}^2|\nabla_{x'}(\tfrac{\lfcn}{\lref })(x')|^2
%+\varphi_z(x',(\tfrac{\lfcn}{\lref })(x')\check{z})\check{z}\Delta_{x'}(\tfrac{\lfcn}{\lref })(x')
%+\omegazero_\lfcn(\check{x})^2 \check{\varphi}_{\check{z}\check{z}}(\check{x})
=\Delta_{x'}\check{\varphi}(\check{x})
- 2\nabla_{x'} \check{\varphi}_{\check{z}}(\check{x})\cdot\nabla_{x'}(\tfrac{\lfcn}{\lref })(x')\, (\tfrac{\lref }{\lfcn})(x')\check{z}
+\check{\varphi}_{\check{z}\check{z}}(\check{x})(\tfrac{\lref }{\lfcn})^2(x')
\check{z}^2|\nabla_{x'}(\tfrac{\lfcn}{\lref })(x')|^2\\
&\qquad\qquad  - \check{\varphi}_{\check{z}}(\check{x})\, (\tfrac{\lref }{\lfcn})(x')\check{z}\Delta_{x'}(\tfrac{\lfcn}{\lref })(x')
+\omegazero_\lfcn(\check{x})^2 \check{\varphi}_{\check{z}\check{z}}(\check{x})\\
&\int_{\Omega_{\text{var}}(\lfcn)} \varphi(x)\, dx =
\int_{\Omega_{\text{var}}(\lref)} \frac{\lfcn(x')}{\lref(x')}\,\check{\varphi}(\check{x})\, d\check{x} , \quad
\int_{\Omega_{\text{var}}(\lfcn)} \varphi(x)\, dx = \int_{\Omega(\lref)}
\frac{1}{\omegazero_\lfcn(\check{x})}
\,\check{\varphi}(\check{x})\, d\check{x}\\
&\int_{\Gamma_N(\lfcn)}\varphi(x)\, dS(x)
= \int_{\Gamflat} \varphi(x',\lfcn(x'))\sqrt{|\nabla_{x'}\lfcn(x')|^2+1}\, dx'
= \int_{\Gamma_N(\lref)}\omegaone_\lfcn(\check{x})\,\check{\varphi}(\check{x})\, dS(\check{x})\\
& \omegazero_\lfcn(\check{x})= 
\begin{cases}
1& \text{ for } \check{x}\in\Omega_{\text{fix}}\\
\frac{\lref(x')}{\lfcn(x')} 
& \text{ for }\check{x}=(x',\check{z})\in\Omega_{\text{var}}(\lref) 
\end{cases}
%\\&\delta_\lfcn(\check{x})=\frac{1}{\omegazero_\lfcn},\qquad
%=\begin{cases}
%1& \text{ for } \check{x}\in\Omega_{\text{fix}}\\
%\frac{\lfcn(x')}{\lref(x')},\quad
%& \text{ for }\check{x}=(x',\check{z})\in\Omega_{\text{var}}(\lref) 
%\end{cases},
\qquad
\omegaone_\lfcn(\check{x})=
\begin{cases}
1& \text{ for } \check{x}\in\Omega_{\text{fix}}\\
\frac{\sqrt{|\nabla_{x'}\lfcn(x')|^2+1}}{\sqrt{|\nabla_{x'}\lref(x')|^2+1}}
& \text{ for }\check{x}=(x',\check{z})\in\Omega_{\text{var}}(\lref) 
\end{cases}.
\end{aligned}
\end{equation}
This conforms to what would be obtained by the general shape calculus approach using the method of mapping (cf., \cite{SokolowskiZolesio} and e.g. \cite{BK-Peichl} in the context of the Westervelt equation), which, for some smooth vector field $\vec{h}:\Omega(\lref)\to\mathbb{R}^d$, applies the transformation $x=\check{x} + \vec{h}(\check{x})$; 
here $\vec{h}(\check{x})=\vec{h}(x',\check{z})=(x',(\tfrac{\lfcn}{\lref }(x')-1)\check{z})$ 
%with $z=\frac{\lfcn(x')}{\lref(x')}\check{z}$ 
on $\Omega_{\text{var}}(\lfcn)$. 

Setting $\Omref=\Omega(\lref)$, $\Gamref=\Gamma_N(\lref)$, we can therefore write the PDE constraint to which we will subject our optimization by using the operator  $A_{\textup{PDE}}$ defined by 
\begin{equation}\label{eqn:APDE}
\begin{aligned}
&\langle A_{\textup{PDE}}((\check{\gfcn},\hfcn,\lfcn),(\check{\pbar},\check{\ptil},\wtil)),(\check{\qbar},\check{\qtil},\vtil,\check{\mu}_N,\mu_{pl})\rangle\\
&:=\int_0^T\Bigl\{
\int_\Omega \Bigl(
\bigl(\pbar_{tt}  - c^{2}\Delta \pbar - b \Delta \pbar_{t}\bigr)\,\qbar\\
&\qquad\qquad+\bigl(\ptil_{tt}  - c^{2}\Delta \ptil - b \Delta \ptil_{t} -2k((\pbar+\ptil)_t)^2 -2k(\pbar+\ptil)(\pbar+\ptil)_{tt}\bigr)\,\qtil\Bigr)\, dx\\
&\qquad +\tfrac{\rho}{\kappa}\int_{\Gampl}  \Bigl(
\bigl(\rho\wtil_{tt}
\Damppl{+ \beta(-\Delpl)^{\gamma} \wtil_{t}} 
- \hfcn-\kappa(\pbar_{t}+\ptil_{t})\bigr)\vtil
+\delta\Delpl\wtil\, \Delpl\vtil \Bigr) \,dS
\\
&\qquad+ 
\int_{\Gamma_N}
%[(\bar{\lambda}+b)\partial_t+c^2]
(\Dnu{\pbar}-\gfcn)\,\mu_{N}\, dS
+\int_{\Gampl}
%[(\tilde{\lambda}(1-2k(\pbar+\ptil))+b)\partial_t+c^2]
(\Dnu{\ptil} + \rho \wtil_t)\,\mu_{pl}\,dS
\Bigr\}\,dt
\\
=&\int_0^T\Bigl\{
\int_{\Omref} \frac{1}{\omegazero_\lfcn}
\Bigl(\check{\pbar}_{tt}\, \check{\qbar} - (c^2\check{D}^2_\lfcn \check{\pbar} +b\check{D}^2_\lfcn \check{\pbar}_t) \, \check{\qbar} \\
&\qquad\qquad+\bigl(\check{\ptil}_{tt}  - c^{2}\check{D}^2_\lfcn  \check{\ptil} - b \check{D}^2_\lfcn  \check{\ptil}_{t}  
-2k((\check{\pbar}+\check{\ptil})_t)^2 -2k(\check{\pbar}+\check{\ptil})(\check{\pbar}+\check{\ptil})_{tt}\bigr)\,\check{\qtil}\Bigr)\, dx
\\
&\qquad +\tfrac{\rho}{\kappa}\int_{\Gampl}  \Bigl(
\bigl(\rho\wtil_{tt}
\Damppl{+ \beta(-\Delpl)^{\gamma} \wtil_{t}} 
+ \hfcn+\kappa(\check{\pbar}_{t}+\check{\ptil}_{t})\bigr)\vtil
+\delta\Delpl\wtil\, \Delpl\vtil \Bigr) \,dS
\\
&\qquad+ 
\int_{\Gamref}\omegaone_\lfcn 
%[(\bar{\lambda}+b)\partial_t+c^2]
(\nu_0\cdot M_\lfcn\,\nabla_{\check{x}}\check{\pbar}-\check{\gfcn})\,\check{\mu}_{N}\, dS
+\int_{\Gampl}
%[(\tilde{\lambda}(1-2k(\pbar+\ptil))+b)\partial_t+c^2]
(\Dnu{\ptil} + \rho \wtil_t)\,\mu_{pl}\,dS
\Bigr\}\,dt,
\end{aligned}
\end{equation}
cf. \eqref{eqn:var_L2L2}.

Additionally, we assume that the initial data $\check{p}_0,\check{p}_1,\wtil_0,\wtil_1$ has been chosen in such a way that the compatibility conditions 
\begin{equation}\label{compat_cor}
\begin{aligned}
%p_1+c\Dnu{p_0}
&\Dnu{p_0}+\absbc[p_0,p_1]
=0, \ 
%p_2+c\Dnu{p_1}
\Dnu{p_2}+\absbc[p_1,p_2]
=0\ \text{ on }\Gamma_a\\
%p_0 =g_D(0), \ p_1 =g_{D\,t}(0), \ p_2 =g_{D\,tt}(0)\ &\text{ on }\Gamma_D\\
&\partial_\nu p_0 = -\rho\wtil_1, \  \partial_\nu p_1 = -\rho\wtil_2, 
%\  \partial_\nu p_2 = -\rho\wtil_2\ 
\text{ on }\Gampl\\
&\text{where }p_2 := \tfrac{1}{1-2kp_0}\Bigl(c^2\Delta p_{0}+b\Delta p_1\Bigr), \quad
\tilde{w}_2\:=\frac{1}{\rho}\Bigl(-\delta\Delpl^2\tilde{w}_0
\Damppl{-\beta(-\Delpl)^\gamma\tilde{w}_1}
+\kappa {p_1}+\hfcn(0)\Bigr)
\end{aligned}
\end{equation}
are satisfied,
%\footnote{Note that $\bar{p}(0)=\bar{p}_t(0)$ and thus, by (homogeneous) PDE for $\bar{p}$, also $\bar{p}_{tt}(0)=\bar{p}_{ttt}(0)$} 
cf. \cite[Corollary 3.7]{KaTu1}.

{The compatibility conditions on the boundary data can be incorporated into the definition of the affine space $G_{ad}$, see \eqref{defG0}, \eqref{compat}  below.}

\medskip

With this, we consider the minimization problem 
\begin{equation}\label{opt}
\begin{aligned}
&\min_{\vec{g}\in G_{\textup{ad}},\,\vec{u}\in U_{\textup{ad}}}\ J(\vec{g},\vec{u}) \\
&\text{s.t.} \langle A_{\textup{PDE}}(\vec{g},\vec{u}),\vec{z}\rangle =0 \quad \vec{z}\in Z,
%\\&\phantom{\text{s.t.}} \langle A_{\textup{CMP}}(\vec{g}),\vec{\lambda}\rangle =0 \quad \vec{\lambda}\in\Lambda
\end{aligned}
\end{equation}
with $\vec{g}=(\check{\gfcn},\hfcn,\lfcn)$, $\vec{u}=(\check{\pbar},\check{\ptil},\wtil)$, $\vec{z}=(\check{\qbar},\check{\qtil},\vtil,\check{\mu}_{N},\mu_{pl})\in Z$ and having in mind a tracking plus regularization type objective functional of the form
\begin{equation}\label{J}
\begin{aligned}
&J((\check{\gfcn},\hfcn,\lfcn),(\check{\pbar},\check{\ptil},\wtil))
:=\frac12\int_0^T\Bigl(\int_{\Omega_{\textup{ROI}}}|\check{\pbar}+\check{\pbar}-p_{\textup{d}}|^2 \, dx
+\int_{\Gampl} |\wtil-\wtil_d|^2\, dS(x) \Bigr)\, dt
%\\&\phantom{J((\check{\gfcn},\hfcn,\lfcn),(\check{\pbar},\check{\ptil},\wtil)):=}
+\frac{\regpar}{2} \mathcal{R}(\check{\gfcn},\hfcn,\lfcn)
\end{aligned}
\end{equation}
for given target pressure and displacement distributions $p_{\textup{d}}$, $\wtil_d$ on a fixed region of interest $\Omega_{\textup{ROI}}\subseteq\Omega_{\text{fix}}=\Omega\setminus\Omega_{\text{var}(\lfcn)}$, (while $\Omega$ and $\Gamma_N$ vary with $\lfcn$)  and a regularization/control cost parameter $\regpar\geq0$.
Here $G$ is defined as in \eqref{G} below, 
\begin{equation}\label{R}
\begin{aligned}
\mathcal{R}(\check{\gfcn},\hfcn,\lfcn)
=&\Bigl(\int_0^T \int_{\Gamref} \Bigl(|\regop^{\text{time}}_\gfcn (\check{\gfcn}-\check{\gfcn}_0)(\check{x})|^2+|\regop^{\text{space}}_\gfcn (\check{\gfcn}-\check{\gfcn}_0)(\check{x})|^2\, dS(\check{x}) \, dt\\
&+\int_{\Gampl} |\regop_\hfcn (\hfcn-\hfcn_0)(x)|^2\, dS(x)\, dt
+\int_{\Gamflat} |\regop_\lfcn (\lfcn-\lref)(x')|^2\, dS(x')\Bigr),
\end{aligned}
\end{equation}
with a priori guesses $\gfcn_0$, $\hfcn_0$, as well as $\lref$ as in \eqref{ass_l0},
and simple space-time differential operators $\regop^{\text{time}}_\gfcn$, $\regop^{\text{space}}_\gfcn$, $\regop_\hfcn$, $\regop_\lfcn$ so that 
\begin{equation}\label{Rnorm}
\mathcal{R}(\check{\gfcn},\hfcn,\lfcn)
\geq c \Vert (\check{\gfcn}-\check{\gfcn}_0,\hfcn-\hfcn_0,\lfcn-\lref)\Vert_G^2,
\end{equation}
see \eqref{regops} below.
That is, the $\regpar$ term defines a squared norm whose boundedness guarantees enough regularity of $\gfcn$, $\hfcn$, $\lfcn$,  
%\footnote{\ntBK{We will need $\gfcn\in L^\infty(H^s)$ so that we can obtain $\pbar\in L^\infty(L^\infty)$ for $\pbar$ being part of a minimizer (and thus PDE solution)}} 
%\footnote{$s_\gfcn$, $s_\hfcn$ and $s_\lfcn$ need to be chosen such that the smoothness requirements for $\partial\Omega$ ($C^{1,1}$?) needed for PDE well-posedness are guaranteed; due to appearance of the curvature $H_\lfcn$ we probably need $s_\lfcn\geq2$}
%\footnote{Sufficient strength of the topology in the existence proof of a minimizer is another criterion for the choice of $\regop_\gfcn$, $\regop_\hfcn$, $\regop_\lfcn$.}
and in particular allows us to avoid employing inequality constraints for this purpose.

{
Here we abbreviate (cf. \eqref{continuity_ell})
\begin{equation}\label{defG0}
\begin{aligned}
\vec{g}&=(\check{\gfcn},\hfcn,\lfcn)\in G_{\textup{ad}}:=\{(\gfcn_0+t\gfcn_1,\hfcn_0,\lref )\}+G_0, \\ 
G_0&=\{(\uld{\check{\gfcn}},\uld{\hfcn},\uld{\lfcn})\in G\, : \, 
\check{\gfcn}(t=0)
=\check{\gfcn}_t(t=0)
=0, \ \hfcn(t=0)=0, \
\text{tr}_{\partial \Gamflat}\uld{\lfcn}=0, \ \text{tr}_{\partial \Gamflat}\Dnu{\uld{\lfcn}}=0,
%, \ \text{tr}_{\partial \Gamflat}\Dnu^2{\uld{\lfcn}}=0
\}
\end{aligned}
\end{equation}
where $\check{\gfcn}_0,\,\hfcn_0$ satisfy the compatibility conditions
%\footnote{\ntBK{Fix compatibility condition as soon as we see how much regularity we need for $\gfcn$, etc., i.e., as soon as we have determined which norms on $\gfcn$ will be used in the objective function}}
\begin{equation}\label{compat}
\check{\gfcn}_0(0)=\Dnu{\check{p}_0}, \quad 
\check{\gfcn}_1(0)=\Dnu{\check{p}_1} 
\end{equation}
and use the spaces defined in \eqref{fcnsp} together with the parameter space
\begin{equation}\label{G}
G=W^{1,\infty}(\zeroT H^{s_\gfcn}(\Gamref))\times 
H^1(\zeroT L^2(\Gampl))\times 
H^{s_\lfcn}(\Gamflat).
\end{equation}
}

To achieve \eqref{Rnorm}, we may choose
\begin{equation}\label{regops}
\regop^{\text{time}}_\gfcn= \partial_t^2,\quad
\regop^{\text{space}}_\gfcn= (-\Delta_N+\text{id})^{s_\gfcn},\quad
\regop_\hfcn=\partial_t,\quad 
\regop_\lfcn=(-\Delta_N+\text{id})^{s_\lfcn/2} 
\end{equation}
with the homogeneous Neumann Laplacian $-\Delta_N$ on $\Gamref$ and on $\Gamflat$, respectively.
Indeed, by the one-dimensional version of Agmon's inequality we have, for $g'=(\gfcn-\gfcn_0)_t$
\[\begin{aligned}
\Vert g'(t)\Vert_{H^{s_\gfcn}}^2
&=\int_{\Gamref} ((-\Delta_N+\text{id})^{s_\gfcn/2}g'(t,\check{x}))^2\, dS(\check{x})\\
&=\int_0^t\int_{\Gamref} 2((-\Delta_N+\text{id})^{s_\gfcn/2}g'(\tau,\check{x}))\,((-\Delta_N+\text{id})^{s_\gfcn/2}g'_t(\tau,\check{x}))\, dS(\check{x})\, d\tau\\
&=\int_0^t\int_{\Gamref} 2\,(-\Delta_N+\text{id})^{s_\gfcn}g'(\tau,\check{x})\,
g'_t(\tau,\check{x})\, dS(\check{x})\, d\tau\\
&\leq \int_0^t\int_{\Gamref} ((-\Delta_N+\text{id})^{s_\gfcn}g'(\tau,\check{x}))^2\, dS(\check{x})\, d\tau
+\int_0^t\int_{\Gamref} (g'_t(\tau,\check{x}))^2\, dS(\check{x})\, d\tau,
\end{aligned}\]
where we have used Cauchy-Schwarz' and Young's inequalities in the last step.
Due to $(\gfcn-\gfcn_0)(0)=0$, this already provides equivalence to the full norm.

We will work with three levels of regularity: 
For mere well-definedness of the PDE constraint we assume
\begin{equation}\label{lowreg}
s_\gfcn\geq-\frac12, \quad 
s_\lfcn >\frac{d+1}{2},
\end{equation}
so that $H^{s_\lfcn}(\Gamflat)$ continuously embeds into $W^{1,\infty}(\Gamflat)$ (in fact, even into $C^1(\Gamflat)$);
\\
For existence of a minimizer, requiring uniform boundedness of the trace operator on $\Gamma_N$, which affects regularity of both $\gfcn$ and $\lfcn$, we assume 
\begin{equation}\label{medreg}
s_g\begin{cases}\geq0 \text{ if }d=2\\>0 \text{ if }d=3\end{cases}, \quad s_\ell>\frac{d+1}{2}, 
%\min\{s_g,s_\ell-\frac{d+3}{2}\}\geq0, 
\end{equation}
so that $\gfcn_t\in L^\infty(0,T;L^{\mathfrak{s}}(\Gamma_N))$, with $\mathfrak{s}\geq 2$, $\mathfrak{s}>d-1$,
%s_g-d/2 > -d/(d-1)  
$\lfcn\in W^{1,\infty}(\Gamflat)$
%$\min\{s_g,s_\ell-\frac{d+3}{2}\}\geq0$
cf. Lemma~\ref{lem:estAPDE0} and in particular 
\eqref{g}, \eqref{Stampacchia}, \eqref{Stampacchia2}, \eqref{l} in its proof.
\\
For proving Fr\'{e}chet differentiability of the PDE constraint in the justification of first order necessary optimality conditions, full $H^2(\Omega)$ regularity of $\pbar$ will needed, which we will obtain from membership in $H^2_\Delta(\Omega)$ via elliptic regularity. To this end,
\begin{equation}\label{highreg}
s_\gfcn\geq\frac12, \quad 
s_\lfcn >\frac{d+3}{2},
\end{equation}
(so that $H^{s_\lfcn}(\Gamflat)\hookrightarrow W^{2,\infty}$)
% s_\lfcn-2-(d-1)/2>0
will be required. Indeed, with $\lfcn\in W^{2,\infty}(\Gamflat)$ and $\Gamma_a\cup\Gampl\in C^{1,1}$, the overall boundary is piecewise $W^{2,\infty}$ with $C^1$ transitions according to \eqref{defG0}, thus globally in $W^{2,\infty}=C^{1,1}$.
% for a proof of this fact, see, e.g., Theorem 2.4 in NumPDENotes21.pdf, Christian Clason
With the choice \eqref{regops}, the $\regop_\lfcn$ term in \eqref{J}, \eqref{R}  guarantees boundedness of $\lfcn$ in $H^{s_\lfcn}(\Gamflat)$.
Boundedness of $J$ (and in particular its last term with large enough penalty parameter $\regpar$) guarantees smallness of $\lfcn-\lref$ in $H^{s_\lfcn}(\Gamflat)$ and thus, via \eqref{ass_diffl}, nondegeneracy \eqref{bdd_lfcn} of the Neumann boundary part.

\section{Existence of a minimizer}\label{sec:existence_min}

\begin{Theorem}
Let the spaces $G_{ad}$, $U_{ad}=U(\Omref)$ be  as in \eqref{fcnsp}, \eqref{defG0}, \eqref{G}, 
%\eqref{lowreg}
\eqref{medreg}. Then there exists a minimizer $(\vec{g}_*,\vec{u}_*) \in G_{ad} \times U_{ad} $  to the problem defined in \eqref{opt}.
\end{Theorem}

\begin{proof}
We use the direct method of calculus of variation.
%For existence of a minimizer to \eqref{opt} we require the following
%\begin{enumerate}
%\item[1.] 
\\[1ex]\textbf{Step 1.} 
$G_{\textup{ad}}\times U_{\textup{ad}}\not=\emptyset$.\\
To this end, it suffices to take $\vec{g}=(\gfcn_0,\hfcn_0,\lref)$ and $(\pbar,\ptil,\wtil)$ some sufficiently smooth extension of the initial and boundary data. Note that a point in $G_{\textup{ad}}\times U_{\textup{ad}}$ does not need to satisfy the (PDE) equality constraints.
%\item[2.] 
\\[1ex]\textbf{Step 2.}
$J$ is bounded from below on $G_{\textup{ad}}\times U_{\textup{ad}}\cap \{A_{\textup{PDE}}=0\}$.\\
This is trivially satisfied since we take $J$ to be a linear combination of norms. We then denote the infimum of $J$ on $G_{\textup{ad}}\times U_{\textup{ad}}$ by $J_*$, and consider  a minimizing sequence $\{ (\vec{g}_{n}, \vec{u}_{n})\} \in G_{\textup{ad}}\times U_{\textup{ad}}$
such that $J(\vec{g}_{n}, \vec{u}_{n}) \to J_*$ as $n \to \infty$.
%\item[3.] 
\\[1ex]\textbf{Step 3.}
Compactness of sublevel sets.\\
In order to extract a a convergent subsequence of $\{ (\vec{g}_{n}, \vec{u}_{n})\}$, we show  that sublevel sets of $J$ in $G_{\textup{ad}}\times U_{\textup{ad}}\cap \{A_{\textup{PDE}}=0\}$ are  compact with respect to a suitable topology $\mathcal{T}_G\times \mathcal{T}_U \supset G_{\textup{ad}}\times U_{\textup{ad}} $.
We can achieve this by taking $\mathcal{T}_G$ to be the weak(*) topology with respect to the regularization terms used in $J$ (which for this purpose have been chosen as norms on reflexive spaces or on duals of separable spaces).

Since the regularization/penalty terms in $J$ first of all only imply bounds on $\vec{g}_n$, we will have to make use of energy estimates in order to obtain bounds on $\vec{u}_n$ as well. 
To this end, we introduce the energy
\begin{align}\label{def_energy}    
&\mathcal{E}[\pbar,\ptil,\wttil](t) =
\mathcal{E}_p[\pbar](t)+\mathcal{E}_p[\ptil](t) + \mathcal{E}_w[\wttil](t) \\
&\mathcal{E}_p[\ptil](t) = 
\int_0^t\|\ptil_{tt}(s)\|_{L^2(\Omega)}^2\, ds
+\|\ptil_t(t)\|_{H^{1}(\Omega)}^2
+\|\Delta\ptil(t)\|_{L^2(\Omega)}^2+ b \int_0^t\|\Delta\ptil_t\|_{L^2(\Omega)}^2\, ds\\
&\hspace*{3cm}+\betaabs \int_0^t\|\ptil_{tt}\|_{L^2(\Gamma_a)}^2\, ds
+\frac{\gammaabs}{2} \|\ptil_{t}(t)\|_{L^2(\Gamma_a)}^2\\
&\mathcal{E}_w[\wttil](t)=
\|\wttil_t(t)\|_{L^2(\Gampl)}^2
+\|\Delta_{pl}\wttil(t)\|_{L^2(\Gampl)}^2
\Damppl{+\beta \int_0^t\|(-\Delpl)^{\gamma/2} \wttil_{t}\|_{L^2(\Gampl)}^2\, ds}
\end{align}
%For this purpose it suffices to derive lower order (note that the terms in the limits \eqref{differenceAPDEinterior} below only involve up to first derivative terms) energy estimates 
\begin{Lemma}\label{lem:estAPDE0}
There exist constants $\tilde{C}(T),\, \bar{C}(T),\, m_0>0$ depending only on $T$, $\Omref$, $\Gamref$ (but not on $\Omega_n$, $\Gamma_{N,n}$) such that for all $\vec{g}_n\in G_{\text{ad}}$ with 
$s_g\geq0$ if $d=2$, $s_g>0$ if $d=3$, $s_\ell > \frac{d-1}{2}+1$ cf. \eqref{lowreg}, \eqref{medreg}, 
such that $\lfcn=\lfcn_n$ satisfies \eqref{ass_diffl}, the PDE constraint $A_{\text{PDE}}(\vec{g}_n,\vec{u}_n)=0$ implies the estimate 
\[
\mathcal{E}[\pbar,\ptil,\wtil_t](t)
\leq \tilde{C}(T)\Bigl(\mathcal{E}[\pbar,\ptil,\wtil_t](0) 
+\|\text{data}\|^2
\Bigr)
\]
(noting that $\wtil_t(0)=w_2$ as in \eqref{compat_cor}) with 
\begin{equation}\label{data}
\|\text{data}\|^2:= \|g\|_{L^2(\zeroT H^{s_g}(\Gamma_N))}^2
+\|g_{t}\|_{H^1(\zeroT {L^2}
%H^{-1/2}
(\Gamma_{N,n}))}^2
\Damppl{+\|h_{t}\|_{L^2(\zeroT H^{-\gamma}(\Gampl))}^2}
+\|h_{t}\|_{L^1(\zeroT L^2(\Gampl))},
\end{equation}
provided $\mathcal{E}[\pbar,\ptil,\wtil_t](0) +\|\text{data}\|^2\leq m_0$.
\\
Moreover, for any $n,m\in\mathbb{N}$,  we have the full norm and trace estimates
\begin{equation}\label{est_full_and_trace}
\begin{aligned}
&\|\pbar_{n}\|_{L^2(\zeroT H^{\frac{3}{2}+\epsilon}(\Omega_n))}^2
+\|\text{tr}_{\Gamma_{N,m}}\nabla\pbar_{n}\|_{L^2(\zeroT H^{\epsilon}(\Gamma_{N,m}))}^2
\\&\leq \bar{C}(T)\bigl(\|\Delta\pbar_{n}\|_{L^2(\zeroT L^2(\Omega_n))}^2 + \|g_n\|_{L^2(\zeroT H^{s_g}(\Gamma_N))}^2\bigr)\\
&\|\ptil_{n}\|_{L^2(\zeroT H^{\frac{3}{2}+\epsilon}(\Omega_n))}^2
+\|\text{tr}_{\Gamma_{N,m}}\nabla\ptil_{n}\|_{L^2(\zeroT H^{\epsilon}(\Gamma_{N,m}))}^2
\\&\leq C\bigl(\|\Delta\ptil_{n}\|_{L^2(\zeroT L^2(\Omega_n))}^2 + \|\wtil_{n\,t}\|_{L^2(\zeroT H^{s_w}(\Gampl))}^2\bigr)
\\
&\text{with }0\leq\epsilon\leq\min\{s_g,s_w,\frac12\}, \quad 
%\epsilon<s_\ell-\frac{d+1}{2}, \quad 
s_w\geq0
\end{aligned}
\end{equation}
\end{Lemma}
%\begin{proof}
\textit{Proof.} For the energy estimate we refer to the appendix and note that this estimate does not follow from those already made in \cite{KaTu1}. 
%since we use $\wtil$ with a different meaning here.
\\
For the full norm and trace estimates we refer to, e.g., 
\cite[Theorem 4]{Savare1998} 
combined with interpolation in case $\epsilon>0$
%\footnote{\colBK{But I think we actually just need the case $\epsilon=0$ here}}
% by interpolation we expect to obtain elliptic regularity $p\in H^{3/2+\epsilon}$ for $\lfcn\in W^{1+\epsilon,\infty}(\Gamflat)$
and \cite[Theorem 1]{Ding_1996}. 
%and point to the fact that $\lfcn\in H^{s_\lfcn}(\Gamflat)\subseteq W^{1,\infty}$  implies Lipschitzness of $\Gamma_N$.
%\end{proof}
%\item[4.]
\\[1ex]\textbf{Step 4.}
Lower semicontinuity of $J$ and closedness of overall admissible set (including the PDE constraints) with respect to  $\mathcal{T}_G\times\mathcal{T}_U$.\\
 Since $J$ is convex and lower semicontinuous on $G_{\textup{ad}}\times U_{\textup{ad}}$, it is $\mathcal{T}_G\times\mathcal{T}_U$ lower semicontinuous.
In fact, with the weak(*) topology induced by $J$, lower semicontinuity is always satisfied due to the Theorems by Banach-Alaoglu or Eberlein-Smulian with the one by Katukani, under the above conditions. 
The difficulty lies in proving that the weak limit solves the nonlinear PDE. By choosing sufficiently strong regularization norms in $J$, we lift $\mathcal{T}_G$ to a strength that enables this.
%\footnote{\ntBK{Since we have to take limits in \eqref{eqn:APDE}, we will have to make the $\lfcn$ term in $J$ strong enough so that we can use compactness to get convergence of the nonlinear terms $\omegazero_\lfcn$, $\omegaone_\lfcn$, etc. Concerning the nonlinear $\pbar,\ptil$ terms, we can make use of the fact that they are not only in $U$ but also satisfy the boundary conditions, which we can use to lift regularity; if this does not suffice, we can also add some regularizing $\pbar$, $\ptil$ terms in $J$.}}
 
 More precisely, we show the following.
\[\begin{aligned}
&\Bigl((\vec{g}_n,\vec{u}_n)\stackrel{\mathcal{T}_G\times\mathcal{T}_U}{\longrightarrow} (\vec{g}_*,\vec{u}_*) \text{ and }A_{\textup{PDE}}(\vec{g}_n,\vec{u}_n)=0\ \forall n \Bigr)\\ 
&\Rightarrow \ \Bigl(J(\vec{g}_*,\vec{u}_*)\leq\lim_{n\to\infty}J(\vec{g}_n,\vec{u}_n)\text{ and }A_{\textup{PDE}}(\vec{g}_*,\vec{u}_*)=0\Bigr)
\end{aligned}
\]

Due to $A_{\textup{PDE}}(\vec{g}_n,\vec{u}_n)=0$, we have
\begin{equation}\label{differenceAPDE}
A_{\textup{PDE}}(\vec{g}_*,\vec{u}_*)=
A_{\textup{PDE}}(\vec{g}_*,\vec{u}_*)-A_{\textup{PDE}}(\vec{g}_n,\vec{u}_n).
\end{equation}
The limits below are taken along subsequences whose convergence follows from compactness as well as weak compactness due to boundedness and reflexivity of the spaces, while suppressing subsequence indices notationally. Due to uniqueness of (weak) limits, these in fact need to coincide with the respective components of $\vec{g}_*,\vec{u}_*$.
Besides weak convergence according to the uniform bounds from Lemma~\ref{lem:estAPDE0}, we will also use compact embeddings to conclude the following strong convergence (along subsequences)
\begin{equation}\label{limits_muMomega}
\begin{aligned}
&\omegazero_{\ell_n}\to\omegazero_{\ell_*}, \quad M_{\ell_n}\to M_{\ell_*}, \quad \omegaone_{\ell_n}\to\omegaone_{\ell_*}
\quad \text{ in }C(\Gamflat)\text{ for }s_\ell> \frac{d}{2}+1 \text{ cf. \eqref{lowreg}}\\
&\check{\pbar}_n+\check{\ptil}_n\to \check{\pbar}_*+\check{\ptil}_*
\quad \text{ in }L^2(\zeroT L^2(\Omref))
\end{aligned}
\end{equation} 

We use the fact that $A_{\textup{PDE}}$ splits into three parts: (a) an interior part, (b) a plate part, (c) a Neumann boundary part. 
\[
\begin{aligned}
&\langle A_{\textup{PDE}}(\vec{g}_*,\vec{u}_*)-A_{\textup{PDE}}(\vec{g}_n,\vec{u}_n),(\check{\qbar},\check{\qtil},\vtil,\check{\mu}_N,\mu_{pl})\rangle\\
&=
\langle A_{\textup{PDE}}(\vec{g}_*,\vec{u}_*)-A_{\textup{PDE}}(\vec{g}_n,\vec{u}_n),
(\check{\qbar},\check{\qtil},0,0,0)\rangle\\
&\quad+\langle A_{\textup{PDE}}(\vec{g}_*,\vec{u}_*)-A_{\textup{PDE}}(\vec{g}_n,\vec{u}_n),
(0,0,\vtil,0,\mu_{pl})\rangle\\
&\quad+\langle A_{\textup{PDE}}(\vec{g}_*,\vec{u}_*)-A_{\textup{PDE}}(\vec{g}_n,\vec{u}_n),
(0,0,0,\check{\mu}_N,0)\rangle.
\end{aligned}
\]
PDE nonlinearity only affects the interior part; domain variation takes effect only in the interior and the Neumann boundary part.
\\
We start with the interior part by considering $(\check{\qbar},\check{\qtil})\in C_0^\infty((0,T)\times\Omega)^2$ arbitrary fixed and use integration by parts to move most of the derivatives to the smooth test functions (we do not move the whole Laplacian over, since this would involve second derivatives of $\lfcn$ in $D^2_{\lfcn_n}$)
\[ 
\begin{aligned}
&\langle A_{\textup{PDE}}((\check{\gfcn}_n,\hfcn_n,\lfcn_n),(\check{\pbar}_n,\check{\ptil}_n,\wtil_n)),(\check{\qbar},\check{\qtil},0,0,0)\rangle
\\
&=\int_0^T\Bigl\{
\int_{\Omega(\lfcn_n)} \Bigl(
\pbar_n\,\qbar^{\lfcn_n}_{tt}  + \nabla \pbar_n\cdot\nabla(c^2\qbar^{\lfcn_n} - b\qbar^{\lfcn_n}_t)
\\
&\qquad\qquad+
\ptil_n\,\qtil^{\lfcn_n}_{tt}  + \nabla \ptil_n\cdot\nabla(c^2\qtil^{\lfcn_n} - b\qtil^{\lfcn_n}_t)
-k(\pbar_n+\ptil_n)^2\,\qtil^{\lfcn_n}_{tt}\Bigr)\, dx
\Bigr\}\,dt
\\
&=\int_0^T\Bigl\{
\int_{\Omref} \frac{1}{\omegazero_{\lfcn_n}}
\Bigl(
\check{\pbar}_n\,\check{\qbar}_{tt}  + M_{\lfcn_n}\check{\nabla} \check{\pbar}_n\cdot (M_{\lfcn_n}\check{\nabla}(c^2\check{\qbar} - b \check{\qbar}_t))
\\
&\qquad\qquad\qquad+
\check{\ptil}_n\,\check{\qtil}_{tt}  + M_{\lfcn_n}\check{\nabla} \check{\ptil}_n\cdot (M_{\lfcn_n}\check{\nabla}(c^2\check{\qtil} - b \check{\qtil}_t))
-k(\check{\pbar}_n+\check{\ptil}_n)^2\,\check{\qtil}_{tt}\Bigr)
\, dx \Bigr\}\,dt
\end{aligned}
\]
(where the superscript ${}^{\ell_n}$ indicates that the transformations of the test functions to the domain $\Omega_{\lfcn_n}$ depend on $\lfcn_n$) 
and likewise for $\langle A_{\textup{PDE}}((\check{\gfcn}_*,\hfcn_*,\lfcn_*),(\check{\pbar}_*,\check{\ptil}_*,\wtil_*)),(\check{\qbar},\check{\qtil},\vtil,0,0)\rangle$.
Therewith % \eqref{differenceAPDE}
\begin{equation}\label{differenceAPDEinterior}
\begin{aligned}
&\langle A_{\textup{PDE}}(\vec{g}_*,\vec{u}_*)-A_{\textup{PDE}}(\vec{g}_n,\vec{u}_n),
(\check{\qbar},\check{\qtil},0,0,0)\rangle\\
&=\int_0^T\Bigl\{
\int_{\Omref} \Bigl(\frac{1}{\omegazero_{\lfcn_*}}-\frac{1}{\omegazero_{\lfcn_n}}\Bigr)
\Bigl(
\check{\pbar}_n\,\check{\qbar}_{tt}  + M_{\lfcn_n}\check{\nabla} \check{\pbar}_n\cdot (M_{\lfcn_n}\check{\nabla}(c^2\check{\qbar} - b \check{\qbar}_t))
\\
&\qquad\qquad\qquad\qquad\quad+
\check{\ptil}_n\,\check{\qtil}_{tt}  + M_{\lfcn_n}\check{\nabla} \check{\ptil}_n\cdot (M_{\lfcn_n}\check{\nabla}(c^2\check{\qtil} - b \check{\qtil}_t))
-k(\check{\pbar}_n+\check{\ptil}_n)^2\,\check{\qtil}_{tt}\Bigr)
\\
&\qquad\qquad+
\frac{1}{\omegazero_{\lfcn_*}}
\Bigl(
(M_{\lfcn_*}-M_{\lfcn_n})\check{\nabla} \check{\pbar}_n\cdot (M_{\lfcn_n}\check{\nabla}(c^2\check{\qbar}-b\check{\qbar}_t)
+M_{\lfcn_*}\check{\nabla} \check{\pbar}_n\cdot ((M_{\lfcn_*}-M_{\lfcn_n})\check{\nabla}(c^2\check{\qbar}-b\check{\qbar}_t)
\\
&\qquad\qquad\qquad\qquad+
(\check{\pbar}_*-\check{\pbar}_n)\,\check{\qbar}_{tt}  
+ (M_{\lfcn_*}\check{\nabla} (\check{\pbar}_*-\check{\pbar}_n))\cdot (M_{\lfcn_*}\check{\nabla}(\check{\qbar} - b \check{\qbar}_t))
\\
&\qquad\qquad\qquad\qquad
-k(\check{\pbar}_*+\check{\ptil}_*+\check{\pbar}_n+\check{\ptil}_n)((\check{\pbar}_*-\check{\pbar}_n)+(\check{\ptil}_*-\check{\ptil}_n))\,\check{\qtil}_{tt}\Bigr)\, dx
\Bigr\}\,dt\\
&\to 0 \text{ as }n\to\infty, 
\end{aligned}
\end{equation}
where we have used the uniform bounds from Lemma~\ref{lem:estAPDE0},  as well as the limits \eqref{limits_muMomega}.

For the plate part with arbitrary $(\vtil,\mu_{pl})\in C_0^\infty((0,T)\times\Gampl)^2$, again moving most of the derivatives to the smooth test functions and using linearity, we have 
\[ 
\begin{aligned}
&\langle A_{\textup{PDE}}(\vec{g}_*,\vec{u}_*)-A_{\textup{PDE}}(\vec{g}_n,\vec{u}_n),
(0,0,\vtil,0,\mu_{pl})\rangle
\\
&=\int_0^T\Bigl\{
\tfrac{\rho}{\kappa}\int_{\Gampl}  \Bigl(
(\wtil_*-\wtil_n)\,(\rho\vtil_{tt}
\Damppl{- \beta(-\Delpl)^{\gamma} \vtil_{t})} 
- \kappa(\pbar_*-\pbar_n+\ptil_*-\ptil_n)\vtil_t
+\delta\Delpl(\wtil_*-\wtil_n)\, \Delpl\vtil \Bigr) \,dS\\
&\qquad
+\int_{\Gampl}
\bigl(\Dnu{(\ptil_*-\ptil_n)}\mu_{pl} - \rho (\wtil_*-\wtil_n)\,\mu_{pl\,t}\bigr)\,dS
\Bigr\}\,dt\\
&\to 0 \text{ as }n\to\infty 
\end{aligned}
\]
by using boundedness and hence weak convergence according to Lemma~\ref{lem:estAPDE0}, in particular also the trace estimate \eqref{est_full_and_trace} with 
%$\min\{s_g,s_\ell-\frac{d+3}{2}\}\geq0$, 
$\epsilon=0$,
that implies weak convergence of $\Dnu{(\ptil_*-\ptil_n)}$.

Finally, for the Neumann boundary part, testing with arbitrary $\check{\mu}_{N}\in C_0^\infty((0,T)\times\Gamref)$
\[ 
\begin{aligned}
&\langle A_{\textup{PDE}}(\vec{g}_*,\vec{u}_*)-A_{\textup{PDE}}(\vec{g}_n,\vec{u}_n),
(0,0,0,\check{\mu}_{N},0)\rangle
\\
&=\int_0^T
\int_{\Gamref}\Bigl(
(\omegaone_{\lfcn_*}-\omegaone_{\lfcn_n}) (\nu_0\cdot M_{\lfcn_n}\,\nabla_{\check{x}}\check{\pbar}_n-\check{\gfcn}_n)\\
&\qquad\qquad\qquad
+\omegaone_{\lfcn_*} \Bigl((\nu_0\cdot (M_{\lfcn_*}-M_{\lfcn_n})\,\nabla_{\check{x}}\check{\pbar}_n
+\nu_0\cdot M_{\lfcn_*}\,\nabla_{\check{x}} (\check{\pbar}_*-\check{\pbar}_n)
-(\check{\gfcn}_*-\check{\gfcn}_n)
\Bigr)\check{\mu}_{N}\, dS
\\
&\to 0 \text{ as }n\to\infty 
\end{aligned}
\]
again using of the trace estimate \eqref{est_full_and_trace} with \eqref{medreg}, as well as \eqref{limits_muMomega}.
%
\begin{comment}
%\item[5.]
\\[1ex]\textbf{Step 5.}
\colAT{Uniqueness of the minimizer $(\vec{g}_{\star}, \vec{u}_{\star})$ follows from strict convexity of $J$ in $(\vec{g}, \vec{u})$.}

\ntBK{I am not completely sure about this since due to nonlinearity of $A_{PDE}$ the overall admissible set might be nonconvex. Local strict convexity might follow by controlling the nonlinearity in the way we have done in the energy estimates.}
\end{comment}

Since $(\check{\qbar},\check{\qtil},\check{\mu}_{N},\vtil,\mu_{pl})\in C_0^\infty((0,T)\times\Omega)^2\times C_0^\infty((0,T)\times\Gampl)\times C_0^\infty((0,T)\times\Gamref)\times C_0^\infty((0,T)\times\Gampl)$ is arbitrary here, due to \eqref{differenceAPDE} we have shown $A_{\textup{PDE}}(\vec{g}_*,\vec{u}_*)=0$.
\end{proof}

\section{First Order Optimality Conditions}

The first order necessary optimality conditions for a minimizer 
$(\gfcn^*,\hfcn^*,\lfcn^*),(\pbar^*,\ptil^*,\wtil^*)$ can formally be obtained by setting all partial derivatives of the Lagrange function
\[
\begin{aligned}
&\mathcal{L} \Bigl(\check{\gfcn},\hfcn,\lfcn),(\check{\pbar},\check{\ptil},\wtil),(\check{\qbar},\check{\qtil},\vtil, \check{\mu}_{N}, \mu_{pl} )\Bigr)\\
&=J((\check{\gfcn},\hfcn,\lfcn),(\check{\pbar},\check{\ptil},\wtil))
+\langle A_{\textup{PDE}}((\check{\gfcn},\hfcn,\lfcn),(\check{\pbar},\check{\ptil},\wtil)),(\qbar,\qtil,\vtil,\check{\mu}_{N}, \mu_{pl} )\rangle
%+\langle A_{\textup{CMP}}(\check{\gfcn},\hfcn,\lfcn),(\check{\lambda}_0,\check{\lambda}_1,\check{\lambda}_2)\rangle
\end{aligned}
\]
to zero.\\
Differentiation with respect to $(\qbar,\qtil,\vtil, \mu_{N}, \mu_{pl} )$ gives the weak form of the state equation \eqref{eqn:var_L2L2}.\\
Differentiation with respect to $\pbar$ 
%(recalling that $\alpha(\pbar,\ptil)=\alphp$) 
yields the weak form of the adjoint equation for $\qbar$:
\begin{equation}\label{adj_qbar}
\begin{aligned}
&\qbar\in 
L^{2}(\zeroT L^2({\Omega(\lfcn)})),
%\quad \qbar(T)=0 
%\ \text{ (cf. the definition of $Z_T$) } \
\text{ and }\\
&0=\int_0^T\Bigl\{\int_{\Omega(\lfcn)}\chi_{\textup{ROI}}(\pbar+\ptil-p_{\textup{d}})\phibar \, dx\\
&\quad
+\int_{\Omega(\lfcn)} \Bigl( ( \phibar_{tt} 
- c^2\Delta \phibar - b\Delta \phibar_t)\, \qbar\Bigr)\, dx
%+\int_{\Gamma_a}(c\phibar_t+\tfrac{b}{c}\phibar_{tt})\,\qbar\, dS
\\
&\qquad
+ \int_{\Omega(\lfcn)}  (- 2k\phibar \ptil_{tt} -4 k (\pbar+\ptil)_{t} \phibar_{t} -2k\phibar \pbar_{tt} -2k(\pbar+\ptil)\phibar_{tt}) \qtil) \, dx
-\rho\int_{\Gampl}  \phibar_{t}\, \vtil \,dS
\Bigr\}dt
\\
&\qquad
+\int_{\Gamma_N}
%[(\bar{\lambda}+b)\partial_t+c^2]
\Dnu{\phibar} \,\mu_{N}\, dS
%-\int_{\Gampl} 2k \tilde{\lambda}  \phibar \, ( \Dnu{\ptil}_{t} + \rho \wtil_{tt}) \,\mu_{pl}\,dS
\\
&\qquad
\text{for all }\phibar\in 
W^{2,\infty}(\zeroT L^2({\Omega(\lfcn)}))\cap 
%H^1(\zeroT H^2_{\Delta,0}({\Omega(\lfcn)}))\cap L^2(\zeroT H^2_{\Delta,1}({\Omega(\lfcn)}))
H^1(\zeroT H^2_{\Delta,1}({\Omega(\lfcn)}))\\
&\qquad\phantom{\text{for all }\phibar\in}\text{s.t. }
\phibar(0)=0, \phibar_t(0)=0
%\ \text{ (cf. the definition of $U_0$.) }
\end{aligned}
\end{equation}
%with
%$f_0(\pbar,\ptil)=\frac{c^2\Delta \ptil +b\Delta \ptil_t
%+2k((\pbar+\ptil)_t)^2+2k(\pbar+\ptil)\pbar_{tt}}{\alpha(\pbar,\ptil)^2}$.
%
Differentiation with respect to $(\ptil,\wtil)$ yields the weak form of the adjoint equation for $(\qtil,\vtil, \mu_N, \mu_{pl})$: 
%(which should again be some coupled wave-plate system):
\begin{equation}\label{adj_qtilvtil}
\begin{aligned}
&(\qtil,\vtil, \mu_{N}, \mu_{pl} )\in 
L^{2}(\zeroT L^2({\Omega(\lfcn)}))\times L^{2}(\zeroT H^{2}_{\diamondsuit}(\Gampl))
\times L^{2}(\zeroT L^{2}(\Gamma_{N})) \times L^{2}(\zeroT L^{2}(\Gampl)) 
\\
&0=\int_0^T\Bigl\{\int_{\Omega(\lfcn)}\chi_{\textup{ROI}}(\pbar+\ptil-p_{\textup{d}})\phitil \, dx\\
&\quad+\int_{\Omega(\lfcn)} 
 \Bigl(
 (1-2k(\pbar+\ptil)) \phitil_{tt} 
  -2k  \phitil \ptil_{tt} 
 -  c^2\Delta \phitil
 - b\Delta \phitil_t
-4k(\pbar+\ptil)_t \phitil_t
-2k\phitil \pbar_{tt})
\Bigr)
\qtil
\, dx
\\
&\qquad
%+\int_{\Gamma_{a}} c\partial_\nu \phitil_t\, \partial_\nu \qtil\, dS
 +\tfrac{\rho}{\kappa}\int_{\Gampl} 
 	 \Bigl( (\rho\psitil_{tt} 
     \Damppl{+ \beta(-\Delpl)^{\gamma} \psitil_{t} }
	 		- \kappa \phitil_{t} ) \vtil
 +\delta\Delpl\psitil\, \Delpl\vtil  \Bigr) \,dS
\\
&\qquad
+\int_{\Gampl}
%[(\tilde{\lambda}(1-2k(\pbar+\ptil))+b)\partial_t+c^2]
(\Dnu{\phitil} + \rho \psitil_t)\,\mu_{pl}\,dS
%- \int_{\Gampl}2k\tilde{\lambda} \phitil ( \Dnu{\ptil}_{t} + \rho \wtil_{tt}) \,\mu_{pl}\,dS
\Bigr\}\, dt
\\
&\text{for all }(\phitil,\psitil)\in 
W^{2,\infty}(\zeroT L^2({\Omega(\lfcn)}))\cap H^1(\zeroT H_{\Delta,0}^2({\Omega(\lfcn)})) \cap L^2(\zeroT H_{\Delta,1}^2({\Omega(\lfcn)}))
\\
&\phantom{\text{for all }(\phitil,\psitil)\in}
 \times (W^{1,\infty}(\zeroT L^2(\Gampl))\cap L^\infty(\zeroT H^2_\diamondsuit(\Gampl))),\\
&\phantom{\text{for all }(\phitil,\psitil)\in}
\text{s.t. }
% \Dnu{\phitil} =-\rho\psitil\text{ on }\Gampl,\ \Dnu{\phitil}=0\mbox{ on }\Gamma_{N}, \quad
 \phitil(0)=0, \ \psitil(0)=0, \ 
  \phitil_t(0)=0, \ \psitil_t(0)=0.
%\ \text{ (cf. the definition of $U_0$). }
\end{aligned}
\end{equation}
%\footnote{The strong form of $\qtil$ seems to contain terms of the type $\qtil$, $-\Delta\qtil$, $(-\Delta)^2\qtil$. This happens even if we only use $\hfcn$ as a control and set $\gfcn=0$, $\pbar=0$, due to the nonlinearity. (This refers to a previous version.)} 
%
Note that \eqref{adj_qbar} and \eqref{adj_qtilvtil} are coupled by the terms 
$-2k(\pbar+\ptil)\phibar_{tt}) \qtil$, $-\rho\int_{\Gampl}  \phibar_{t}\, \vtil$.

\medskip

Existence of variational solutions 
$(\qbar,\qtil,\vtil)\in Z$ to \eqref{adj_qbar}, \eqref{adj_qtilvtil} follows from theory for first order optimality conditions, see Theorem~\ref{thm:optcond} below.
Differentiating with respect to $\check{\gfcn}$ and $\hfcn$, 
we end up with two of the three gradient equations
\begin{equation} \label{gradonetwo}
\begin{aligned}
&0=\int_{\Gamref}
%\Gamma_N(\lfcn)
\Bigl(\int_0^T\uld{\gfcn}\bigl(\regpar 
((\regop^{\text{time}}_\gfcn)^* \regop^{\text{time}}_\gfcn+(\regop^{\text{space}}_\gfcn)^* \regop^{\text{space}}_\gfcn) (\check{\gfcn}-\check{\gfcn}_0) -
%[(\bar{\lambda}+b)\partial_t-c^2] 
\mu_{N} \bigr)\,dt 
%-(\bar{\lambda}+b)\uld{\gfcn}(0)\check{\mu}_{N}(0) 
\Bigr) dS
%+ \uld{\gfcn}(0)\lambda_0+ \uld{\gfcn}_t(0)\lambda_1+ \uld{\gfcn}_{tt}(0)\lambda_2\Bigr)\, dS
\text{ for all }\uld{\gfcn} \in W^{1,\infty}(\zeroT H^{s_\gfcn}(\Gamref))\\
&0=\int_{\Gampl}\int_0^T\uld{\hfcn}\Bigl(\regpar \regop_\hfcn^* \regop_\hfcn (\hfcn-\hfcn_0) - \tfrac{\rho}{\kappa}\vtil\Bigr)\,dt \, dS
\text{ for all }\uld{\hfcn} \in L^2(\zeroT L^2(\Gampl))
\end{aligned}
\end{equation}
The third gradient equation is obtained by differentiation with respect to $\lfcn$,
using the identities
\[
\begin{aligned}
&\frac{d}{d\lfcn}\left[\int_{\Omega(\lfcn)} \varphi(x)\, dx\right]\uld{\lfcn}
=\frac{d}{d\lfcn}\left[\int_{\Gamflat} \int_0^{\lfcn(x')} \varphi(x',z)\, dx'\, dz\right]\uld{\lfcn}
=\int_{\Gamflat} \varphi(x',\lfcn(x'))\, \uld{\lfcn}\, dx'
=\int_{\Gamma_N(\lfcn)} \varphi \, \widehat{\left(\frac{\uld{\lfcn}}{\sigma_{\lfcn}}\right)}\, dS
\end{aligned}
\]
(with $\widehat{f}(x',z)=f(x')$ for some $f$ depending only on $x'$)
and, applying integration by parts with $\uld{\lfcn}=0$ on $\partial\Gamflat$, as well as 
\[
\begin{aligned}
&\frac{d}{d\lfcn}\left[\int_{\Gamma_N(\lfcn)} \varphi\, dS\right]\uld{\lfcn}
=\frac{d}{d\lfcn}\left[\int_{\Gamflat} \varphi(x',\lfcn(x'))\sqrt{|\nabla_{x'}\lfcn(x')|^2+1}\, dx'\right]\uld{\lfcn}\\
&=\int_{\Gamflat} \Bigl(
\varphi_z(x',\lfcn(x'))\uld{\lfcn}\sqrt{|\nabla_{x'}\lfcn(x')|^2+1}
+\varphi(x',\lfcn(x'))\tfrac{1}{\sqrt{|\nabla_{x'}\lfcn(x')|^2+1}} \nabla_{x'}\lfcn(x')\cdot \nabla_{x'}\uld{\lfcn}(x')
\Bigr)\, dx'\\
&=\int_{\Gamflat} \Bigl(
\varphi_z(x',\lfcn(x'))\sqrt{|\nabla_{x'}\lfcn(x')|^2+1}
- \frac{d}{d x'}
%\nabla_{x'}
\Bigl[\varphi(x',\lfcn(x'))\tfrac{1}{\sqrt{|\nabla_{x'}\lfcn(x')|^2+1}} \nabla_{x'}\lfcn(x')\Bigr]\Bigr)\,\uld{\lfcn}(x')
\, dx'\\
&=\int_{\Gamflat} \Bigl(
\bigl(-\nabla_{x'}\varphi(x',\lfcn(x'))+\varphi_z(x',\lfcn(x'))\bigr)\tfrac{1}{\sqrt{|\nabla_{x'}\lfcn(x')|^2+1}}
-\varphi(x',\lfcn(x'))
\frac{d}{d x'}
%\nabla_{x'}
\Bigl[\tfrac{1}{\sqrt{|\nabla_{x'}\lfcn(x')|^2+1}} \nabla_{x'}\lfcn(x')\Bigr]\Bigr)\,\uld{\lfcn}(x')
\, dx'\\
&=\int_{\Gamma_N(\lfcn)}\bigl(\partial_{\nu_{\lfcn}} \varphi + \varphi\, \widehat{H}_\lfcn\bigr)\, \widehat{\left(\frac{\uld{\lfcn}}{\sigma_{\lfcn}}\right)}\, dS 
\quad \text{ for all }\uld{\lfcn} \in H^{s_\lfcn}(\Gamflat))
\end{aligned}
\]
with 
\[
\nu_\lfcn(x')=\frac{1}{\sigma_\lfcn(x')}\left(\begin{array}{c}-\nabla_{x'}\lfcn(x')\\1\end{array}\right), \qquad 
\sigma_\lfcn(x')=\sqrt{|\nabla_{x'}\lfcn(x')|^2+1}, \qquad
H_\lfcn(x')=-\sigma_{\lfcn(x')} 
%\nabla_{x'}
\frac{d}{d x'}
\Bigl[\tfrac{1}{\sigma_{\lfcn(x')}} \nabla_{x'}\lfcn(x')\Bigr]
\]
(where $H_\lfcn$ is the mean curvature and the total derivative $\frac{d}{d x'}$ is to be understood as a divergence) as
\begin{equation}\label{gradthree}
\begin{aligned}
0=&
\int_{\Gamma_N(\lfcn)} 
\Bigl\{
\int_0^T \Bigl(
\bigl(\pbar_{tt}  - c^{2}\Delta \pbar - b \Delta \pbar_{t}\bigr)\,\qbar\\
&\qquad\qquad+\bigl((1-2k(\pbar+\ptil))\, \ptil_{tt}  - c^{2}\Delta \ptil - b \Delta \ptil_{t}  
-2k((\pbar+\ptil)_t)^2 -2k(\pbar+\ptil)\pbar_{tt}\bigr)\,\qtil\Bigr)\, dt\, \\
&\qquad\qquad
+\regpar \widehat{\regop_\lfcn^*\regop_\lfcn\lfcn}
+ \bigl(\partial_{\nu_{\lfcn}} \varphi + \varphi\, \widehat{H}_\lfcn\bigr)\, 
\Bigr\}
\widehat{\left(\frac{\uld{\lfcn}}{\sigma_{\lfcn}}\right)}\, dS
\quad \text{ for all }\uld{\lfcn} \in H^{s_\lfcn}(\Gamflat))
\\
&\text{with }\varphi=
\int_0^T
%[(\bar{\lambda}+b) \partial_t+c^2]
(\Dnu{\pbar}-\gfcn)\,\mu_{N} \,dt
%\\&\qquad\qquad+(\gfcn(0)-\Dnu{p_0})\lambda_0 
%+(\gfcn_{t}(0)-\Dnu{p_1})\lambda_1+(\gfcn_{tt}(0)-\Dnu{(\tfrac{1}{1-2p_0}(c^2\Delta p_{0}+b\Delta p_1}\,\lambda_2,
\end{aligned}
\end{equation}
(see also \cite[Theorem 8]{BK-Peichl}).

Strong forms of these first order optimality conditions are listed in the appendix.

\medskip

For a full justification of the first order necessary optimality conditions, we invoke \cite[Theorem 6.3]{Troeltzsch2010}. To this end, we need to show 
\begin{enumerate}
\item Fr\'{e}chet differentiability of $J$ and $A_{\textup{PDE}}$; For $J$, this is obvious; for $A_{\textup{PDE}}$, we will prove
\[
\|A_{\textup{PDE}}(\vec{g}^*+\uld{ \vec{g}},\vec{u}^*+\uld{\vec{u}})-A_{\textup{PDE}}(\vec{g}^*,\vec{u}^*)-A_{\textup{PDE}}'(\vec{g}^*,\vec{u}^*)(\uld{ \vec{g}}, \uld{\vec{u}})\|_{Z^*} 
=o(\|(\uld{ \vec{g}}, \uld{\vec{u}})\|_{G\times U}),
\]
in Subsection~\ref{subsec:diff}.
\item Well-definedness and a certain surjectivity of the linearized equality constraint operator\\ $A_{\textup{PDE}}'((\gfcn^*,\hfcn^*,\lfcn),(\pbar^*,\ptil^*,\wtil^*))$
%\[A'(\vec{g}^*,\vec{u}^*)=\Bigl(A_{\textup{PDE}}'((\gfcn^*,\hfcn^*,\lfcn),(\pbar^*,\ptil^*,\wtil^*)),\, A_{\textup{CMP}}'(\gfcn^*,\hfcn^*,\lfcn)\Bigr)\] 
(more precisely, the Zowe - Kurcyusz condition \cite[Eq. (6.15)]{Troeltzsch2010}, that here reads as 
%$A'(\vec{g}^*,\vec{u}^*)(\mathbb{R}^+(G_{\textup{ad}}-\{\vec{g}^*\})\times \mathbb{R}^+(U_{\textup{ad}}-\{\vec{u}^*\}))=Z^*\times\Lambda^*$).\\
$A_{\textup{PDE}}'(\vec{g}^*,\vec{u}^*)(\mathbb{R}^+(G_{\textup{ad}}-\{\vec{g}^*\})\times \mathbb{R}^+(U_{\textup{ad}}-\{\vec{u}^*\}))=Z^*$).\\
We avoid inequality constraints by assuming the admissible set
$G_{\textup{ad}}\times U_{\textup{ad}}$ to be an affinely linear space 
$\{(\gfcn_0,\hfcn_0,\lref )\}+G_0\times U_0$ and 
putting all bounds needed for the analysis into the penalty term $\tfrac{\regpar}{2}\|(\gfcn-\gfcn_0,\hfcn-\hfcn_0,\lfcn-\lref)\|_{G}^2$; 
then the 
Zowe - Kurcyusz condition simplifies to 
%$A'(\vec{g}^*,\vec{u}^*)(G\times U)=Z^*\times\Lambda^*$
$A_{\textup{PDE}}'(\vec{g}^*,\vec{u}^*)(G\times U)=Z^*$,
that is, 
\begin{equation}\label{ZoweKurcyusz}
\text{For any }\vec{f}\in Z^* \text{ there exists }
(\uld{\vec{g}}, \uld{\vec{u}})\in {G\times U} \text{ such that }
A_{\textup{PDE}}'(\vec{g}^*,\vec{u}^*)(\uld{ \vec{g}}, \uld{\vec{u}}) =\vec{f},
\end{equation}
which will be verified in Subsection~\ref{subsec:surj}.
\end{enumerate}

\begin{Theorem}\label{thm:optcond}
Let the spaces $G_{ad}$, $U_{ad}=U(\Omref)$ be  as in \eqref{fcnsp}, \eqref{defG0}, \eqref{G}, \eqref{highreg} and let \eqref{ass_l0} be satisfied. Then for any local minimizer $(\vec{g}_*,\vec{u}_*) \in G_{ad} \times U_{ad} $  to \eqref{opt}, there exists an adjoint state $\vec{z}\in Z$, whose components satisfy \eqref{adj_qbar}, \eqref{adj_qtilvtil}, \eqref{gradonetwo}, \eqref{gradthree}.
\end{Theorem}

\subsection{Fr\'{e}chet differentiability}\label{subsec:diff}
Here we are assuming higher regularity \eqref{highreg} of $\gfcn$,  $\lfcn$ and of the boundary overall (see the comment after $\eqref{highreg}$, so that in the definition of $U$, cf. \eqref{fcnsp} with $s=\frac12$, we have full $H^2(\Omega)$ spaces, that is, we may skip the subscripts $\Delta$.
%\ntBK{Can we also do so in $H^2_{\Delta,0}$ (which would be needed for the  $D_{\lfcn}^{2}\check{\pbar}_{t}$ term in $I'_{2}$ below)? Yes, we actually don't need $H^2_{\Delta,0}$ and can use $H^2_{\Delta,1}$ throughout, see the energy estimate \eqref{enest_pp} below.}

\begin{Proposition}\label{prop:Frechet}
Let the spaces $G_{ad}$, $U_{ad}=U(\Omref)$ be  as in \eqref{fcnsp}, \eqref{defG0}, \eqref{G}, \eqref{highreg}.
Under assumption \eqref{ass_l0}, the map $A_{\textup{PDE}}$ defined in \eqref{eqn:APDE}  is Fr\'{e}chet  differentiable as a mapping $G_{ad} \times U(\Omref)  \to Z(\Omref)^*$ (cf. \eqref{fcnsp}, \eqref{G}) at any $(\vec{g},\vec{u})\in G_{ad} \times U(\Omref)$ such that $\lfcn$ satisfies \eqref{ass_diffl} (cf. \eqref{bdd_lfcn}).
\end{Proposition}
\begin{proof}
Since $\wtil$, $\check{g}$, $h$ enter the definition of $A_{\textup{PDE}}$ linearly, we only have to estimate the Taylor remainders arising due to nonlinearity in $\pbar$, $\ptil$, $\lfcn$.

We first estimate the contributions over the reference domain $\Omref= \Omega_{\text{fix}}\cup \Omega_{\text{var}} (\lref)$, that is, each of the following terms $I_1$ -- $I_4$, in $L^{2}(\zeroT L^2(\Omref))$:
\begin{align*}
I_{1} = &  \frac{1}{ \omegazero_{\lfcn+ \uld{\lfcn} }} ( \check{\pbar}_{tt} + \uld{\check{\pbar}}_{tt} ) - \frac{1}{ \omegazero_{\lfcn }}  \check{\pbar}_{tt} - \frac{1}{ \omegazero_{\lfcn }}  \uld{\check{\pbar}}_{tt}  - \frac{ \uld{\lfcn}}{\lref} \check{\pbar}_{tt} \\
I_{2}  = & \frac{1}{ \omegazero_{\lfcn+ \uld{\lfcn} }} D^{2}_{\lfcn+ \uld{\lfcn}} ( \check{\pbar}+ \uld{ \check{\pbar}} ) 
			- \frac{1}{ \omegazero_{\lfcn }}  D^{2}_{\lfcn}  \check{\pbar}
			- \frac{1}{ \omegazero_{\lfcn }}  D^{2}_{\lfcn} \uld{ \check{\pbar}} 
				- \frac{1}{\omegazero_{\uld{\lfcn}}}  \Delta_{x'} \check{\pbar} 
%			+ 2 \cdot \nabla_{x'}(\tfrac{\lfcn}{\lref }) \check{z} \, \nabla_{x'}  \partial_{ \check{z}} \uld{ \check{\pbar} }
				+ 2 \nabla_{x'}(\tfrac{\uld{\lfcn}}{\lref }) \check{z} \cdot \ \nabla_{x'}    \partial_{ \check{z}} \check{\pbar} 
			\\
%			 +   (\tfrac{\lref }{\lfcn})  \check{z}^2 |\nabla_{x'}(\tfrac{\lfcn}{\lref })|^2 \partial^{2}_{ \check{z}}  \check{\uld{\pbar}}
			& -    2 (\tfrac{\lref }{\lfcn}) \check{z}^2 \nabla_{x'}(\tfrac{\uld{\lfcn}}{\lref } )\cdot \nabla_{x'}(\tfrac{\lfcn}{\lref }) \partial^{2}_{ \check{z}}  \check{\pbar}	
			+	 (\tfrac{\lref   \uld{\lfcn}}{\lfcn^{2}}) \, \check{z}^2 |\nabla_{x'}(\tfrac{\lfcn}{\lref })|^2 \partial^{2}_{ \check{z}}  \check{\pbar}
\\
		&
%		+  (\tfrac{\lref }{\lfcn})\check{z}\Delta_{x'}(\tfrac{\lfcn}{\lref })  \partial_{\check{z}} \check{\uld{\pbar}}
 		+  \check{z} \Delta_{x'}(\tfrac{\uld{\lfcn}}{\lref })  \partial_{\check{z}} \check{\pbar}			
				+ \frac{\lref \uld{\lfcn} }{\lfcn^{2}}   \partial_{ \check{z}}^{2}  \check{\pbar} 
\\						
I'_{2} = &  \frac{d}{dt} I_{2}
%\frac{1}{ \omegazero_{\lfcn+ \uld{\lfcn} }} D^{2}_{l+ \uld{\lfcn}} ( \check{\pbar}_{t}+ \uld{ \check{\pbar}_{t}} ) 
%			- \frac{1}{ \omegazero_{\lfcn }}  D^{2}_{\lfcn}  \check{\pbar}_{t}§§§§-
%			- \frac{1}{ \omegazero_{\lfcn }}  D^{2}_{\lfcn} \uld{ \check{\pbar}_{t}} §****************
%				- \frac{1}{\omegazero_{\uld{\lfcn}}}  \Delta_{H} \check{\pbar}_{t} 
%				+ \frac{l_{0} \uld{\lfcn} }{l^{2}}   \partial_{z}^{2}  \check{\pbar}_{t}
\\
\tilde{I}_{2} =& \text{ as $I_2$ with  $\ptil$ in place of $\pbar$}
\\						
\tilde{I}'_{2} = &  \frac{d}{dt} \tilde{I}_{2}
\\
I_{3}  = &     \frac{1}{ \omegazero_{\lfcn+ \uld{\lfcn} }} 
		(1- 2k (\check{\pbar} + \uld{ \check{\pbar}}+ \check{\ptil}+ \uld{\check{\ptil}} ))
		( \check{\pbar}_{tt} + \uld{ \check{\pbar}}_{tt}+\check{\ptil}_{tt} + \uld{\check{\ptil}}_{tt} ) 
		- \frac{1}{ \omegazero_{\lfcn }} 
		(1- 2k ( \check{\pbar} + \check{\ptil}))
		(\check{\pbar}_{tt}+\check{\ptil}_{tt})   
		\\
		& \indeq - \frac{\uld{\lfcn}}{\lref} (1- 2k (\check{\pbar} + \check{\ptil})) (\check{\pbar}_{tt}+\check{\ptil}_{tt})
		+ \frac{1}{ \omegazero_{\lfcn}}  
			 2k ( \uld{\check{\pbar}}+ \uld{\check{\ptil}} )(\check{\pbar}_{tt}+\check{\ptil}_{tt})
		- \frac{1}{ \omegazero_{\lfcn}}  
			 (1- 2k (\check{\pbar} + \check{\ptil}))
		 (\uld{\check{\pbar}}_{tt}+\uld{\check{\ptil}}_{tt})	\\
I_{4} = & \frac{1}{ \omegazero_{\lfcn+ \uld{\lfcn} }} ( \check{\ptil}_{t} +\uld{\check{\ptil}_{t}}+\check{\pbar}_{t}+\uld{\check{\pbar}_{t}})^{2}
 -\frac{1}{ \omegazero_{\lfcn}} ( \check{\ptil}_{t} +\check{\pbar}_{t})^{2}
 - \frac{\uld{\lfcn}}{\lref}  ( \check{\ptil}_{t} +\check{\pbar}_{t})^{2}
 -\frac{1}{ \omegazero_{\lfcn}} 2( \check{\ptil}_{t} +\check{\pbar}_{t})(\uld{\check{\ptil}_{t}}+\uld{\check{\pbar}_{t}})	
\end{align*}

The terms  $I_{1}$, $I_{2}$, $I'_2$, $\tilde{I}_{2}$, $\tilde{I}'_2$ restricted on  $\Omega_{\text{fix}}$ are linear and thus vanish, so one only needs to compute the norm on  $\Omega_{\text{var}}(\lref)$, where 
$\omegazero_{\lfcn }=\frac{\lref}{\lfcn}$ and thus 
$\frac{1}{ \omegazero_{\lfcn+ \uld{\lfcn} }}-\frac{1}{ \omegazero_{\lfcn }}=-\frac{\uld{\lfcn}}{\lref}$. We then estimate

\begin{align*}
\Vert I_{1} \Vert_{L^{2}(\zeroT L^2(\Omref))}&= \left\Vert \frac{1}{ \omegazero_{\lfcn+ \uld{\lfcn} }} ( \check{\pbar}_{tt} + \uld{\check{\pbar}}_{tt} ) - \frac{1}{ \omegazero_{\lfcn }}  \check{\pbar}_{tt} - \frac{1}{ \omegazero_{\lfcn }}  \uld{\check{\pbar}}_{tt}  - \frac{ \uld{\lfcn}}{\lref} \check{\pbar}_{tt} \right\Vert_{L^{2}(L^2)} 
= \left\Vert \frac{ \uld{\lfcn}}{\lref}  \uld{\check{\pbar}}_{tt}  \right\Vert_{L^{2}(L^2)} \\
&\leq  C(\Vert \lref \Vert_{L^{\infty}}) \Vert  \uld{\lfcn} \Vert_{L^{\infty}(B)} \Vert \uld{\check{\pbar}}_{tt} \Vert_{L^{2}(L^2)} 
\leq C(\Vert \lref \Vert_{L^{\infty}}) \Vert  \uld{\lfcn} \Vert_{L^{\infty}(B)} \Vert \uld{\vec{u}} \Vert_{U} 
= o(\|(\uld{ \vec{g}}, \uld{\vec{u}})\|_{G\times U}),
\end{align*}
where we denote by $C(\Vert \lref \Vert_{L^{\infty}})$ a constant depending on $\Vert \lref \Vert_{L^{\infty}}$.
Similarly, simplifying the expression for $I_{2}$, we have
\begin{align*}
I_{2} & = \frac{ \uld{\lfcn}}{\lref} \Delta_{x'}     \uld{\check{\pbar}}
		-  2 \nabla_{x'}(\tfrac{\uld{\lfcn}}{\lref }) \check{z} \nabla_{x'} \partial_{\check{z}}   \uld{\check{\pbar}}
		- \Delta_{x'}(\tfrac{\uld{\lfcn}}{\lref })\check{z} \partial_{\check{z}}    \uld{\check{\pbar}}
				- \left( \frac{\lref}{ \lfcn + \uld{\lfcn}}  \left\vert \nabla \left( \frac{\lfcn +\uld{\lfcn} }{\lref}\right)\right\vert^{2} 
				- \frac{\lref}{\lfcn}  \left \vert \nabla \left( \frac{\lfcn } {\lref}\right) \right\vert^{2} 
		\right) \check{z}^{2} \partial^{2}_{\check{z}} (\check{\pbar} +\uld{\check{\pbar}}) \\
		&  \indeq -    2 (\tfrac{\lref }{\lfcn}) \check{z}^2 \nabla_{x'}(\tfrac{\uld{\lfcn}}{\lref } )\cdot \nabla_{x'}(\tfrac{\lfcn}{\lref }) \partial^{2}_{ \check{z}}  \check{\pbar}	
			+	 (\tfrac{\lref   \uld{\lfcn}}{\lfcn^{2}}) \, \check{z}^2 |\nabla_{x'}(\tfrac{\lfcn}{\lref })|^2 \partial^{2}_{ \check{z}}  \check{\pbar}
	-\frac{\lref \uld{\lfcn}}{(\lfcn+ \uld{\lfcn}) \lfcn}  \partial^{2}_{\check{z}}  \uld{\check{\pbar}}
	+\frac{\lref( \uld{\lfcn})^2}{(\lfcn+ \uld{\lfcn}) \lfcn^2}  \partial^{2}_{\check{z}} \check{\pbar}
\\
 & = \frac{ \uld{\lfcn}}{\lref} \Delta_{x'}     \uld{\check{\pbar}}
		-  2 \nabla_{x'}(\tfrac{\uld{\lfcn}}{\lref }) \check{z} \nabla_{x'} \partial_{\check{z}}   \uld{\check{\pbar}}
		- \Delta_{x'}(\tfrac{\uld{\lfcn}}{\lref })\check{z} \partial_{\check{z}}    \uld{\check{\pbar}}
			- \frac{\lref \uld{\lfcn}}{(\lfcn+ \uld{\lfcn}) \lfcn} \left\vert \nabla \left( \frac{\lfcn}{\lref}\right)\right\vert^{2}  \check{z}^{2} \partial^{2}_{\check{z}} \uld{\check{\pbar}}
		+ \frac{\lref( \uld{\lfcn})^2}{(\lfcn+ \uld{\lfcn}) \lfcn^2} \left\vert \nabla \left( \frac{\lfcn}{\lref}\right)\right\vert^{2}  \check{z}^{2} \partial^{2}_{\check{z}} \check{\pbar}
		\\
		& \indeq 
		+2  \frac{\lref \uld{\lfcn}}{(\lfcn+ \uld{\lfcn}) \lfcn}  \nabla \left( \frac{\lfcn}{\lref}\right) \cdot   \nabla \left( \frac{\uld{\lfcn}}{\lref} \right) \check{z}^{2} \partial^{2}_{\check{z}} \check{\pbar}
		+ \frac{\lref }{(\lfcn+ \uld{\lfcn})}  \left\vert \nabla \left( \frac{\uld{\lfcn}}{\lref}\right)\right\vert^{2}  \check{z}^{2} \partial^{2}_{\check{z}} ( \check{\pbar} +\uld{\check{\pbar}} )
		%		& -    2 (\tfrac{\lref }{\lfcn}) \check{z}^2 \nabla_{x'}(\tfrac{\uld{\lfcn}}{\lref } )\cdot \nabla_{x'}(\tfrac{\lfcn}{\lref }) \partial^{2}_{ \check{z}}  \check{\pbar}	
%			+	 (\tfrac{\lref   \uld{\lfcn}}{\lfcn^{2}}) \, \check{z}^2 |\nabla_{x'}(\tfrac{\lfcn}{\lref })|^2 \partial^{2}_{ \check{z}}  \check{\pbar}
\\
& 
	\indeq 
	+  2\frac{\lref }{(\lfcn+ \uld{\lfcn})}	 \nabla \left( \frac{\lfcn}{\lref}\right)\cdot \nabla \left( \frac{\uld{\lfcn}}{\lref}\right)\	 \check{z}^{2} \partial^{2}_{\check{z}} \uld{\check{\pbar}}
	-\frac{\lref \uld{\lfcn}}{(\lfcn+ \uld{\lfcn}) \lfcn}  \partial^{2}_{\check{z}}  \uld{\check{\pbar}}
	+\frac{\lref( \uld{\lfcn})^2}{(\lfcn+ \uld{\lfcn}) \lfcn^2 }  \partial^{2}_{\check{z}} \check{\pbar}
%
% \frac{1}{ \omegazero_{\lfcn+ \uld{\lfcn} }} D^{2}_{\lfcn+ \uld{\lfcn}} ( \check{\pbar} + \uld{\check{\pbar}} ) 
%			- \frac{1}{ \omegazero_{\lfcn }}  D^{2}_{\lfcn} \check{\pbar} 
%			- \frac{1}{ \omegazero_{\lfcn }}  D^{2}_{\lfcn} \uld{\check{\pbar}} 
%				- \frac{1}{\omegazero_{\uld{\lfcn}}}  \Delta_{H} \check{\pbar} 
%				+ \frac{\lref \uld{\lfcn} }{\lfcn^{2}}   \partial_{z}^{2} \check{\pbar}
%					\\
%					& =	
%					\frac{\uld{\lfcn}}{\lref} \Delta_{H} \uld{\check{\pbar}}
%					+ \frac{\lref (\uld{\lfcn})^{2} }{\lfcn^{2} (\lfcn+ \uld{\lfcn}) }   \partial_{z}^{2} \check{\pbar}
%					- \frac{\lref \uld{\lfcn}}{(\lfcn+ \uld{\lfcn}) \lfcn}  \partial_{z}^{2} \uld{\check{\pbar}}
\end{align*}
Hence, assuming that the term $\uld{\lfcn}$ is sufficiently small in $L^{\infty}$, that is, using 
%the bounded below assumption on $\lfcn$ ($\Vert \lfcn \Vert_{L^{\infty}} \geq \frac{1}{2} \Vert \lref \Vert_{L^{\infty}}  $), 
%\eqref{ass_l0}, 
\eqref{ass_diffl}, we have the estimate 
\begin{align*}
\Vert I_{2} \Vert_{L^{2}}
%& \leq 
%	C_{ \lfcn, \lref} ( \Vert \uld{\lfcn} \Vert_{W^{1,\infty}(B)} \Vert \uld{\check{\pbar}} \Vert_{H^{2}}
%+
%\Vert \uld{\lfcn} \Vert_{W^{2,3}(B)} \Vert \uld{\check{\pbar}} \Vert_{L^{6}}  
%  +  \Vert \uld{\lfcn} \Vert^{2}_{L^{\infty}} 
%				\Vert \check{\pbar} \Vert_{H^{2}}
%			+ \Vert \uld{\lfcn} \Vert_{L^{\infty}} 
%				\Vert \uld{\check{\pbar}} \Vert_{H^{2}})
%\\
& \leq C( \Vert \lfcn \Vert_{W^{1,\infty}(B)} , \Vert \lref \Vert_{W^{1,\infty}(B)}, \Vert \tfrac{1}{\lref} \Vert_{L^{\infty}(B)} , \Vert \check{\pbar}  \Vert_{L^2(H^{2})})  (\Vert \uld{\lfcn} \Vert^{2}_{W^{1,\infty}(B)} +  \Vert \uld{\lfcn} \Vert^{4}_{W^{1,\infty}(B)}+  \Vert \uld{\check{\pbar}} \Vert^{2}_{L^2(H^{2})} )\\
& =o(\|(\uld{ \vec{g}}, \uld{\vec{u}})\|_{G\times U}).
\end{align*}
A similar estimate holds for $I'_{2}$  involving a $D_{\lfcn}^{2}\check{\pbar}_{t}$ term, 
and analogously for $\tilde{I}_{2}$, $\tilde{I}'_{2}$.
\\
As for $I_{3}$ and $I_{4}$, which involve $\check{\pbar}$ and $\check{\ptil}$ terms, we first rewrite
\begin{align*}
I_{3} = &   
	 -\frac{\uld{\lfcn}}{\lref} 2k (\uld{\check{\pbar}}+\uld{\check{\ptil}}) (\check{\pbar}_{tt}+\check{\ptil}_{tt})
			+ \frac{\uld{\lfcn}}{\lref} (1-2k(\check{\pbar}+ \check{\ptil}))
            (\uld{\check{\pbar}}_{tt}+\uld{\check{\ptil}}_{tt})
			+\frac{\lfcn+\uld{\lfcn}}{\lref} (1- 2k (\uld{\check{\pbar}}+\uld{\check{\ptil}}))
            (\uld{\check{\pbar}}_{tt}+\uld{\check{\ptil}}_{tt}),
	\end{align*}
for which the space time $L^{2}$ norm can be similarly estimated using the continuous embedding of $H^{2}$ into $L^{\infty}$ to obtain 
\begin{align*}
\Vert I_{3} \Vert_{L^{2}} & \leq C(\Vert \check{\pbar}_{tt} \Vert_{L^{2}},\Vert \check{\ptil}_{tt} \Vert_{L^{2}}, 
				\Vert \check{\ptil} \Vert_{L^2(H^{2})},
				\Vert \check{\pbar} \Vert_{L^2(H^{2})}
				,\Vert \lfcn \Vert_{L^{\infty}(B)}, \Vert \tfrac{1}{\lref} \Vert_{L^{\infty}(B)} ) \\
  &\qquad \cdot P_{3}(\Vert \uld{\lfcn}	\Vert_{W^{1,\infty}(B)},\Vert \uld{\check{\pbar}}	\Vert_{H^{2}}, \Vert \uld{\check{\ptil}}\Vert_{L^2(H^{2})} ,\Vert \uld{\check{\pbar}_{tt}} \Vert_{L^2(L^{2})},\Vert \uld{\check{\ptil}_{tt}} \Vert_{L^2(L^{2})}),
\end{align*}
where $P_{3}$ indicates polynomial terms of order 2 and 3 in the indicated norms, while $C$ is a function of the indicted terms.
Similarly, 
\begin{align*}
 & I_{4}= \frac{\uld{\lfcn}}{\lref} (\uld{\check{\ptil}_{t}}+\uld{\check{\pbar}_{t}})^{2}
 	+ \frac{\uld{\lfcn}}{\lref} 2( \check{\ptil}_{t} +\check{\pbar}_{t})(\uld{\check{\ptil}_{t}}+\uld{\check{\pbar}_{t}})
	+ \omegazero_{\lfcn}(\uld{\check{\ptil}_{t}}+\uld{\check{\pbar}_{t}})^{2}
\end{align*}
which using the continuous embedding of $H^{1}$ into $L^{4}$ again can be estimated by
\begin{align*}
\Vert I_{4} \Vert_{L^{2}} & \leq C( 
				\Vert \check{\ptil}_{t} \Vert_{L^2(H^{1})},
				\Vert \check{\pbar}_{t} \Vert_{L^2(H^{1})}
				,\Vert \lfcn \Vert_{L^{\infty}(B)} ,  \Vert \tfrac{1}{\lref} \Vert_{L^{\infty}(B)}) 
  \, P_{3}(\Vert \uld{\lfcn}	\Vert_{W^{1,\infty}(B)},\Vert \uld{\check{\pbar}_t}	\Vert_{L^2(H^{1})}, \Vert \uld{\check{\ptil}}_{t}\Vert_{L^2(H^{1})})
  \\ &
  =o(\|(\uld{ \vec{g}}, \uld{\vec{u}})\|_{G\times U}).
\end{align*}
The terms in $\wtil$ on $\Gampl$ are linear, so it remains to estimate the boundary terms  on $\Gamma_{0N}$. In particular, we estimate the $L^{2}$ norm of 
\begin{align*}
I_{5} = 
& 
\omegaone_{\lfcn+ \uld{\lfcn}}\,  [\nu_0\cdot M_{\lfcn+\uld{\lfcn}} \,\nabla (\check{\pbar}+\uld{\check{\pbar}}) -(\check{\gfcn}+ \uld{\check{\gfcn}} )] 
- \omegaone_{\lfcn} \, [\nu_0\cdot M_{\lfcn} \,\nabla \check{\pbar}- \check{\gfcn}  ] 
\\
&
\indeq + \frac{1}{\sqrt{|\nabla \lfcn|^{2}+1}} \omegaone_{\lfcn}   \nabla \lfcn \cdot \nabla \uld{\lfcn}
\, [\nu_0\cdot M_{\lfcn} \,\nabla \check{\pbar}- \check{\gfcn}  ] 
- \omegaone_{\lfcn}\,  [\nu_0\cdot M_{\lfcn} \,\nabla \uld{\check{\pbar}} - \uld{\check{\gfcn}} ] 
- \omegaone_{\lfcn} \,  [\nu_0\cdot  M'_{\lfcn}  \,\nabla \check{\pbar}]
\\
& = A+B +C +D+ E
\end{align*}
where
\begin{align*}
M'_{\lfcn} = \left(\begin{array}{cc}
  0 & \left(\frac{\nabla_{x'} \uld{\lfcn} }{\lfcn} -\frac{\nabla_{x'} \lfcn  }{\lfcn^{2}} \uld{\lfcn}\right)    \check{z} \\
  0& - \frac{\lref}{\lfcn^{2}} \uld{\lfcn} \end{array}\right).
 \end{align*}
We estimate five principal parts:
\begin{align*}
A= 
& 
( \omegaone_{\lfcn+ \uld{\lfcn}} - \omegaone_{\lfcn} + \frac{1}{ |\nabla \lfcn|^{2}+1 } \omegaone_{\lfcn}   \nabla \lfcn \cdot \nabla \uld{\lfcn} ) \,  [\nu_0\cdot M_{\lfcn} \,\nabla \check{\pbar} -\check{\gfcn} )] \\
B = &  \omegaone_{\lfcn+ \uld{\lfcn}}  [\nu_0\cdot ( M_{\lfcn+\uld{\lfcn}}-M_{\lfcn} -  M'_{\lfcn}) \,\nabla \check{\pbar} \\
C = &  ( \omegaone_{\lfcn+ \uld{\lfcn}} - \omegaone_{\lfcn})  [\nu_{0}\cdot  M'_{\lfcn} \,\nabla \check{\pbar}] \\
D = & ( \omegaone_{\lfcn+ \uld{\lfcn}} - \omegaone_{\lfcn} ) [\nu_0\cdot M_{\lfcn} \,\nabla \uld{\check{\pbar}} - \uld{\check{\gfcn}} ] \\
F = &  \omegaone_{\lfcn+ \uld{\lfcn}}  [\nu_0\cdot ( M_{\lfcn+\uld{\lfcn}}-M_{\lfcn}) \,\nabla \uld{\check{\pbar}}]
\end{align*}
To estimate the $L^{2}$ norm, we note that
\begin{align*}
( \omegaone_{\lfcn+ \uld{\lfcn}} - \omegaone_{\lfcn} + \frac{1}{ |\nabla \lfcn|^{2}+1 } \omegaone_{\lfcn}   \nabla \lfcn \cdot \nabla \uld{\lfcn} ) = \frac{  \nabla \uld{\lfcn} \cdot \nabla \lfcn ( 2 |\nabla \uld{\lfcn}| | \nabla \lfcn| +|\nabla \uld{\lfcn}|^{2} )
+ |\nabla \uld{\lfcn}|^{2} \sqrt{ 1+ |\nabla \lfcn|^{2}} } 
{ \sqrt{ 1+ |\nabla \uld{\lref}|^{2}}  \sqrt{ 1+ |\nabla \lfcn|^{2}}
(\sqrt{ 1+ |\nabla \lfcn|^{2}}
+ \sqrt{ 1+ |\nabla \lfcn + \nabla \uld{\lfcn} |^{2}})^{2}},
\end{align*}
and hence 
\begin{align*}
\left\Vert \omegaone_{\lfcn+ \uld{\lfcn}} - \omegaone_{\lfcn} + \frac{1}{ |\nabla \lfcn|^{2}+1 } \omegaone_{\lfcn}   \nabla \lfcn \cdot \nabla \uld{\lfcn} 
\right\Vert_{L^{\infty}(B)} 
\leq C( \Vert\lfcn \Vert_{W^{1,\infty}(B)}) \sum_{j=2}^{3}\Vert  \uld{\lfcn} \Vert_{W^{1,\infty}(B)}^{j} \\
= C( \Vert\lfcn \Vert_{W^{1,\infty}(B)}) o(\Vert  \uld{\vec{g}} \Vert_{G})
\end{align*}
Hence, 
\begin{align*}
\Vert A \Vert_{L^2(\zeroT L^{2}(\Gamma_N))} 
		\leq  C( \Vert \lfcn \Vert_{W^{1,\infty}(B)}, \Vert \check{\pbar} \Vert_{L^2(H^{2})}, \Vert g \Vert_{L^\infty(L^{2}(\Gamma_{0N}))})\,  o(\Vert  \uld{\vec{g}} \Vert_{G}).
\end{align*}
For $B$, we use 
\begin{align*}
M_{\lfcn + \uld{\lfcn}} - M_{\lfcn} - M'_{\lfcn} =
  \left(\begin{array}{cc}
  0 & \left( \frac{- \nabla_{x'} \uld{\lfcn} \, \lfcn \, \uld{\lfcn} +\nabla_{x'} \lfcn \uld{\lfcn}^{2} }{(\lfcn+\uld{\lfcn}) \lfcn^{2}} \right)    \check{z} \\
  0&  \frac{\lref (\uld{\lfcn})^{2}}{(\lfcn+\uld{\lfcn}) \lfcn^{2}} \end{array}\right);
\end{align*}
estimating the $L^{\infty}$ norm we get
\begin{align*}
\Vert M_{\lfcn + \uld{\lfcn}} - M_{\lfcn} - M'_{\lfcn} \Vert_{L^{\infty}(B)}
\leq C( \Vert \lref \Vert_{L^{\infty}(B)}, 1/ \Vert \lfcn \Vert_{L^{\infty}(B)})   \Vert \uld{\lfcn} \Vert_{W^{1,\infty}(B)}^{2}  
\end{align*}
and hence, 
\begin{align*}
\Vert B \Vert_{L^2(\zeroT L^{2}(\Gamma_N))} 
\leq C( \Vert \lfcn \Vert_{W^{1,\infty}(B)}, \Vert \check{\pbar} \Vert_{L^2(H^{2})} )\, o(\Vert  \uld{\vec{g}} \Vert_{G})
\end{align*}
The terms  $C$, $D$ and $F$ can be similarly estimated by
 $o(\Vert  \uld{\vec{g}} \Vert_{G})$ with constants depending on the norm of $\Vert \vec{g} \Vert_{G}$ and $\Vert \vec{u} \Vert_{U}$.
\end{proof}

\subsection{Surjectivity}\label{subsec:surj}
\begin{Proposition}\label{prop:surj}
Let the spaces $G_{ad}$, $U_{ad}=U(\Omref)$, $Z=Z(\Omref)$ be  as in \eqref{fcnsp}, \eqref{defG0}, \eqref{G}, \eqref{medreg} and let \eqref{ass_l0} hold.
Then for any $\vec{f}\in Z^*$ there exists $(\uld{\vec{g}}, \uld{\vec{u}})\in G_{ad}\times U_{ad}$  such that $A_{\textup{PDE}}'(\vec{g}^*,\vec{u}^*)(\uld{ \vec{g}}, \uld{\vec{u}}) =\vec{f}$.
\end{Proposition}

\begin{proof}
For better readability we skip the superscript ${}^*$ in $(\vec{g}^*,\vec{u}^*)=(\check{\gfcn}^*,\hfcn^*,\lfcn^*,\pbar^*,\ptil^*,\wtil^*)$.

Verifying condition \eqref{ZoweKurcyusz} means, given any $\vec{f}=(f_{\pbar}, f_{\ptil}, f_{\wtil}, f_N, f_{pl})\in Z^*$ we have to be able to choose $(\uld{\gfcn},\uld{\hfcn},\uld{\lfcn},\uld{\pbar},\uld{\ptil},\uld{\wtil})\in G\times U(\Omega)$, cf.\eqref{fcnsp}, \eqref{G}, such that the linearized PDE with inhomogeneity $\vec{f}$ holds.
In strong form and written in terms of the domain $\Omega=\Omega(\lfcn)$ rather than in terms of the reference domain $\Omref$, this implies that solvability of the following system \eqref{dpbar_decoup}, \eqref{dptil_decoup} needs to be shown.
\begin{equation}\label{dpbar_decoup}
\begin{aligned}
\uld{\pbar}_{tt}  - c^{2}\Delta \uld{\pbar} - b \Delta \uld{\pbar}_{t}  &= f_{\pbar} ~~\mbox{in}~~ \Omega \times [0,T] \\
%c\Dnu{\uld{\pbar} } +\uld{\pbar}_t
\Dnu{\uld{\pbar}}+\absbc[\uld{\pbar},\uld{\pbar}_t]
&= 0 ~~\mbox{on}~~ \Gamma_{a} \times [0,T]  \\
\Dnu{ \uld{\pbar}}&=\uld{\gfcn}+f_N ~~\mbox{on}~~ \Gamma_{N} \times [0,T] \\
\Dnu{ \uld{\pbar}} &= 0  ~~\mbox{on}~~ \Gampl \times [0,T] \\
\uld{\pbar}(0)&=0, ~~\uld{\pbar}_{t}(0)=0,
\end{aligned}
\end{equation}
\begin{equation}\label{dptil_decoup}
\begin{aligned}
(1-2k(\pbar+\ptil))\, \uld{\ptil}_{tt}  - c^{2}\Delta \uld{\ptil} - b \Delta \uld{\ptil}_{t}  
&=4k(\pbar+\ptil)_t (\uld{\pbar}+\uld{\ptil})_t 
+2k(\pbar+\ptil)\uld{\pbar}_{tt} \\
&\qquad
+2k(\uld{\pbar}+\uld{\ptil})\, (\pbar+\ptil)_{tt}+f_{\ptil}~~\mbox{in}~~ \Omega \times [0,T] \\
% c\Dnu{\uld{\ptil}}+\uld{\ptil}_{t} 
 \Dnu{\uld{\ptil}}+\absbc[\uld{\ptil},\uld{\ptil}_{t}] 
 &= 0  ~~\mbox{on}~~ \Gamma_{a} \times [0,T]  \\
\Dnu{\uld{\ptil}}&=0 ~~\mbox{on}~~ \Gamma_{N} \times [0,T] \\
%\Dnu{\uld{\ptil}} &= - \rho \uld{\wtil}  ~~\mbox{on}~~ \Gampl \times [0,T] \\
\Dnu{\uld{\ptil}} + \rho \uld{\wtil}_{t}&= f_{pl}  ~~\mbox{on}~~ \Gampl \times [0,T] \\
%\rho\uld{\wtil}_{tt} + \delta\Delpl^{2} \uld{\wtil} 
%\Damppl{+  \beta(-\Delpl)^{\gamma} \uld{\wtil}_{t}} 
%&= \kappa \uld{\ptil}_{tt} +\kappa \uld{\pbar}_{tt} +\uld{\hfcn}+f_{\wtil}  ~~\mbox{on}~~ \Gampl \times [0,T] \\
\rho\uld{\wtil}_{tt} + \delta\Delpl^{2} \uld{\wtil} 
\Damppl{+  \beta(-\Delpl)^{\gamma} \uld{\wtil}_{t} }
&= \kappa \uld{\ptil}_t +\kappa \uld{\pbar}_t + \uld{\hfcn}+f_{\wtil}  ~~\mbox{on}~~ \Gampl \times [0,T] \\
\uld{\ptil}(0)&= 0, ~~\uld{\ptil}_{t}(0)= 0\\
\uld{\wtil}(0)&= 0, ~~\uld{\wtil}_{t}(0)= 0,
\end{aligned}
\end{equation}
with 
%$I_t=\int_0^t \, ds$ and 
$\uld{\wtil}=\uld{w}_{t}$.\\
%\colBK{$p=\rho\psi_t$}\\
Since $\uld{\lfcn}$ anyway does not help with reaching $\vec{f}$ and nonvanishing boundary variation would just complicate the situation, we choose $\uld{\lfcn}=0$.

In \eqref{dpbar_decoup}, we can just set $\uld{\gfcn}=-f_N$ and in \eqref{dptil_decoup} we can set 
$\uld{\hfcn}:=-f_{\wtil}-\kappa\,\text{tr}_{\partial \Gampl}(\uld{\pbar}+\uld{\ptil})_{t}\in L^2(\zeroT H^{2}_{\diamondsuit}(\Gampl)^*)$, which allows to choose $\uld{\wtil}=0$.
%\footnote{If we don't optimize for $\hfcn$, we cannot do this, though; in that case we must resort to \cite[Proposition 3.5]{KaTu1} for estblishing surjectivity} 
With this, only the $\pbar$ and $\ptil$ equations 
\begin{equation}\label{pp_lin}
\begin{aligned}
\uld{\pbar}_{tt}  - c^{2}\Delta \uld{\pbar} - b \Delta \uld{\pbar}_{t}  &= f_{\pbar} ~~\mbox{in}~~ \Omega \times [0,T] \\
(1-2k(\pbar+\ptil))\, \uld{\ptil}_{tt}  - c^{2}\Delta \uld{\ptil} - b \Delta \uld{\ptil}_{t}  
&=4k(\pbar+\ptil)_t (\uld{\pbar}+\uld{\ptil})_t 
+2k(\pbar+\ptil)\uld{\pbar}_{tt} \\
&\qquad
+2k(\uld{\pbar}+\uld{\ptil})\, (\pbar+\ptil)_{tt}+f_{\ptil}~~\mbox{in}~~ \Omega \times [0,T],
\end{aligned}
\end{equation}
equipped with homogeneous absorbing boundary conditions on $\Gamma_a$, homogeneous Neumann conditions on $\Gamma_N$ and $\Dnu{\uld{\ptil}}= f_{pl}$ on $\Gampl$
as well as homogeneous initial conditions, remain.
%They can be covered (successively; first $\uld{\pbar}$, then $\uld{\ptil}$) by existing results.

The variational form of \eqref{pp_lin} (conforming to the linearization of \eqref{eqn:var_L2L2} with $\uld{\wtil}=0$ and the function space setting \eqref{fcnsp}) is    
\begin{equation}\label{eqn:var_L2L2_pponly}
\begin{aligned}
&\int_0^T\Bigl\{
\int_\Omega \Bigl(
\bigl(\uld{\pbar}_{tt}  - c^{2}\Delta \uld{\pbar} - b \Delta \uld{\pbar}_{t} -f_{\pbar}\bigr)\,\qbar\\
&\qquad\qquad+\bigl((1+\mathfrak{a})\, \uld{\ptil}_{tt}  - c^{2}\Delta \uld{\ptil} - b \Delta \uld{\ptil}_{t} +\mathfrak{b} (\uld{\pbar}+\uld{\ptil})_t 
+\mathfrak{c}(\uld{\pbar}+\uld{\ptil})+\mathfrak{a}\uld{\pbar}_{tt} -f_{\ptil}\bigr)\,\qtil\Bigr)\, dx\\
&\qquad+ 
\int_{\Gamma_N}
%[(\bar{\lambda}+b)\partial_t+c^2]
\Dnu{\uld{\pbar}}\,\mu_{N}\, dS
+\int_{\Gampl}
%[(\tilde{\lambda}(1+\mathfrak{a})+b)\partial_t+c^2]
(\Dnu{\uld{\ptil}}-f_{pl})\,\mu_{pl}\,dS
\Bigr\}\,dt\\
&=0 \qquad\text{ for all }(\qbar,\qtil)\in Z_{pp}
\end{aligned}
\end{equation}
with 
\begin{equation}\label{abc}
\mathfrak{a}=-2k(\pbar+\ptil) , \quad
\mathfrak{b}=-4k(\pbar+\ptil)_t, \quad
\mathfrak{c}=-2k(\pbar+\ptil)_{tt},
\end{equation}
initial conditions 
$(\uld{\pbar},\uld{\ptil})(0)=(0,0)$, 
$(\uld{\pbar}_t,\uld{\ptil}_t)(0)=(0,0)$ 
on the function spaces
\begin{equation}\label{fcnsp_pponly}
\begin{aligned}
&U_{pp}=\{p\in H^2(\zeroT L^2(\Omega))\cap 
%H^1(\zeroT H^2_{\Delta,0}(\Omega))\cap L^2(\zeroT H^2_{\Delta,1}(\Omega))
H^1(\zeroT H^2_{\Delta,1}(\Omega))
\,:\, 
%c\Dnu p+p_t
\Dnu p+\absbc[p,p_t]
=0\text{ on } \Gamma_{a}, \ \Dnu p=0\text{ on } \partial \Gampl\}\\
&\qquad\times\{p\in H^2(\zeroT L^2(\Omega))
%\cap H^1(\zeroT H^2_{\Delta,0}(\Omega))\cap L^2(\zeroT H^2_{\Delta,1}(\Omega))
\cap H^1(\zeroT H^2_{\Delta,1}(\Omega))
\,:\, 
%c\Dnu p+p_t
\Dnu p+\absbc[p,p_t]
=0\text{ on } \Gamma_{a}, \ \Dnu p=0\text{ on } \Gamma_{N}\} \\
&Z_{pp}=\bigl(L^2(\zeroT L^2(\Omega))\bigr)^2
\times L^2(\zeroT H^{-s}(\Gamma_N)) \times L^2(\zeroT 
\Damppl{H^{-\min\{s,\gamma\}}(\Gampl))}
H^{-s}(\Gampl)),
\end{aligned}
\end{equation}
that is, the $\pbar,\ptil$ part of \eqref{fcnsp}.

Existence of a solution in $U_{pp}$ to \eqref{pp_lin} can be concluded from the following lemma with
$\mathfrak{a}=-2k(\pbar+\ptil)\in 
%H^1(\zeroT H^2_{\Delta,0}(\Omega))\cap L^2(\zeroT H^2_{\Delta,1}(\Omega)) 
H^1(\zeroT H^2_{\Delta,1}(\Omega))$,
$\mathfrak{b}=-4k(\pbar+\ptil)_t\in 
%L^2(\zeroT H^2_{\Delta,0}(\Omega)),
L^2(\zeroT H^2_{\Delta,1}(\Omega))$,
$\mathfrak{c}=-2k(\pbar+\ptil)_{tt}\in L^2(\zeroT L^2(\Omega))$
together with Sobolev embeddings.

\begin{Lemma}\label{lem:enest_surj}
There exists $r>0$ such that for any 
$\mathfrak{a}\in L^\infty(\zeroT L^\infty(\Omega)) $,
$\mathfrak{b}\in L^2(\zeroT L^3(\Omega))$,
$\mathfrak{c}\in L^2(\zeroT L^2(\Omega))$ with 
\begin{equation}\label{smallness_coeffs}
\|\mathfrak{a}\|_{L^\infty(\zeroT L^\infty(\Omega))}+
\|\mathfrak{b}\|_{L^2(\zeroT L^3(\Omega))}+
\|\mathfrak{c}\|_{L^2(\zeroT L^2(\Omega))}\leq r
\end{equation}
and any 
$f_{\pbar},\, f_{\ptil}\, \in L^2(\zeroT L^2(\Omega))$,
$f_{pl}\in L^2(\zeroT H^s(\Gampl))$,
there exists a solution to \eqref{eqn:var_L2L2_pponly} in $U_{pp}$.
\end{Lemma}
Proof. See the appendix.
\\
The smallness conditions \eqref{smallness_coeffs} can be concluded from 
\begin{equation}\label{smallness_nonlin}
\Vert\mathfrak{a}\Vert_{H^1(\zeroT H^2_{\Delta,1}(\Omega))}+
\Vert\mathfrak{b}\Vert_{L^2(\zeroT 
%H^2_{\Delta,0}(\Omega)
H^2_{\Delta,1}(\Omega)
)}+
\Vert\mathfrak{c}\Vert_{L^2(\zeroT L^2(\Omega))}\leq \tilde{r}
\end{equation}
with small enough $\tilde{r}>0$
for $\pbar,\ptil$ being part of a minimizer, thus solving the PDE so that we can use the energy estimates from Lemma~\ref{lem:estAPDE0} to bound it. This includes the $L^\infty(\zeroT H^2_{\Delta,1}(\Omega))$ norm, which together with Stampacchia's method allows to guarantee nondegeneracy by smallness of $\|\mathfrak{a}\|_{L^\infty(\zeroT L^\infty(\Omega))}$.
%So it does not matter that we stay with the somewhat weaker Hilbert space setting for defining $A_{PDE}$). 
Condition \eqref{smallness_nonlin} in its turn can be achieved by smallness of the prescribed initial data and smallness of $\gfcn$, $\hfcn$ due to the $\theta$ penalty term in the cost function; comparing the cost function value, e.g., to the one at $(\vec{0},\vec{u})$ with $\vec{u}$ some extension of the prescribed initial data. 
%-- just $\vec{u}=0$ if we assume vanishing initial data.

This proves \eqref{ZoweKurcyusz}; note that we have heavily meade use of $\hfcn$ being one of the design variables $\vec{g}$. 
%\footnote{\ntBK{Otherwise we would have to show well-posedness of the linarization of \eqref{eqn:var_L2L2}. This could possibly work out for $\gamma$ large enough in the plate damping; probably $\gamma\geq2$ would be needed.}}
\end{proof}
%\ntBK{
% DONE Use nonlinear version of \eqref{dpbar_decoup}, \eqref{dptil_decoup} (maybe with $w$ or $w'$ rather than $\wtil=w_{tt}=w'_t$ as a variable) in $Z^*=L^2(\zeroT L^2(\Omega))^2\times L^2(\zeroT H^{-2}(\Gampl))$ to define the PDE constraint $A_{PDE}$. 
%Exploit Lemma 3.4 (but we have lower regularity than we assume there) in KaTu1 to return to variational formulation and prove additional (to being in $Z$) regularity of the adjoint state.}

%Uniqueness of a critical point together with existence of a global minimizer (see Section~\ref{sec:existence_min}) would imply that the first order necessary conditions are also sufficient.

%Alternatively, we can look at second order sufficient conditions, see page 334-335 in \cite{Troeltzsch2010}.

\section{Conclusions and Outlook}
In this paper, we have considered simultaneous optimization of the shape of a boundary part and of two excitation control functions in a coupled nonlinear acoustics-plate system.
\\
Future work might be concerned with controllability of this system, as well as with advanced models replacing the Westervelt and Kirchhoff plate equations. 
In particular, considering a deformable shell might also allow us to view its shape as part of the design variables.
However, the use of piezoelectric plates or shells would incorporate modeling of the excitation mechanism.

\bigskip
\section*{Appendix}
\setcounter{equation}{0}
\appendix
\input{appendix_strong1storder}
\newpage 
\input{appendix_energy-estimates}
\newpage 
\input{appendix_surjectivity}

\section{Funding and/or Conflicts of interests/Competing interests}
This work was supported by the American University of Sharjah (grant number FRG23-E-S70). This work was also partially supported by the Austrian Science Fund (FWF) (grant number [10.55776/P36318]).
\\
The authors have no relevant financial or non-financial conflicts or competing interests to disclose.
\\
All authors read and approved the final manuscript.
%\bibliographystyle{plain}
%\bibliography{lit_KaTu2}

\input{KaTu2_2024-07-02.bbl}
\end{document}

%% file: appendix_strong1storder.tex
\section{Strong form of first order optimality conditions}
Summarizing, (for illustration purposes) we have formally obtained the following first order optimality conditions.

\noindent
\underline{State equations:}
\begin{equation}\label{stateeq}
\text{\eqref{pbar_decoup} and \eqref{ptil_decoup}}
\text{ with }\Omega=\Omega(\lfcn) 
\text{ and homogeneous initial data}
\end{equation}

\noindent
\underline{Adjoint equations:}
\def\bbar{\bar{b}}
\def\btil{\tilde{b}}
We note that with the abbreviations 
%$\ptot=\pbar+\ptil$,
%$\bbar=\bar{\lambda}+b$,
%$\btil=\tilde{\lambda}(1-2k(\pbar+\ptil))+b$
and collecting some terms, \eqref{adj_qbar} and \eqref{adj_qtilvtil}, can be written more compactly as 
\begin{equation}\label{adj_weak}
\begin{aligned}
&(\qbar,\qtil,\vtil, \mu_{N}, \mu_{pl} )\in Z(\lfcn) 
%\bigl(L^{2}(\zeroT L^2({\Omega(\lfcn)}))\bigr)^2\times L^{2}(\zeroT H^{2}_{\diamondsuit}(\Gampl))\times L^{2}(\zeroT L^{2}(\Gamma_{N})) \times L^{2}(\zeroT L^{2}(\Gampl)),
\text{ and }\\
&0=\int_0^T\Bigl\{\int_{\Omega(\lfcn)}\chi_{\textup{ROI}}(\pbar+\ptil-p_{\textup{d}})(\phibar+\phitil) \, dx\\
&\quad
+\int_{\Omega(\lfcn)} \Bigl( 
( \phibar_{tt} - c^2\Delta \phibar - b\Delta \phibar_t)\, \qbar
+( \phitil_{tt} - c^2\Delta \phitil - b\Delta \phitil_t
-2k\bigl((\pbar+\ptil)(\phibar+\phitil)\bigr)_{tt})\, \qtil\Bigr)\, dx
\\
&\qquad
 +\tfrac{\rho}{\kappa}\int_{\Gampl} 
 	 \Bigl( (\rho\psitil_{tt} + \beta(-\Delpl)^{\gamma} \psitil_{t} 
	 		- \kappa (\phitil+\phibar)_{t} ) \vtil
 +\delta\Delpl\psitil\, \Delpl\vtil  \Bigr) \,dS
\\
&\qquad
+\int_{\Gamma_N}
%[\bbar \partial_t+c^2]
\Dnu{\phibar} \,\mu_{N}\, dS
+\int_{\Gampl}\Bigl(
%[\btil \partial_t+c^2]
(\Dnu{\phitil} + \rho \psitil_t)\,\mu_{pl}
%-2k \tilde{\lambda} (\phibar+\phitil) \, ( \Dnu{\ptil}_{t} + \rho \wtil_{tt}) \,\mu_{pl}
\Bigr)\,dS
\Bigr\}dt
\\
&\qquad
\text{for all }(\phibar,\phitil,\psitil)\in U_0(\lfcn):=\{(\phibar,\phitil,\psitil)\in U(\lfcn)\, : \,
\phibar(0)=0,\, \phibar_t(0)=0,\,\phitil(0)=0,\, \phitil_t(0)=0,\,\\
&\hspace*{5cm}\psitil(0)=0,\, \psitil_t(0)=0\}
\end{aligned}
\end{equation}
where $U(\lfcn)$, $Z(\lfcn)$ are defined as in \eqref{fcnsp} with $\Omega=\Omega(\lfcn)$.
We (formally, since general elements of $Z(\lfcn)$ have hardly any differentiability) integrate by parts to remove all derivatives from the test functions, e.g.
\begin{align*}
   &\int_0^T\int_{\Omega(\lfcn)} 
   [b \partial_t+c^2](-\Delta)\phibar\, \qbar\, dx\, dt
   =\int_0^T\int_{\Omega(\lfcn)} 
   \phibar\, [-b \partial_t+c^2](-\Delta)\qbar\, dx\, dt\\
   &+\int_0^T\int_{\partial\Omega(\lfcn)} \Bigl(-\Dnu{\phibar}\,[-b \partial_t+c^2]\qbar
   +\phibar\,[-b \partial_t+c^2]\Dnu{\qbar}\Bigr)\, dS\, dt\\
   &+b\left[\int_{\Omega(\lfcn)} 
   \phibar\, (-\Delta)\qbar\, dx
   +\int_{\partial\Omega(\lfcn)} \Bigl(-\Dnu{\phibar}\,\qbar
   +\phibar\,\Dnu{\qbar}\Bigr)\, dS\right]_0^T,
\end{align*}
and obtain
\begin{equation*}
\begin{aligned}
&0=\int_0^T\Bigl\{
\int_{\Omega(\lfcn)}
\phibar\Bigl(\chi_{\textup{ROI}}(\pbar+\ptil-p_{\textup{d}})
+\qbar_{tt} - c^2\Delta \qbar + b\Delta \qbar_t -2k(\pbar+\ptil)\qtil_{tt}\Bigr)\, dx\\
&\qquad\qquad\qquad +\int_{\Gampl}\phibar\Bigl(
\rho \vtil_t 
+[-b \partial_t+c^2]\Dnu{\qbar}
%-2k \tilde{\lambda} \, ( \Dnu{\ptil}_{t} + \rho \wtil_{tt}) \,\mu_{pl}
\Bigr)\,dS\\
&\qquad\qquad\qquad
+\int_{\Gamma_N}\Dnu{\phibar}
%\Bigl(-(\bbar\mu_{N})_t+c^2\mu_{N} -[-b \partial_t+c^2]\qbar \Bigr)
\Bigl(\mu_{N} -[-b \partial_t+c^2]\qbar \Bigr)
\,dS
\\
&\qquad\qquad\qquad +\int_{\Gamma_N}\phibar\,[-b \partial_t+c^2]\Dnu{\qbar}\,dS
+\int_{\Gamma_a}\phibar\,[-b \partial_t+c^2](\Dnu{\qbar}-\tfrac{1}{c}\qbar_t)\,dS\\
&\qquad\qquad + \int_{\Omega(\lfcn)}
\phitil\Bigl(\chi_{\textup{ROI}}(\pbar+\ptil-p_{\textup{d}})
+\qtil_{tt} - c^2\Delta \qtil + b\Delta \qtil_t -2k(\pbar+\ptil)\qtil_{tt}\Bigr)\, dx\\
&\qquad\qquad\qquad +\int_{\Gampl}\phitil\Bigl(
\rho \vtil_t 
+[-b \partial_t+c^2]\Dnu{\qtil}
%-2k \tilde{\lambda} \, ( \Dnu{\ptil}_{t} + \rho \wtil_{tt}) \,\mu_{pl}
\Bigr)\,dS\\
&\qquad\qquad\qquad
+\int_{\Gampl}\Dnu{\phitil}
%\Bigl(-(\btil\mu_{pl})_t+c^2\mu_{pl} -[-b \partial_t+c^2]\qtil \Bigr)
\Bigl(\mu_{pl} -[-b \partial_t+c^2]\qtil \Bigr)
\,dS
\\
&\qquad\qquad\qquad +\int_{\Gamma_N}\phitil\,[-b \partial_t+c^2]\Dnu{\qtil}\,dS
+\int_{\Gamma_a}\phitil\,[-b \partial_t+c^2](\Dnu{\qtil}-\tfrac{1}{c}\qtil_t)\,dS
\\
&\qquad\qquad
 +\tfrac{\rho}{\kappa}\int_{\Gampl} 
 	 \psitil\Bigl( (\rho\vtil_{tt} - \beta(-\Delpl)^{\gamma} \vtil_{t} +\delta(-\Delpl)^2\vtil  
-\kappa 
%\bigl(-(\btil\mu_{pl})_t+c^2\mu_{pl}\bigr)_t
\mu_{pl\,t}
\Bigr) \,dS
\Bigr\}dt
\\
&\qquad
\text{for all }(\phibar,\phitil,\psitil)\in U_0(\lfcn):=\{(\phibar,\phitil,\psitil)\in U(\lfcn)\, : \,
\phibar(0)=0,\, \phibar_t(0)=0,\,\phitil(0)=0,\, \phitil_t(0)=0,\,\psitil(0)=0,\, \psitil_t(0)=0\}
\end{aligned}
\end{equation*}
Here we have used 
$\Dnu{\phibar}=0$ on $\Gampl$, 
$\Dnu{\phitil}=0$ on $\Gamma_N(\lfcn)$, 
$\Dnu{\phibar}+\tfrac{1}{c}\phibar_t=0$ on $\Gamma_a$,
$\Dnu{\phitil}+\tfrac{1}{c}\phitil_t=0$ on $\Gamma_a$ (by definition of the function space $U(\lfcn)$), 
and already skipped the final time terms, as their vanishing (due to variation of the test functions) is equivalent to homogeneous zero and first order final time conditions on $(\qbar,\qtil,\vtil)$.

Since the third and seventh lines imply $\mu_{N} =[-b \partial_t+c^2]\qbar\vert_{\Gamma_N}$ and
$\mu_{pl} =[-b \partial_t+c^2]\qtil\vert_{\Gampl}$
we arrive at the following strong formulation of the adjoint system.
\begin{equation}\label{adj_strong}
\begin{aligned}
\qbar_{tt}  - c^{2}\Delta \qbar + b \Delta \qbar_{t}  &= \chi_{\textup{ROI}}(\pbar+\ptil-p_{\textup{d}})+2k(\pbar+\ptil)\qtil_{tt} ~~\mbox{in}~~ \Omega(\lfcn) \times [0,T] \\
(1-2k(\pbar+\ptil))\qtil_{tt}  - c^{2}\Delta \qtil + b \Delta \qtil_{t}  &= \chi_{\textup{ROI}}(\pbar+\ptil-p_{\textup{d}}) ~~\mbox{in}~~ \Omega(\lfcn) \times [0,T] \\
\rho\vtil_{tt} + \delta\Delpl^{2} \vtil -  \beta(-\Delpl)^{\gamma} \vtil_{t} &= \kappa [-b\partial_t+c^2]\qtil_{t}  ~~\mbox{on}~~ \Gampl \times [0,T] \\
\rho\vtil_{t}=- [-b\partial_t+c^2]\Dnu{\qbar} &=-[-b\partial_t+c^2]\Dnu{\qtil} ~~\mbox{on}~~ \Gampl \times [0,T] \\
c\Dnu{ \qbar} +\qbar_t=c\Dnu{ \qtil} +\qtil_t&= 0 ~~\mbox{on}~~ \Gamma_{a} \times [0,T]  \\
\Dnu{\qbar}=\Dnu{\qtil}&=0 ~~\mbox{on}~~ \Gamma_{N}(\lfcn) \times [0,T] \\
\qbar(T)=0, ~~\qbar_{t}(T)=0, ~~
\qtil(T)=0, ~~\qtil_{t}(T)&=0, ~~
\vtil(T)=0, ~~\vtil_{t}(T)=0.
\end{aligned}
\end{equation}

\noindent
\underline{Gradient equations:}
\begin{equation}\label{gradienteq}
\begin{aligned}
&\regpar \regop_\gfcn^* \regop_\gfcn (\gfcn-\gfcn_0) = 
%[-\bbar\partial_t+c^2] 
\mu_{N}\\
&\regpar \regop_\hfcn^* \regop_\hfcn (\hfcn-\hfcn_0) = \tfrac{\rho}{\kappa}\,\vtil\\
&\regpar \widehat{\regop_\lfcn^*\regop_\lfcn(\lfcn-\lref)} = 
-\int_0^T \Bigl(
\bigl(\pbar_{tt}  - c^{2}\Delta \pbar - b \Delta \pbar_{t}\bigr)\,\qbar
+\bigl(\ptil_{tt}  - c^{2}\Delta \ptil - b \Delta \ptil_{t}-k(\pbar+\ptil)_{tt}\bigr)\,\qtil
\\
&\qquad\qquad\qquad\qquad
+\partial_{\nu_{\lfcn}} \bigl(
%[\bbar \partial_t+c^2]
(\Dnu{\pbar}-\gfcn)\,\mu_{N}\bigr) + 
%[\bbar \partial_t+c^2]
(\Dnu{\pbar}-\gfcn)\,\mu_{N}\, \widehat{H}_\lfcn\Bigr)\, dt
\end{aligned}
\end{equation}

\bigskip

Note that with \eqref{pbar_decoup}, the state equation for $\pbar$ is decoupled from the equation for $(\ptil,\wtil)$. However, in \eqref{adj_strong}, the equation for $(\qtil,\vtil)$ is not decoupled from the equation for $\qbar$.

%% file: appendix_energy-estimates.tex
\section{Energy estimates for existence of a minimizer; proof of Lemma~\ref{lem:estAPDE0}}

To conclude boundedness of some norm of a sequence of states $(\vec{u}_n)_{n\in\mathbb{N}}$ from boundedness of the controls $(\vec{g}_n)_{n\in\mathbb{N}}$ through $A_{\text{PDE}}(\vec{g}_n,\vec{u}_n)=0$ in step 4. of the proof of existence of a minimizer, we can in principle make use of \cite[Lemma 3.1]{KaTu1} (reducing the temporal differentiability by one order) and \cite[Proposition 3.5 (i)]{KaTu1} with $a=0$, $\Gamma_D=\emptyset$, $g_N=g$, $\alpha=1$, $f=k((\pbar+\ptil)^2)_{tt}$, $\wtil=\wtil_t$.
In order to track dependence of constants on the domain (which varies here), we provide the estimates explicitly here. This also gives us the opportunity to save a bit on temporal differentiability of $g$ and to incorporate the slightly more general absorbing boundary conditions that we are using here.

Due to $A_{\text{PDE}}(\vec{g}_n,\vec{u}_n)=0$, differentiating the PDE for $\wtil_n$ with respect to time, omitting the subscript $n$ and using the abbreviation $\wttil=\wtil_t=w_{tt}$ we have 
\begin{equation}\label{PDE_appendix}
\begin{aligned}
\pbar_{tt}  - c^{2}\Delta \pbar - b \Delta \pbar_{t}  &= 0 ~~\mbox{in}~~ \Omega \times [0,T] \\
%\pbar_{tt}  - c^{2}\Delta \pbar - b \Delta \pbar_{t}  &= 0 ~~\mbox{in}~~ \Omega \times [0,T] \\
\ptil_{tt}  - c^{2}\Delta \ptil - b \Delta \ptil_{t}  &= k((\pbar+\ptil)^2)_{tt} ~~\mbox{in}~~ \Omega \times [0,T] \\
\rho\wttil_{tt} + \delta\Delpl^{2} \wttil 
\Damppl{+  \beta(-\Delpl)^{\gamma} \wttil_{t} }
&= \kappa (\ptil_{tt} + \pbar_{tt}) +\hfcn_t  ~~\mbox{on}~~ \Gampl \times [0,T] \\
\Dnu{ \pbar_t}&=\gfcn_t, \quad \Dnu{ \ptil_t}=0 ~~\mbox{on}~~ \Gamma_{N} \times [0,T] \\
\Dnu{\pbar}+\absbc[\pbar,\pbar_{t}]&= 0, \quad \Dnu{\ptil}+\absbc[\ptil,\ptil_t] =0 ~~\mbox{on}~~ \Gamma_{a} \times [0,T]  \\
\Dnu{ \pbar} &= 0, \quad \Dnu{\ptil}=-\rho\wttil  ~~\mbox{on}~~ \Gampl \times [0,T] \\
\pbar(0)=0, ~~\pbar_{t}(0)=0, \quad \ptil(0)&=p_0,~~\ptil_{t}(0)=p_1, \quad \wttil(0)=w_1,~~\wttil_{t}(0)=w_2, 
\end{aligned}
\end{equation}
with $w_2=\tfrac{1}{\rho}\bigl(\kappa p_1 +\hfcn(0)
-\delta\Delpl^{2} w_0 
\Damppl{-  \beta(-\Delpl)^{\gamma} w_1 }
\bigr)$.

We test the first equation in \eqref{PDE_appendix} with $-\Delta\pbar_t$ and integrate by parts to obtain
\begin{equation}\label{enid_pbar_appendix}
\begin{aligned}
& \frac{1}{2} \|\nabla\pbar_{t}(t) \|_{L^2(\Omega)}^2
+\frac{c^2}{2} \|\Delta\pbar(t) \|_{L^2(\Omega)}^2 
+ b   \int_{0}^{t} \|\Delta\pbar_t\|_{L^2(\Omega)}^2 \, ds
+\betaabs \int_0^t\|\pbar_{tt}\|_{L^2(\Gamma_a)}^2\, ds
+\frac{\gammaabs}{2} \|\pbar_{t}(t) \|_{L^2(\Gamma_a)}^2\\
&=\int_0^t\int_{\Gamma_N}\pbar_{tt} g_{t} \, dS\, ds
\end{aligned}
\end{equation}

Multiplying the second and the third equation in \eqref{PDE_appendix} by $( -\Delta\pbar_t, \frac{\rho}{\kappa}\wttil_t)$ yields
\begin{equation}\label{enid_ptilwtiltil_appendix}
\begin{aligned}
& \frac{1}{2}\left[\|\nabla\ptil_t\|_{L^2(\Omega)}^2\right|_0^t
+\frac{c^2}{2}\left[\|\Delta\ptil\|_{L^2(\Omega)}^2\right|_0^t
+b \int_0^t\|\Delta\ptil_t\|_{L^2(\Omega)}^2\, ds
+\betaabs \int_0^t\|\ptil_{tt}\|_{L^2(\Gamma_a)}^2\, ds
+\frac{\gammaabs}{2} \left[\|\ptil_{t}\|_{L^2(\Gamma_a)}^2\right|_0^t\\
&+\frac{\rho}{\kappa}\Bigl(
\frac{\rho}{2}\left[\|\wttil_t\|_{L^2(\Gampl)}^2\right|_0^t
+\frac{\delta}{2}\left[\|\Delta_{pl}\wttil\|_{L^2(\Gampl)}^2\right|_0^t
\Damppl{+\beta \int_0^t\|(-\Delpl)^{\gamma/2} \wttil_{t}\|_{L^2(\Gampl)}^2\, ds}
\Bigr) 
\\
&=-\int_0^t\int_\Omega k((\pbar+\ptil)^2)_{tt}\, \Delta\ptil_t\, dx\, ds
+ \int_0^t\int_{\Gampl}\bigl(\rho\pbar_{tt}+\tfrac{\rho}{\kappa}\hfcn_t\bigr)\, \wttil_t\, dS \,ds,
\end{aligned}
\end{equation}
 where we have used cancellation of the $\ptil_{tt}\wttil_t$ terms on $\Gampl$. In order to cancel the term  $\int_0^t\int_{\Gampl} \pbar_{tt}\wttil_t \, dS\,ds$ term as well, 
 we additionally multiply the first and the second equations in \eqref{PDE_appendix} by $(-\Delta \tilde{p}_{t}, -\Delta \bar{p}_{t}) $ 
 and integrate by parts to obtain
 \begin{equation}\label{enid_pbarptil_appendix}
\begin{aligned}
&  \left[ \int_{\Omega}  \nabla \pbar_{t} \nabla \ptil_{t} \, dx \right|_{0}^{t}
+\left[ \int_{\Omega}  c^{2} \Delta \pbar \Delta \ptil \, dx\right|_{0}^{t} 
+ 2 b \int_{0}^{t}  \int_{\Omega}  \Delta \pbar_{t} \Delta \ptil_{t} dx  \,ds\\
& \indeq\indeq
 + 2 \betaabs \int_{0}^{t} \int_{\Gamma_{a}} \ptil_{tt} \pbar_{tt} \, dS \, ds
+ \gammaabs \left[ \int_{\Gamma_{a}} \ptil_{t} \pbar_{t} \, dS \right|_{0}^{t}
\\
 & \indeq\indeq=
 - \int_{0}^{t} \int_{\Gampl} \pbar_{tt} \rho \wttil_{t} \, dS \, ds 
 + \int_{0}^{t}\int_{\Gamma_{N}} \ptil_{tt} g_{t} \, dS \, ds 
  -\int_0^t\int_\Omega k((\pbar+\ptil)^2)_{tt}\, \Delta\pbar_t\, dx\, ds
 \end{aligned}
 \end{equation}
Here the initial terms vanish due to the homogeneous initial conditions on $\pbar$.

Adding \eqref{enid_ptilwtiltil_appendix} and \eqref{enid_pbarptil_appendix} to cancel the boundary term $\pbar_{tt}\wttil_t$ on $\Gampl$, and estimating all the other integrals on the left hand side of the second equation using H\"older's and Young's inequalities, we get
\[
\begin{aligned}
& \frac{1}{2}\left[\|\nabla\ptil_t\|_{L^2(\Omega)}^2\right|_0^t
+\frac{c^2}{2}\left[\|\Delta\ptil\|_{L^2(\Omega)}^2\right|_0^t
+b \int_0^t\|\Delta\ptil_t\|_{L^2(\Omega)}^2\, ds
+\betaabs \int_0^t\|\ptil_{tt}\|_{L^2(\Gamma_a)}^2\, ds
+\frac{\gammaabs}{2} \left[\|\ptil_{t}\|_{L^2(\Gamma_a)}^2\right|_0^t\\
&+\frac{\rho}{\kappa}\Bigl(
\frac{\rho}{2}\left[\|\wttil_t\|_{L^2(\Gampl)}^2\right|_0^t
+\frac{\delta}{2}\left[\|\Delta_{pl}\wttil\|_{L^2(\Gampl)}^2\right|_0^t
\Damppl{+\beta \int_0^t\|(-\Delpl)^{\gamma/2} \wttil_{t}\|_{L^2(\Gampl)}^2\, ds}
\Bigr) 
\\
&\leq 
\left | \int_0^t\int_\Omega k((\pbar+\ptil)^2)_{tt}\, \Delta (\ptil_t +\pbar_{t}) 
\, dx\, ds \right|
+ \left |\int_{0}^{t}\int_{\Gamma_{N}} \ptil_{tt} g_{t} \, dS \, ds \right|
+\left | \int_0^t\int_{\Gampl}\tfrac{\rho}{\kappa}\hfcn_t \, \wttil_t\, dS \,ds\right|
\\
& \indeq \indeq
+ \frac{1}{4} \|\nabla\ptil_t(t) \|_{L^2(\Omega)}^2
+  \|\nabla\pbar_t(t) \|_{L^2(\Omega)}^2
+ \frac{c^{2}}{4} \|\Delta\ptil(t) \|_{L^2(\Omega)}^2
 + c^{2} \|\Delta\pbar(t) \|_{L^2(\Omega)}^2
\\
& \indeq \indeq
+ \frac{b}{2} \int_{0}^{t} \|\Delta\ptil_t \|_{L^2(\Omega)}^2 \, ds
+ 2b \int_{0}^{t} \|\Delta\pbar_t \|_{L^2(\Omega)}^2 \, ds
+\frac{\gammaabs}{4} \|\ptil_{t}(t) \|_{L^2(\Gamma_a)}^2
+ \gammaabs \|\pbar_{t}(t)\|_{L^2(\Gamma_a)}^2
\\
& \indeq \indeq
+\frac{\betaabs}{2} \int_0^t\|\ptil_{tt}\|_{L^2(\Gamma_a)}^2\, ds
+2\betaabs \int_0^t\|\pbar_{tt}\|_{L^2(\Gamma_a)}^2\, ds.
\end{aligned}
\]
where 
\[
\begin{aligned}
& k \int_0^t\int_\Omega ((\pbar+\ptil)^2)_{tt}\, \Delta( \ptil_t + \pbar_{t}) \, dx\, ds\\
&\leq 
\frac{k^2}{b} \|((\pbar+\ptil)^2)_{tt}\|_{L^2(L^2(\Omega))}^2
+ \frac{b}{4} \|\Delta\ptil_t\|_{L^2(L^2(\Omega))}^2 + \frac{b}{4} \|\Delta\pbar_t\|_{L^2(L^2(\Omega))}^2
\end{aligned}
\]

The norms of $\ptil$ on the right hand side can then be absorbed into the left hand side of the inequality while the terms involving norms of $\pbar$ can be absorbed after adding four times inequality 
\eqref{enid_pbar_appendix}.
Hence, we have the inequality
\begin{equation}\label{enest_appendix-ex-min}
\begin{aligned}
& \frac{1}{4}\|\nabla\ptil_t(t)\|_{L^2(\Omega)}^2
+\frac{c^2}{4}\|\Delta\ptil(t) \|_{L^2(\Omega)}^2
+\frac{b}{4} \int_0^t\|\Delta\ptil_t\|_{L^2(\Omega)}^2\, ds
+\frac{\betaabs}{2} \int_0^t\|\ptil_{tt}\|_{L^2(\Gamma_a)}^2\, ds
+\frac{\gammaabs}{4} \|\ptil_{t}(t)\|_{L^2(\Gamma_a)}^2\\
&+ \frac{1}{4} \|\nabla\pbar_{t}(t) \|_{L^2(\Omega)}^2
+\frac{c^2}{4} \|\Delta\pbar(t) \|_{L^2(\Omega)}^2 
+ \frac{b}{4}   \int_{0}^{t} \|\Delta\pbar_t\|_{L^2(\Omega)}^2 \, ds
+\frac{\betaabs}{2} \int_0^t\|\pbar_{tt}\|_{L^2(\Gamma_a)}^2\, ds
+\frac{\gammaabs}{4} \|\pbar_{t}(t) \|_{L^2(\Gamma_a)}^2\\
&+\frac{\rho}{\kappa}\Bigl(
\frac{\rho}{2}\|\wttil_t(t) \|_{L^2(\Gampl)}^2
+\frac{\delta}{2}\|\Delta_{pl}\wttil(t) \|_{L^2(\Gampl)}^2
\Damppl{+\beta \int_0^t\|(-\Delpl)^{\gamma/2} \wttil_{t}\|_{L^2(\Gampl)}^2\, ds}
\Bigr) 
\\
&\leq 
\frac{1}{2} \|\nabla p_1\|_{L^2(\Omega)}^2
+\frac{c^2}{2}\|\Delta p_{0} \|_{L^2(\Omega)}^2
+\frac{\gammaabs}{2} \| p_{1}\|_{L^2(\Gamma_a)}^2
+\frac{\rho}{\kappa}\Bigl(
\frac{\rho}{2}\| w_2\|_{L^2(\Gampl)}^2
+\frac{\delta}{2} \|\Delta_{pl}w_{1} \|_{L^2(\Gampl)}^2 \Bigr) 
\\ 
& \indeq \indeq
+ \frac{k^2}{b} \|((\pbar+\ptil)^2)_{tt}\|_{L^2(L^2(\Omega))}^2
+ \left |\int_{0}^{t}\int_{\Gamma_{N}} (\ptil_{tt}+\pbar_{tt}) g_{t} \, dS \, ds \right|
+\left | \int_0^t\int_{\Gampl}\tfrac{\rho}{\kappa}\hfcn_t \, \wttil_t\, dS \,ds\right|
%+ 2\left |\int_{0}^{t}\int_{\Gamma_{N}} \pbar_{tt} g_{t} \, dS \, ds \right|
\end{aligned}
\end{equation}

An estimate on the second time derivative of $\pbar$ can be bootstrapped from the PDE, 
%(i.e, by setting $\qbar=\pbar_{tt}$)
\[
\begin{aligned}
&\Vert\pbar_{tt}\Vert_{L^2(0,t;L^2(\Omega))}^2= 
\Vert c^2\Delta \pbar + b \Delta \pbar_{t}\Vert_{L^2(0,t;L^2(\Omega))}^2\\
&=\Vert c^2\int_0^t\Delta \pbar_t(s)\,ds + b \Delta \pbar_{t}\Vert_{L^2(0,t;L^2(\Omega))}^2
\leq 
(c^4T+b^2) \Vert\Delta \pbar_{t}\Vert_{L^2(0,t;L^2(\Omega))}^2,  
\end{aligned}
\]
where we have used $\pbar(0)=0$ and the right hand side can be bounded by left hand side terms in \eqref{g}. 
%by time integration, this also provides us with bound on the zero order term $\Vert \pbar(t)\Vert_{L^2(\Omega)}$ in the right hand side of \eqref{Stampacchia}.
Likewise, we can proceed for $\ptil$ to obtain
\[
\Vert\ptil_{tt}\Vert_{L^2(0,t;L^2(\Omega))}^2\leq 
4c^4T\Vert \Delta p_0\Vert_{L^2(\Omega))}^2
+ 4(c^4T+b^2) \Vert\Delta \ptil_{t}\Vert_{L^2(0,t;L^2(\Omega))}^2
+2k^2\Vert((\pbar+\ptil)^2)_{tt}\Vert_{L^2(0,t;L^2(\Omega))}^2.
\]
Adding a sufficiently small multiple $\mu=\frac{b}{32(c^4T+b^2)}$ of this to  \eqref{enest_appendix-ex-min} we obtain
\begin{equation}\label{enest_appendix-ex-min_withptt} 
\begin{aligned}
& \sum_{p\in\{\pbar,\ptil\}} \Bigl(
\mu \int_0^t\|p_{tt}\|_{L^2(\Omega)}^2\, ds
+\frac{1}{4}\|\nabla p_t(t)\|_{L^2(\Omega)}^2
+\frac{c^2}{8}\|\Delta p(t) \|_{L^2(\Omega)}^2
+\frac{b}{8} \int_0^t\|\Delta p_t\|_{L^2(\Omega)}^2\, ds\\
&\phantom{\sum_{p\in\{\pbar,\ptil\}} \Bigl(}
+\frac{\betaabs}{2} \int_0^t\|p_{tt}\|_{L^2(\Gamma_a)}^2\, ds
+\frac{\gammaabs}{4} \|p_{t}(t)\|_{L^2(\Gamma_a)}^2\Bigr)\\
%\mu \int_0^t\|\ptil_{tt}\|_{L^2(\Omega)}^2\, ds
%+\frac{1}{4}\|\nabla\ptil_t(t)\|_{L^2(\Omega)}^2
%+\frac{c^2}{8}\|\Delta\ptil(t) \|_{L^2(\Omega)}^2
%+\frac{b}{8} \int_0^t\|\Delta\ptil_t\|_{L^2(\Omega)}^2\, ds
%+\frac{\betaabs}{2} \int_0^t\|\ptil_{tt}\|_{L^2(\Gamma_a)}^2\, ds
%+\frac{\gammaabs}{4} \|\ptil_{t}(t)\|_{L^2(\Gamma_a)}^2\\
%&+\mu \int_0^t\|\pbar_{tt}\|_{L^2(\Omega)}^2\, ds
%+ \frac{1}{4} \|\nabla\pbar_{t}(t) \|_{L^2(\Omega)}^2
%+\frac{c^2}{8} \|\Delta\pbar(t) \|_{L^2(\Omega)}^2 
%+ \frac{b}{8}   \int_{0}^{t} \|\Delta\pbar_t\|_{L^2(\Omega)}^2 \, ds
%+\frac{\betaabs}{2} \int_0^t\|\pbar_{tt}\|_{L^2(\Gamma_a)}^2\, ds
%+\frac{\gammaabs}{4} \|\pbar_{t}(t) \|_{L^2(\Gamma_a)}^2\\
&\phantom{\sum_{p\in\{\pbar,\ptil\}}}+\frac{\rho}{\kappa}\Bigl(
\frac{\rho}{2}\|\wttil_t(t) \|_{L^2(\Gampl)}^2
+\frac{\delta}{2}\|\Delta_{pl}\wttil(t) \|_{L^2(\Gampl)}^2
\Damppl{+\beta \int_0^t\|(-\Delpl)^{\gamma/2} \wttil_{t}\|_{L^2(\Gampl)}^2\, ds}
\Bigr) 
\\
&\leq 
\frac{1}{2} \|\nabla p_1\|_{L^2(\Omega)}^2
+(\frac{c^2}{2}+\mu 4c^4T)\|\Delta p_{0} \|_{L^2(\Omega)}^2
+\frac{\gammaabs}{2} \| p_{1}\|_{L^2(\Gamma_a)}^2
\\&\quad
+\frac{\rho}{\kappa}\Bigl(
\frac{\rho}{2}\| w_2\|_{L^2(\Gampl)}^2
+\frac{\delta}{2} \|\Delta_{pl}w_{1} \|_{L^2(\Gampl)}^2 \Bigr) 
+ (\frac{1}{b}+2\mu)\,k^2 \|((\pbar+\ptil)^2)_{tt}\|_{L^2(L^2(\Omega))}^2
\\& \quad
+ \left |\int_{0}^{t}\int_{\Gamma_{N}} (\ptil_{tt}+\pbar_{tt}) g_{t} \, dS \, ds \right|
+\left | \int_0^t\int_{\Gampl}\tfrac{\rho}{\kappa}\hfcn_t \, \wttil_t\, dS \,ds\right|
%+ 2\left |\int_{0}^{t}\int_{\Gamma_{N}} \pbar_{tt} g_{t} \, dS \, ds \right|
\end{aligned}
\end{equation}
The left hand side of \eqref{enest_appendix-ex-min_withptt} 
%(up to lower order terms that are needed to obtain the full $H^1$ norm of $\pbar_t(t)$ and $\ptil_t(t)$) 
induces the energy defined by \eqref{def_energy}.
Note that in $\mathcal{E}$, the zero order in space term in the full $H^{1}(\Omega)$ norm of $\pbar_t(t)$ and $\ptil_t(t)$ is obtained from the estimate
\[
\|\ptil_t(t)\|_{L^{2}(\Omega)}^2
=\|\ptil_t(0)+\int_0^t\ptil_{tt}(s)\, ds\|_{L^{2}(\Omega)}^2
\leq 2\|\ptil_t(0)\|_{L^{2}(\Omega)}^2 +2T\int_0^t\|\ptil_{tt}(s)\|_{L^2(\Omega)}^2\, ds.
\]
%\ntBK{Re-insert the lower order estimate from the 2025-07-10 version here, because Gronwall does not work with the squared sup on the right hand side! Maybe with barrier's method, we can do}

We continue by estimating the right hand side terms in \eqref{enest_appendix-ex-min_withptt} and begin with the boundary terms. 
Since we do not have traces of $\pbar_{tt}$ and $\ptil_{tt}$ on $\Gamma_N$, we use the integration by parts identity 
\[
\begin{aligned}
\int_{0}^{t}\int_{\Gamma_{N}} (\ptil_{tt}+\pbar_{tt}) g_{t} \, dS \, ds
=-\int_{0}^{t}\int_{\Gamma_{N}} (\ptil_{t}+\pbar_{t}) g_{tt} \, dS \, ds
+\left[\int_{\Gamma_{N}} (\ptil_{t}+2\pbar_{t}) g_{t} \, dS\right|_0^t
\end{aligned}
\]
and estimate as 
\begin{equation}
\begin{aligned}\label{g}
 \left|\int_{0}^{t}\int_{\Gamma_{N}} (\ptil_{tt}+\pbar_{tt}) g_{t} \, dS \, ds \right|
& \leq C_{tr} \int_{0}^{t} (\Vert \ptil_{t} \Vert_{H^{1}}^{2}  
+ \Vert \pbar_{t} \Vert_{H^{1}}^{2})\, ds 
+  \int_{0}^{t} \Vert g_{tt} \Vert_{
{L^2} %H^{-1/2}
(\Gamma_{N})}^{2} \, ds
\\ & \indeq \indeq
+ \epsilon C_{tr} \Vert \ptil_{t}(t) \Vert_{H^{1}}^{2}
+\epsilon C_{tr} \Vert \pbar_{t}(t) \Vert_{H^{1}}^{2}
+ \frac{1}{\epsilon} \Vert g_{t} \Vert_{L^{\infty}(\zeroT
{L^2} %H^{-1/2}
(\Gamma_{N}))}^{2}
\\ & \indeq \indeq
+ C_{tr} \Vert p_{1} \Vert^{2}_{H^{1}} 
+ \Vert g_{t}(0) \Vert^{2}_{L^2(\Gamma_{N})}.
\end{aligned}
\end{equation}
where $C_{tr} \equiv \|\text{tr}\|^{2}_{H^1(\Omega_n)\to L^2(\Gamma_{N,n})}$.
{While using $\|\text{tr}\|^{2}_{H^1(\Omega_n)\to H^{1/2}(\Gamma_{N,n})}$ would allow for lower regularity in $\gfcn$, it would require higher regularity in $\lfcn$ to show uniformity of the embedding constant with respect to domain variations (that is, its independence on $n$, see below); for this reason we stay with the coarser estimate. 
Due to the compatibility condition \eqref{compat}, we have 
$\Vert g_{t}(0) \Vert_{L^2(\Gamma_{N})}=\Vert \Dnu{\check{p}_1} \Vert_{L^2(\Gamma_{N})}\leq C \mathcal{E}[\pbar,\ptil,\wttil](0)$.
}

Similarly, the last term on the right hand side of \eqref{enest_appendix-ex-min}
can be estimated as
\[
\begin{aligned}
\left| \int_0^t\int_{\Gampl} \hfcn_t \, \wttil_t\, dS \,ds\right| 
\leq \frac12\Vert \hfcn_{t} \Vert^{2}_{L^{2}(0,t; L^2(\Gamma_{pl}))}
+ \frac12\int_{0}^{t} \Vert \wttil_{t} \Vert^{2}_{L^{2}(\Gamma_{pl})}\, ds.
\end{aligned}
\]

It is important to note that the nonlinear term on the right hand side in \eqref{enest_appendix-ex-min_withptt} appears with a higher power than the corresponding term on the left hand side.
It can be estimated by 
\[
\|((\pbar+\ptil)^2)_{tt}\|_{L^2(L^2(\Omega))}^2
\leq 2\|\pbar+\ptil\|_{L^\infty(L^\infty(\Omega))}^2 \|\pbar_{tt}+\ptil_{tt}\|_{L^2(L^2(\Omega))}^2
+2\|\pbar_t+\ptil_t\|_{L^4(L^4(\Omega))}^4
\]
By embedding and interpolation, as well as Stampacchia's method \cite[Lemma 4.1]{KaTu1} we have
\begin{equation}\label{Stampacchia}
\Vert \pbar(t)\Vert_{L^\infty(\Omega)}\leq K(\mathfrak{s},\Omega) (\Vert \Delta \pbar(t)\Vert_{L^2(\Omega)} + \Vert \pbar(t)\Vert_{L^2(\Omega)} +\Vert \gfcn(t)\Vert_{L^{\mathfrak{s}}(\Gamma_N)}+ \Vert a[\pbar,\pbar_t]\Vert_{L^{\mathfrak{s}}(\Gamma_a)})
\end{equation}
and similarly
\begin{equation}\label{Stampacchia2}
\begin{aligned}
&\Vert \ptil(t)\Vert_{L^\infty(\Omega)}
\leq K(\mathfrak{s},\Omega) (\Vert \Delta \ptil(t)\Vert_{L^2(\Omega)} 
+ \Vert \ptil(t)\Vert_{L^2(\Omega)} 
+ \Vert \wttil(t)\Vert_{L^2(\Gampl)}
+ \Vert a[\ptil,\ptil_t]\Vert_{L^{\mathfrak{s}}(\Gamma_a)});
\\
&\text{with }\mathfrak{r}=2\geq d/2, \quad \mathfrak{s}>d-1
\end{aligned}\end{equation}
Here the boundary term stemming from the absorbing boundary conditions can be bounded by trace estimates 
\[
\begin{aligned}
&\Vert a[\pbar,\pbar_t]\Vert_{L^{\mathfrak{s}}(\Gamma_a)}
\leq
\betaabs\Vert \pbar_t \Vert_{L^{\mathfrak{s}}(\Gamma_a)}
+\gammaabs\Vert \pbar \Vert_{L^{\mathfrak{s}}(\Gamma_a)}\\
&\leq \|\text{tr}\|_{H^1(\Omega)\to L^2(\Gamma_a)} (\betaabs \|\pbar_t\|_{H^1(\Omega)} +\gammaabs \|\pbar\|_{H^1(\Omega)})
\end{aligned}\]
Moreover, by Sobolev's embedding 
\begin{equation}
\label{pbartt}
\begin{aligned}
\Vert \pbar_{t} +\ptil_{t} \Vert^{4}_{L^4(0,t;L^4(\Omega))}
&\leq C_{H^{1},L^4}^{\Omega}  \int_{0}^{t }(\Vert \pbar_{t} \Vert_{H^{1}}^{4}+\Vert \ptil_{t} \Vert_{H^{1}}^{4} )\, ds. 
\end{aligned}
\end{equation}
This altogether yields
\[
\|((\pbar+\ptil)^2)_{tt}\|_{L^2(L^2(\Omega))}^2
\leq C\Bigl(\int_{0}^{t} (\mathcal{E}[\pbar,\ptil,\wttil](s))^2\, ds
+ \sup_{s\in [0,t]}(\mathcal{E}[\pbar,\ptil,\wttil](s))^2 \Bigr).
\]

With these estimates, the relation \eqref{enest_appendix-ex-min_withptt}, provides us with an energy estimate of the form 
\begin{align*}
\mathcal{E}[\pbar,\ptil,\wttil](t) 
\leq \overline{C}(T)\Bigl(\mathcal{E}[\pbar,\ptil,\wttil](0)+ 
\int_0^t(\mathcal{E}[\pbar,\ptil,\wttil](s))^2\, ds
\\
\indeq 
+ \sup_{s\in [0,t]} (\mathcal{E}[\pbar,\ptil,\wttil](s))^2
+\|\text{data}\|^2
\Bigr)
\end{align*}
where $\mathcal{E}$ is defined as in \eqref{def_energy}
and $\|\text{data}\|^2$ as in \eqref{data}, and $\overline{C}(T)$ can be chosen as $\bar{C}(1+T)$ with $\bar{C}$ independent of $T$.
{We now apply an obstacle / barrier argument under the assumption of small enough 
 data $\overline{C}(T)(\mathcal{E}[\pbar,\ptil,\wttil](0))+\|\text{data}\|^2)<m_0<\frac{1}{8\overline{C}(T)(T+1)}$ and choosing $\bar{m}>0$ so that 
$\frac{m_0}{1-(T+1)\,\bar{m}} = \frac{\bar{m}}{2\overline{C}(T)}$; this can be achieved by setting
$\bar{m}=\frac{1+\sqrt{1-8(T+1)\overline{C}(T)m_0}}{(T+1)}$.
This allows us to conclude that for all $t\in[0,T]$ we have $\overline{C}(T)\mathcal{E}[\pbar,\ptil,\wttil](t))<\bar{m}$ by means of a contradiction argument:
If, on the contrary there exists $t_0>0$ such that $\overline{C}(T) \mathcal{E}[\pbar,\ptil,\wttil](t_0)>\bar{m}$ and $t_0$ is the smallest time where this happens; then for all $t\in(0,t_0)$ we have
\begin{align*}
\mathcal{E}[\pbar,\ptil,\wttil](t) 
\leq m_0 +\overline{C}(T)\Bigl(\int_0^t(\mathcal{E}[\pbar,\ptil,\wttil](s))^2\, ds
+ \sup_{s\in [0,t]} (\mathcal{E}[\pbar,\ptil,\wttil](s))^2\Bigr)\\
\leq m_0+ (T+1)\,\bar{m}\,\sup_{s\in [0,t]} (\mathcal{E}[\pbar,\ptil,\wttil](s))
\end{align*}
Letting $t\nearrow t_0$ and taking the sup over $t\in[0,t_0]$ on the left hand side implies 
\begin{align*}
\sup_{t\in [0,t_0]}\mathcal{E}[\pbar,\ptil,\wttil](t) 
\leq m_0+ (T+1)\,\bar{m}\,\sup_{s\in [0,t_0]} (\mathcal{E}[\pbar,\ptil,\wttil](s))
\end{align*}
that is, after rearranging 
\begin{align*}
\sup_{t\in [0,t_0]}\mathcal{E}[\pbar,\ptil,\wttil](t) 
\leq \frac{m_0}{1-(T+1)\,\bar{m}} = \frac{\bar{m}}{2\overline{C}(T)}
\end{align*}
hence $\overline{C}(T)\mathcal{E}[\pbar,\ptil,\wttil](t_0)\leq\frac{\bar{m}}{2}$, a contradiction.
}

%From this we obtain, for sufficiently small data by Gr\"onwall's inequality and for small enough $T$, 
Thus we have obtained, for sufficiently small data, an energy estimate of the form
\[
\sup_{s\in[0,T]}\mathcal{E}[\pbar,\ptil,\wttil](t)
\leq \tilde{C}(T)\Bigl(\mathcal{E}[\pbar,\ptil,\wttil](0) 
+\|\text{data}\|^2\Bigr).
\]

In order to derive uniform in $n$ estimates from this, we need to uniformly bound the $\Omega$ dependent constants 
\[
K(\mathfrak{s},\Omega_n), \quad 
C_{H^{d/4},L^4}^{\Omega_n}, \quad 
\|\text{tr}\|_{H^1(\Omega_n)\to L^2(\Gamma_{N,n})}.
\]
by making use of the fact that $\|\tfrac{1}{\lfcn_n}\|_{L^\infty(\Gamflat)}\leq 2\|\tfrac{1}{\lref }\|$ due to closeness  $\|\lfcn_n-\lref \|_{L^\infty(\Gamflat)}\leq\tfrac12\|\tfrac{1}{\lref }\|_{L^\infty(\Gamflat)}$, cf. \eqref{ass_l0}, \eqref{ass_diffl}, \eqref{bdd_lfcn} 
(note that $K(\mathfrak{s},\Omega_n)$ just relies on the embedding constant $C_{H^{1},L^p}^{\Omega_n}$).
For the norm of the trace operator see, e.g. \cite[Theorems 15.8, 15.23]{Leoni:2009} or \cite[Theorem 1.5.1.10]{Grisvard}. 

To obtain uniform bounds on these constants we use 
\[
\begin{aligned}
\|\phi\|_{H^1(\Omega(\lfcn))}^2
&=\int_{\Omega(\lfcn)}|\phi|^2\, dx + \int_{\Omega(\lfcn)}|\nabla\phi|^2\, dx
%\\&
=\int_{\Omega(\lref)}\frac{1}{\omegazero_\lfcn}|\check{\phi}|^2\, d\check{x} + \int_{\Omega(\lref)}\frac{1}{\omegazero_\lfcn}|M_\lfcn  \check{\nabla} \check{\phi}|^2\, d\check{x}\\
&\geq \frac{1}{\|\omegazero_\lfcn\|_{L^\infty}\,\min\{1,\|M_\lfcn^{-1}\|_{L^\infty}^2\}}
\|\check{\phi}\|_{H^1(\Omega(\lref))}^2
\end{aligned}
\]
\[
\begin{aligned}
\|\phi\|_{L^p(\Omega(\lfcn))}
&=\Bigl(\int_{\Omega(\lfcn)}|\phi|^p\, dx\Bigr)^{1/p}
=\Bigl(\int_{\Omega(\lref)}\frac{1}{\omegazero_\lfcn}|\check{\phi}|^p\, d\check{x}\Bigr)^{1/p}
\leq \|\tfrac{1}{\omegazero_\lfcn}\|_{L^\infty}^{1/p} \|\check{\phi}\|_{L^p(\Omega(\lref))}
\end{aligned}
\]
\[
\begin{aligned}
\|\phi\|_{L^2(\Gamma_N(\lfcn))}
&=\Bigl(\int_{\Gamma_N(\lfcn)}|\phi|^2\, dx\Bigr)^{1/2}
=\Bigl(\int_{\Gamma_N(\lref)}\omegaone_\lfcn|\check{\phi}|^2\, dS(\check{x})\Bigr)^{1/2}
&\leq \|\omegaone_\lfcn\|_{L^\infty}^{1/2} \|\check{\phi}\|_{L^2(\Gamma_N(\lref))},
\end{aligned}
\]
This implies
\[
C_{H^1,L^p}^{\Omega(\lfcn)}\leq C_{H^1,L^p}^{\Omega(\lref)}
\|\tfrac{1}{\omegazero_\lfcn}\|_{L^\infty}^{1/p} 
\|\omegazero_\lfcn\|_{L^\infty}^{1/2} \Bigl(1+\|M_\lfcn^{-1}\|_{L^\infty}^2\Bigr)^{1/2}
\]
and 
\[
\|\text{tr}\|_{H^1(\Omega(\lfcn))\to L^2(\Gamma_{N}(\lfcn))}
\leq \|\text{tr}\|_{H^1(\Omega(\lref))\to L^2(\Gamma_{N}(\lref))}
\|\omegaone_\lfcn\|_{L^\infty(\Gamflat)}^{1/2} 
\|\omegazero_\lfcn\|_{L^\infty(\Gamflat)}^{1/2} \Bigl(1+\|M_\lfcn^{-1}\|_{L^\infty(\Gamflat)}^2\Bigr)^{1/2},
\]
where the norms on the right hand side can be bounded by means of a constant multiple of $\|\lfcn\|_{W^{1,\infty}(\Gamflat)}$, 
\begin{equation}\label{l}
\|\omegazero_\lfcn\|_{L^\infty(\Gamflat)}
+\|\tfrac{1}{\omegazero_\lfcn}\|_{L^\infty(\Gamflat)}
+\|M_\lfcn^{-1}\|_{L^\infty(\Gamflat)}
+ \|\omegaone_\lfcn\|_{L^\infty(\Gamflat)}
\leq C \|\lfcn\|_{W^{1,\infty}(\Gamflat)},
\end{equation}
due to \eqref{bdd_lfcn}. 

The estimate 
$$\|\Dnu{\pbar_{n}}\|_{L^2(\zeroT H^s(\Gamma_N))}^2\leq \|g_n\|_{L^2(\zeroT H^s(\Gamma_N))}^2$$
is a trivial consequence of the identity $\Dnu{\pbar_{n}}=g_n$ that follows from $A_{\text{PDE}}(\vec{g}_n,\vec{u}_n)=0$.

We have thus proven Lemma~\ref{lem:estAPDE0}.

%% file: appendix_surjectivity.tex
\section{Energy estimates for surjectivity; 
Proof of Lemma~\ref{lem:enest_surj}}
%\begin{proof}
We here focus on the energy estimate as a key step in the proof. This together with the usual approach of making a Galerkin discretization in space and taking weak limits will allow us to prove existence of a solution to \eqref{eqn:var_L2L2_pponly} in $U_{pp}$. 
(The uniqueness step in the proof can be skipped this time, since we only need surjectivity of $A_{\textup{PDE}}'(\vec{g}^*,\vec{u}^*)$).  

Testing \eqref{eqn:var_L2L2_pponly} with 
$\qbar= \uld{\pbar}_{tt}-\bar{\lambda}\Delta \uld{\pbar}_{t}$, 
$\qtil= \uld{\ptil}_{tt}-\tilde{\lambda}\Delta \uld{\ptil}_{t}$, 
%$\mu_N=(-\Delta_{\Gamma_N})^s[(\bar{\lambda}+b)\partial_t+c^2]\Dnu{\uld{\pbar}}-\uld{\pbar}_{tt}$, 
$\mu_N=[(\bar{\lambda}+b)\partial_t+c^2]^*\Bigl\{
(-\Delta_{\Gamma_N})^s[(\bar{\lambda}+b)\partial_t+c^2]\Dnu{\uld{\pbar}}+\uld{\pbar}_{tt}\Bigr\}$, 
%$\mu_{pl}=(-\Delpl)^s[(\tilde{\lambda}(1+\mathfrak{a})+b)\partial_t+c^2]\Dnu{\uld{\ptil}}-\uld{\ptil}_{tt}$, 
$\mu_{pl}=[(\tilde{\lambda}(1+\mathfrak{a})+b)\partial_t+c^2]^*\Bigl\{
(-\Delpl)^s[(\tilde{\lambda}(1+\mathfrak{a})+b)\partial_t+c^2]\Dnu{\uld{\ptil}}+\uld{\ptil}_{tt}\Bigr\}$
(with $s\in(0,\frac12)$ and the superscript $*$ denoting the adjoint with respect to $L^2(\zeroT L^2(\Gamma_N))$ or $L^2(\zeroT L^2(\Gampl))$)
we obtain the energy identity
\begin{equation}\label{enid_pp}    
\begin{aligned}
&\tfrac12\|\sqrt{\bar{\lambda}+b}\,\nabla\uld{\pbar}_{t}(t)\|_{L^2(\Omega)}^2
+\tfrac{\bar{\lambda}c^2}{2}\|\Delta\uld{\pbar}(t)\|_{L^2(\Omega)}^2
%+\tfrac{c}{2}\|\uld{\pbar}_t(t)\|_{L^2(\Gamma_a)}^2
+\tfrac{1}{2}\|\sqrt{c^2\betaabs+(\bar{\lambda}+b)\gammaabs}\uld{\pbar}_t(t)\|_{L^2(\Gamma_a)}^2
\\
&+\int_0^t\Bigl(
\|\uld{\pbar}_{tt}\|_{L^2(\Omega)}^2
-c^2\|\nabla\uld{\pbar}_{t}\|_{L^2(\Omega)}^2
+\bar{\lambda}b \|\Delta\uld{\pbar}_{t}\|_{L^2(\Omega)}^2
\\&\qquad
%+\tfrac{1}{c}\|\sqrt{\bar{\lambda}+b}\,\uld{\pbar}_{tt}\|_{L^2(\Gamma_a)}^2
+c^2\|\sqrt{\gammaabs}\,\uld{\pbar}_t\|_{L^2(\Gamma_a)}^2
+\|\sqrt{(\bar{\lambda}+b)\betaabs}\,\uld{\pbar}_{tt}\|_{L^2(\Gamma_a)}^2
+\|[(\bar{\lambda}+b)\partial_t+c^2]\,\Dnu\uld{\pbar}\|_{H^s(\Gamma_N)}^2
\Bigr)\, ds\\
&+\tfrac12\|\sqrt{\tilde{\lambda}(1+\mathfrak{a})+b}\,\nabla\uld{\ptil}_{t}(t)\|_{L^2(\Omega)}^2
+\tfrac{\tilde{\lambda}c^2}{2}\|\Delta\uld{\ptil}(t)\|_{L^2(\Omega)}^2
%+\tfrac{c}{2}\|\uld{\ptil}_t(t)\|_{L^2(\Gamma_a)}^2
+\tfrac{1}{2}\|\sqrt{c^2\betaabs+(\tilde{\lambda}(1+\mathfrak{a})+b)\gammaabs}\uld{\ptil}_t(t)\|_{L^2(\Gamma_a)}^2
\\
&+\int_0^t\Bigl(
\|\sqrt{1+\mathfrak{a}}\,\uld{\ptil}_{tt}\|_{L^2(\Omega)}^2
-c^2\|\nabla\uld{\ptil}_{t}\|_{L^2(\Omega)}^2
+\tilde{\lambda}b \|\Delta\uld{\ptil}_{t}\|_{L^2(\Omega)}^2
\\&\qquad
%+\tfrac{1}{c}\|\sqrt{\tilde{\lambda}(1+\mathfrak{a})+b}\,\uld{\ptil}_{tt}\|_{L^2(\Gamma_a)}^2
+c^2\|\sqrt{\gammaabs}\,\uld{\ptil}_t\|_{L^2(\Gamma_a)}^2
+\|\sqrt{(\tilde{\lambda}(1+\mathfrak{a})+b)\betaabs}\,\uld{\ptil}_{tt}\|_{L^2(\Gamma_a)}^2
+\|[(\tilde{\lambda}(1+\mathfrak{a})+b)\partial_t+c^2]\,\Dnu\uld{\ptil}\|_{H^s(\Gampl)}^2
\Bigr)\, ds\\
&= \int_0^t\int_\Omega \Bigl( 
f_{\pbar}\,(\uld{\pbar}_{tt}-\bar{\lambda}\Delta \uld{\pbar}_{t}) + \tilde{f}_{\ptil}(\uld{\pbar},\uld{\ptil})\,(\uld{\ptil}_{tt}-\tilde{\lambda}\Delta \uld{\ptil}_{t})
+\tfrac{\tilde{\lambda}}{2}\mathfrak{a}_t|\nabla\uld{\ptil}_t|^2\Bigr)\, dx\, ds \\
&\quad-c^2 \int_\Omega \bigl(
\nabla\uld{\pbar}(t)\cdot\nabla \uld{\pbar}_t(t)
+\nabla\uld{\ptil}(t)\cdot\nabla\uld{\ptil}_t(t)\bigr)\, dx,
    \end{aligned}
\end{equation}
with $\tilde{f}_{\ptil}(\uld{\pbar},\uld{\ptil})={f}_{\ptil}
-\mathfrak{b} (\uld{\pbar}+\uld{\ptil})_t 
-\mathfrak{c}(\uld{\pbar}+\uld{\ptil})
-\mathfrak{a}\uld{\pbar}_{tt}$.

Note that we get basically the same energy identity for $\uld{\pbar}$ and $\uld{\ptil}$, just with the roles of $\Gamma_N$ and $\Gampl$ interchanged and with some additional space and time dependent coefficients in case of $\uld{\ptil}$.

We also point to the fact that the left hand side term 
$\int_0^t \|[(\bar{\lambda}+b)\partial_t+c^2]\,\Dnu{\uld{\pbar}}\|_{H^s(\Gamma_N)}^2\, ds$ provides us with a bound on both 
$\|\Dnu\uld{\pbar}\|_{L^2(\zeroT H^s(\Gamma_N))}^2$ and
$\|\Dnu\uld{\pbar}_t\|_{L^2(\zeroT H^s(\Gamma_N))}^2$;
likewise for 
$\|\Dnu\uld{\ptil}\|_{L^2(\zeroT H^s(\Gampl))}^2$ and
$\|\Dnu\uld{\ptil}_t\|_{L^2(\zeroT H^s(\Gampl))}^2$ from the left hand side term
$\int_0^t \|[(\tilde{\lambda}(1+\mathfrak{a})+b)\partial_t+c^2]\,\Dnu\uld{\ptil}\|_{H^s(\Gampl)}^2\, ds$.

For sufficiently large $\bar{\lambda}$, $\tilde{\lambda}\, >0$ together with Sobolev's embeddings and Young's as well as Gronwall's inequalities this yields the energy estimate
\begin{equation}\label{enest_pp}
    \begin{aligned}
&\sum_{\uld{p}\in\{\uld{\pbar},\uld{\ptil}\}}
\|\uld{p}_{tt}\|_{L^2(0,t;L^2(\Omega))}^2
+\|\Delta\uld{p}_{t}\|_{L^2(0,t;L^2(\Omega))}^2
+\|\nabla\uld{p}_{t}\|_{L^\infty(0,t;L^2(\Omega))}^2
+\|\Delta\uld{p}\|_{L^\infty(0,t;L^2(\Omega))}^2\\
&\phantom{\sum_{\uld{p}\in\{\uld{\pbar},\uld{\ptil}\}}}
+\|\Dnu\uld{p}_t\|_{L^2(0,t;H^s(\partial\Omega))}^2
+\|\Dnu\uld{p}\|_{L^2(0,t;H^s(\partial\Omega))}^2
\leq C(T) \sum_{{p}\in\{{\pbar},{\ptil}\}}
\|f_{p}\|_{L^2(0,t;L^2(\Omega))}^2,
    \end{aligned}
\end{equation}
under the smallness assumptions \eqref{smallness_coeffs}, which also implies nondegeneracy $\frac{1}{1+\mathfrak{a}}\, \in L^\infty(\zeroT L^\infty(\Omega))$.
%\ntregl{ For obtaining the $ L^\infty(\Omega)$ bound we might use Stampacchia's method rather than enforcing the full $H^{3/2+s}(\Omega)$ norm by elliptic regularity; PROBABLY NOT NEEDED AT ALL, SINCE LEMMA ~\ref{lem:enest_surj} DOES NOT CLAIM ANYTHING ABOUT $\|\uld{p}\|_{L^\infty(0,t;L^\infty(\Omega))}$}
%\end{proof}